\documentclass{amsart}

\usepackage{mathtools}
\usepackage{latexsym,amssymb,amsmath, amsfonts}
\usepackage{amsthm}
\usepackage{amscd}
\usepackage{bbm}
\usepackage{mathpazo}

\numberwithin{equation}{section}
\usepackage{enumitem,hyperref}
\hypersetup{
    colorlinks = true,
    linkcolor=blue,
}
\newcommand{\R}{\mathbb{R}}
\newcommand{\C}{\mathbb{C}}

\newcommand{\be}{\begin{equation}}
\newcommand{\en}{\end{equation}}
\newcommand{\ee}{\end{equation}}

\DeclareMathOperator{\dist}{dist}
\DeclareMathOperator{\supp}{supp}

\newcommand{\bt}{\begin{theorem}}
\newcommand{\et}{\end{theorem}}
\newcommand{\bp}{\begin{proof}}
\newcommand{\ep}{\end{proof}}
\newcommand{\bc}{\begin{cor}}
\newcommand{\ec}{\end{cor}}
\newcommand{\bl}{\begin{lemma}}
\newcommand{\el}{\end{lemma}}
\newcommand{\bprop}{\begin{prop}}
\newcommand{\eprop}{\end{prop}}

\newtheorem{theolett}{Theorem}

\newtheorem{theorem}{Theorem}[section]
\newtheorem{defi}[theorem]{Definition}
\newtheorem{remark}[theorem]{Remark}
\newtheorem{lemma}[theorem]{Lemma}

\newtheorem{prop}[theorem]{Proposition}

\newtheorem{cor}[theorem]{Corollary}
\newtheorem{claim}[theorem]{Claim}
\newtheorem{example}[theorem]{Example}
\usepackage{lipsum}

\usepackage{tikz} 
\usepackage{pgfplots}
\usetikzlibrary{intersections, pgfplots.fillbetween}
\usepackage{tikz-3dplot}
\pgfplotsset{compat=1.7}
\usetikzlibrary{patterns}

\theoremstyle{definition}

\def\namedlabel#1#2{\begingroup
    #2%
    \def\@currentlabel{#2}%
    \phantomsection\label{#1}\endgroup
}

\author{A. J. Mendez}
\address{Universidade Estadual de Campinas, Rua Sérgio Buarque de Holanda, 651, cep 13083-859, Campinas, S\~{a}o Paulo, Brasil} \email{ajmendez@ime.unicamp.br}

\author[O. Ria\~no]{Oscar Ria\~no}
\address{Departamento de Matem\'aticas, Universidad Nacional de Colombia, Carrera 45 No. 26-85, Edificio Uriel Guti\'errez, Bogot\'a D.C., Colombia} \email{ogrianoc@unal.edu.co}

\thanks{}

\date{\today}
\title[On decay  properties for solutions of the ZK equation]{ On decay  properties for solutions of the Zakharov-Kuznetsov equation}
\keywords{}

\begin{document}
\begin{abstract}
This work mainly focuses on spatial decay properties of solutions to the Zakharov-Kuznetsov equation. For the two- and three-dimensional cases, it was established that if the initial condition $u_0$ verifies $\langle \sigma\cdot x\rangle^{r}u_{0}\in L^{2}(\left\{\sigma\cdot x\geq \kappa\right\}),$ for some $r\in\mathbb{N}$, $\kappa \in\mathbb{R}$, being $\sigma$ be a suitable non-null vector in the Euclidean space,  then the corresponding solution $u(t)$ generated from this initial condition verifies $\langle \sigma\cdot x\rangle ^{r}u(t)\in L^2\left(\left\{\sigma\cdot x>\kappa-\nu t\right\}\right)$, for any  $\nu >0$. Additionally,  depending on the magnitude of the weight $r$, it was also deduced some localized gain of regularity. In this regard, we first extend such results to arbitrary dimensions, decay power $r>0$ not necessarily an integer, and we give a detailed description of the gain of regularity propagated by solutions. The deduction of our results depends on a new class of pseudo-differential operators, which is useful for quantifying decay and smoothness properties on a fractional scale. Secondly, we show that if the initial data $u_{0}$ has a decay of exponential type on a particular half space, that is, $e^{b\, \sigma\cdot x}u_{0}\in L^{2}(\left\{\sigma\cdot x\geq \kappa\right\}),$ then the corresponding solution satisfies $e^{b\, \sigma\cdot x} u(t)\in H^{p}\left(\left\{\sigma\cdot x>\kappa-t\right\}\right),$ for all $p\in\mathbb{N}$,  and time $t\geq \delta,$ where  $\delta>0$. To our knowledge, this is the first study of such property. As a further consequence, we also obtain well-posedness results in anisotropic weighted Sobolev spaces in arbitrary dimensions. 

Finally, as a by-product of the techniques considered here, we show that our results are also valid for solutions of the Korteweg-de Vries equation.

\end{abstract}
\maketitle


\section{Introduction}\label{intro}

We consider solutions of the initial value problem (IVP) associated to the \emph{ Zakharov-Kuznetsov}-(ZK) equation:
\begin{equation}\label{ZK}
    \left\{\begin{aligned}
    &\partial_t u +
    \partial_{x_1}\Delta u+u\partial_{x_{1}}u =0,\quad (x_1,\dots,x_n)\in \R^{n}, \, t\in \mathbb{R}, \\
    &u(x,0)=u_0(x),
    \end{aligned}\right.
\end{equation}
where $n\geq 2$ denotes the spatial dimension and $\Delta=\partial_{x_{1}}^{2}+\partial_{x_{2}}^{2}+\cdots+\partial_{x_{n}}^{2}$ stands for the Laplace operator. The ZK equation was originally deduced in dimension $n=3$ in the context of plasma physics to model weakly nonlinear ion-acoustic waves in the presence of a uniform magnetic field. Nevertheless, it also has a physical meaning when $n=2;$ for a more detailed exposition on this subject we refer to  \cite{NaumkinShishmar1994, ShriraVoronoVyac1996, ZakharovKuznet1974,LannesLinaresSaut2013}. 
%

Referring to well-posedness  in $H^s(\mathbb{R}^n)$ and weighted spaces (a Cauchy problem is well-posed in a functional space if one can assure existence, uniqueness of solution, and continuous dependence of the data-to-solution flow map):

 In \cite{Kinoshita2021}, it was established that the two-dimensional ZK equation is locally well-posed (LWP) in $H^{s}(\mathbb{R}^2)$, $s>-\frac{1}{4}$, and globally well-posed (GWP) in $L^2(\mathbb{R}^2)$. For previous results on the well-posedness, see \cite{Faminski1996, LinaresPastor2009, MolinetPilod2015, GrunrockHerr2014}. It was also proved in \cite{Kinoshita2021} that the well-posedness theory in this case is sharp concerning the application of Banach fixed point theorem based on the integral formulation of \eqref{ZK}. 

Concerning the $n$-dimensional case with $n\geq 3$, in \cite{HerrKinoshita2023},  LWP was established in $H^{s}(\mathbb{R}^n)$, $s>\frac{n-4}{2}$, and thus the Cauchy problem \eqref{ZK} is GWP in $L^2(\mathbb{R}^3)$, and \eqref{ZK} is GWP in $H^1(\mathbb{R}^4)$ under smallness assumption of the $L^2(\mathbb{R}^4)$-norm (which appears as the ZK equation is $L^2$-critical for $n=4$).  For previous results, we also refer to \cite{LinaresSaut2009, MolinetPilod2015, RibaudVento2012}.

 The well-posedness has also been addressed in weighted spaces, such as $H^s(\mathbb{R}^2)\cap L^2((|x|^{2r_1}+|y|^{2r_1})\, \mathrm{d}x \mathrm{d}y)$, $r_1, r_2 \geq 0$, see \cite{BustamanteJimenezMejia2016,FonsecaPachon2016}. Also, in exponential spaces  $H^s(\mathbb{R}^n)\cap L^2(e^{b \cdot x}\, \mathrm{d}x)$, $n\geq 2$, LWP has been studied in \cite{DeleageLinares2023}. 

The ZK equation possesses \emph{soliton solutions} of the form $u(x_1,\dots,x_n,t)=Q_c(x_1-ct,\dots,x_n)$, $c>0$, where $Q_c=cQ(\sqrt{c} x_1,\dots,\sqrt{c}x_n)$, and $Q$ solves the elliptic equation
\begin{equation}\label{groundstateeq}
    \Delta Q-Q+\frac{1}{2}Q^2=0,
\end{equation}
and the dimension $n\leq 6$. It is well-known that there exists a unique positive radial symmetric solution $Q\in H^{1}(\mathbb{R}^n)\cap C^{\infty}(\mathbb{R}^n)$, such that there exists $\delta>0$ for which 
\begin{equation}\label{solitodecay}
|\partial^{\alpha}Q(x)|\lesssim e^{-\delta |x|}, \, \, \text{ for all }\, x\in \mathbb{R}^n,
\end{equation}
for all multi-index $\alpha$. In \cite{deBouard1996}, it was proved orbital stability and instability of solitons. The asymptotic stability for the two-dimensional case was established in \cite{CoteMunozPilodSimpson2016}, and for the three-dimensional case in \cite{FarahHolmerRoudYang2023}. For other studies on the ZK equation, such as the asymptotic behavior of solutions,  unique continuation properties, along with others, we refer to \cite{MendezMunozPobletePozo2021, Valet2021, CossetiFanelliLinares2019, BustamanteIsazaMejia2013} and the references therein. 

The initial valued problem \eqref{ZK} can be regarded as a generalization to higher dimensions of the widely studied initial value problem associated to the \emph{Korteweg-de Vries} equation (KdV)
\begin{equation}\label{KdV}
    \left\{\begin{aligned}
    &\partial_t u +
    \partial_{x}^3 u+u\partial_{x} u =0,\quad x\in \R, \, t\in \mathbb{R}, \\
    &u(x,0)=u_0(x).
    \end{aligned}\right.
\end{equation}
The KdV equation has been studied in several physical contexts such as shallow-water waves, long internal waves in a density-stratified ocean, and ion-acoustic waves in a plasma,  among many others, see \cite{JeffreyKakutani1972, KortewegVries1895, Miura1976} and reference therein.

This paper concerns the study of the propagation of regularity and decay of solutions of \eqref{ZK} when the initial condition enjoys extra localized decay (such as polynomial and exponential). This phenomenon is closely related to smoothing effects for dispersive equations. Such effects have been applied in different situations, among them, we remark on well-posedness theory, where they have been used to determine the minimal regularity required to deduce the existence/uniqueness of solutions. To motivate our results, we consider the initial value problem associated with the KdV equation \eqref{KdV}. Kato \cite{Kato1983} deduced that if $u\in C([0,T];H^{s}(\mathbb{R}))$ solves the KdV equation with an initial condition $u_0 \in H^s(\mathbb{R})$, $s>\frac{3}{2}$, then $\partial_x J^su\in L^{1}([0,T];L^2(-R,R))$ for any $R>0$. In other words, the solution to the initial value problem is, locally, one derivative smoother than the initial data. This result motivates the question: How is extra regularity of the initial data propagated by the solution flow of the equation?
Addressing such a question,  Isaza, Linares, and Ponce \cite{IsazaLinaresPonce2015} found what nowadays is known as  \emph{propagation of regularity phenomena} for solutions of the KdV equation. They proved that extra regularity in the initial data localized on the right-hand side of the real line travels to the left with an infinite speed. 
Since this pioneering work, the study of the propagation of regularity has been investigated for other dispersive equations: 
In dimension $n=1$, see \cite{BolingGuoquan2018, IsazaLinaresPonce2015,IsazaLinaresPonceBO2016, KenigLinaresVega2018,LinaresMiyaGus2019, LinarePOnceSmith2017,Argenis2020, MendezA2020,MunozPonceSaut2021, SegataSmith2017}, and in higher dimensions, $n\geq 2$, see \cite{RicardoArgenisOscar2022, IsazaLinaresPonceKP2016,LinaresPonce2018, Mendez2024,Nascimiento2019, Nascimiento2020,NascimentoA2021, Argenis2023}. For a more recent survey about the study of the propagation of the regularity principle, we refer to \cite{LinaresPonce2023}.

The increasing study of the propagation of regularity gives rise to other questions, including how localized polynomial or exponential behavior is propagated. Referring to polynomial weights, such questions have only been studied for the KdV equation and ZK equation with $n=2,3$, in the case of integer polynomial decay, see \cite{IsazaLinaresPonce2015,LinaresPonce2018}. In this sense, this article extends the results in \cite{IsazaLinaresPonce2015}  and \cite{LinaresPonce2018} to fractional weights, and it seems to be the first study of the propagation of localized fractional weights in higher-dimensional models. We emphasize that the presence of fractional weights offers more flexibility concerning applications and yields a complete study of the balance between decay and regularity.

Certainly, the initial data  $u_{0}$ of the IVP \eqref{ZK}  could have a stronger decay than the polynomial one, and the immediate question is to determine if such decay on a localized region is also propagated by solutions of ZK. Nevertheless, determining the class of weights that are preserved by the flow solution of the IVP \eqref{ZK} is by far a quite hard question to address. This situation is reflected, for example, in this work where it is necessary to define a new class of pseudo-differential operators to study such a problem in the case where the weights are of polynomial type. 
In this sense, inequality \eqref{solitodecay} is quite instructive to be our proposals, since it is natural to study the propagation of weights in the case where the initial data $u_{0}$ of the IVP \eqref{ZK} has a decay almost similar to the soliton solution on a particular half-space.  Such a problem has its roots in  Kato's work \cite{Kato1983}, where he deduced that the KdV equation has a parabolic behavior in $L^{2}$-weighted exponential spaces (see \cite{Kato1983} for more details).  More recently, the study of the persistence of solutions of the ZK equation in weighted exponential spaces was addressed in \cite{DeleageLinares2023}, constituting an extension of Kato's work to higher dimensions.

In this work, 
we establish that if for some $\beta\in \mathbb{R}, b>0,$  and some non-null vector $\sigma\in \mathbb{R}^{n},$ (which depend on the dispersion)  the initial data of the Cauchy problem \eqref{ZK} satisfies $e^{b \sigma\cdot x  }u_{0}\in L^{2}\left(\left\{\sigma \cdot x \geq \beta \right\}\right),$ then the corresponding solution of ZK  preserves this same decay on a moving region depending on $t,$ but unlike the polynomial case, the solution becomes smooth almost instantaneously. We shall point out that the exponential decay studied in this work is more general than that of the soliton solution $Q.$ 
To our knowledge, this work is the first one to address the propagation of exponential decay both in the one-dimensional case (KdV equation) and the multidimensional setting (the ZK equation).

The result already described leads us to conjecture that solutions could be analytic in the moving region where the smoothing effects occur, or in the worst case, it has to be in an intermediate class of functions between $C^{\infty}$ and analytical, e.g., some modified Gevrey class. Additionally, we believe that if some non-null initial data has the localized exponential decay considered in this work, this would imply that the corresponding solution $u(t)$ of ZK cannot have compact support in the region of propagation for any $t \in(0, T).$ Nevertheless, we do not address these questions in this work.


\subsection{Main results}

Before stating some results, let us introduce some preliminary notation. For a given non-null vector $\sigma$ in $\mathbb{R}^n$ and $\kappa \in \mathbb{R}$, we define the \emph{half-space} as the set determined by
\begin{equation*}
\mathcal{H}_{\{\sigma,\kappa\}}=\left\{x\in\mathbb{R}^n|\, \sigma \cdot x >\kappa\right\}.
\end{equation*}
Furthermore, for $\gamma, \kappa \in \mathbb{R}$ with $\gamma<\kappa$ we define the \emph{strip}
\begin{equation}\label{defQ}
Q_{\{\sigma,\gamma, \kappa\}}=\left\{x\in\mathbb{R}^n|\, \gamma<\sigma \cdot x<\kappa \right\}.
\end{equation}
As a further preliminary, we require the propagation of regularity results for solutions of the ZK equation. The following result was proven in \cite{Mendez2024} (for the integer regularity case $s\in \mathbb{Z}$ of the following theorem, see  \cite{LinaresPonce2018}). 
\begin{theolett}\label{TheorArg}
Let $u_0 \in H^{\big(\frac{n+2}{2}\big)^{+}}(\mathbb{R}^n)$. If for some $\sigma=(\sigma_1,\sigma_2,\dots,\sigma_n)\in \mathbb{R}^n$, $n\geq 2$ with
\begin{equation}
\sigma_1>0, \qquad \text{ and } \qquad \sqrt{3}\sigma_1> \sqrt{\sigma_2^2+\dots+\sigma_n^2}
\end{equation}
and for some $s\in \mathbb{R}$, $s>\frac{n+2}{2}$, and $\kappa \in \mathbb{R}$
\begin{equation}
\|J^s u_0\|_{L^2(\mathcal{H}_{\{\sigma,\kappa\}})}<\infty,
\end{equation}
then the corresponding solution $u=u(x,t)\in C([0,T];H^{\big(\frac{n+2}{2}\big)^{+}}(\mathbb{R}^n))$ of the IVP \eqref{ZK} satisfies: for any $\nu > 0$, $\epsilon>0$, 
\begin{equation}
\sup_{0\leq t \leq T} \int_{\mathcal{H}_{\{\sigma,\kappa+\epsilon-\nu t\}}} \big(J^r u(x,t) \big)^2 \mathrm{d}x \leq c^{\ast}
\end{equation}
for any $r\in (0,s]$ with $c^{\ast}=c^{\ast}\big(\epsilon,\sigma,T,\nu,\|u_0\|_{H^{\frac{n+2}{2}^{+}}},\|J^s u_0\|_{L^2(\mathcal{H}_{\{\sigma,\kappa\}})}\big)$. In addition, for any $\nu> 0$, $\epsilon>0$, and $\tau \geq 5\epsilon,$
\begin{equation}
\int_0^T \int_{Q_{\{\sigma, \epsilon-\nu t+\kappa,\tau-\nu t+\kappa\}}} \big(J^{s+1}u\big)^2 \mathrm{d}x \, \mathrm{d}t< c^{\ast},
\end{equation}
where $c^{\ast}=c^{\ast}\big(\epsilon,\sigma,\tau,T,\nu,\|u_0\|_{H^{\frac{n+2}{2}^{+}}},\|J^s u_0\|_{L^2(\mathcal{H}_{\{\sigma,\kappa\}})}\big)$.
\end{theolett}
Our first main result studies how the persistence property in weighted spaces is propagated by solutions of the ZK equation. Roughly, we prove that if the initial condition enjoys extra polynomial decay localized in precise regions of the space, then the solution of \eqref{ZK} generated from it propagates the same amount of polynomial decay, and it gains extra localized regularity determined by the magnitude of the polynomial decay of the initial condition. 
\begin{theorem}\label{mainTHM}
Let $n\geq 2$ be fixed, $u_0 \in H^{\big(\frac{n+2}{2}\big)^{+}}(\mathbb{R}^n)$, and $\sigma=(\sigma_1,\sigma_2,\dots,\sigma_n)\in \mathbb{R}^n$, with
\begin{equation}\label{sigmacond}
\sigma_1>0, \qquad \text{ and } \qquad \sqrt{3}\sigma_1> \sqrt{\sigma_2^2+\dots+\sigma_n^2}.
\end{equation}
If for some $r\in \mathbb{R}^{+}$ and $\kappa \in \mathbb{R}$ \footnote{For $x\in\mathbb{R}^{n}$  we  use the notation   $\langle x \rangle:=(1+|x|^2)^{\frac{1}{2}}$. Additionally, for $x\in \mathbb{R}$, $\lfloor x \rfloor$ denotes the \emph{greatest integer less than or equal to} $x$, and $\lceil x\rceil$ denotes the \emph{least integer greater than or equal to} $x$. }
\begin{equation}\label{inicond} 
\|\langle \sigma \cdot x \rangle^r u_0\|_{L^2(\mathcal{H}_{\{\sigma,\kappa\}})}<\infty,
\end{equation}
then the corresponding solution $u=u(x,t)\in C([0,T];H^{\big(\frac{n+2}{2}\big)^{+}}(\mathbb{R}^n))$ of the IVP \eqref{ZK} satisfies: For any $\nu > 0$, $\epsilon>0$, 
\begin{equation}\label{MainTHMeq1}
\sup_{0\leq t \leq T} \int_{\mathcal{H}_{\{\sigma,\kappa+\epsilon-\nu t\}}} \big(\langle \sigma \cdot x \rangle^r u(x,t) \big)^2\,  \mathrm{d}x \leq c^{\ast}
\end{equation}
with $c^{\ast}=c^{\ast}\big(\epsilon,\sigma,T,\nu,\|u_0\|_{H^{\frac{n+2}{2}^{+}}},\|\langle \sigma\cdot x \rangle^r u_0\|_{L^2(\mathcal{H}_{\{\sigma,\kappa\}})}\big)$. Furthermore, for any $\nu> 0$, $\delta>0$ $\epsilon>0$, and $\tau \geq 5\epsilon$,
\begin{equation}\label{MainTHMeq2.1}
\int_{\delta}^T \int_{Q_{\{\sigma, \epsilon-\nu t+\kappa,\tau-\nu t+\kappa\}}} \big(J^{\lfloor2r\rfloor+1}u\big)^2 \mathrm{d}x \, \mathrm{d}t< c^{\ast},
\end{equation}
and if in addition $2r\notin \mathbb{Z}^{+}$,
\begin{equation}\label{MainTHMeq2}
\int_{\delta}^{T}\int_{\mathcal{H}_{\{\sigma,\kappa+\epsilon-\nu t\}}}  \big(\nabla J^{\lfloor 2r\rfloor} u(x,t) \big)^2 \langle \sigma \cdot x \rangle^{2\big(r-\big(\frac{\lfloor 2r\rfloor+1}{2}\big)\big)} \, \mathrm{d}x\, \mathrm{d}t  \leq c^{\ast}.
\end{equation}
Moreover, if $r\geq \frac{1}{2}$
\begin{equation}\label{MainTHMeq3}
\sup_{\delta\leq t \leq T}  \int_{\mathcal{H}_{\{\sigma,\kappa+\epsilon-\nu t\}}}  \big(J^{s} u(x,t) \big)^2 \langle \sigma \cdot x \rangle^{2r_s} \, \mathrm{d}x 
 \leq c^{\ast},
\end{equation}
where 
\begin{equation*}
 r_s=\max\left\{\left(1-\frac{s}{\lfloor 2r \rfloor}\right)r,r-\frac{\lceil s \rceil}{2}\right\},
 \end{equation*}
$0\leq s \leq \lfloor 2r \rfloor,$ and $c^{\ast}$ has the same dependence that in \eqref{MainTHMeq1}. 
\end{theorem}
The proof of Theorem \ref{mainTHM} is based on weighted energy estimates which extend the arguments introduced by Izasa, Ponce, and Linares \cite{IsazaLinaresPonce2015} to the case where both regularity and decay are measured by fractional parameters. To achieve such a goal, we introduce a new class of pseudo-differential operators (see, Definition \ref{pesudodeff}), and we deduce some properties of the pseudo-differential calculus for such class, see also Appendix  \ref{Appendix}. As an outcome of our results, we establish new commutator formulas between derivatives and weights in Lemma \ref{lemmadecom}, Corollaries \ref{supportsepareted3}, and \ref{remlemmadecom}. Furthermore, we obtain new propagation formulas in Lemma \ref{lemmapropaform}, which are fundamental to keeping localization properties of decay and regularity. We also deduce interpolation formulas in Lemma \ref{InterpLemma}. We believe that the pseudo-differential calculus here defined is of independent interest, it can be applied to study the fractional propagation of weights for different dispersive models in any dimension.

We notice that to justify the energy estimates in the proof of Theorem \ref{mainTHM}, one requires an approximation argument based on the existence of smooth solutions of \eqref{ZK} with enough decay, as well as, a continuous dependence statement. To justify such an approach rigorously, we establish a more general local well-posedness result in weighted anisotropic Sobolev spaces for solutions of \eqref{ZK} in any spatial dimension $n\geq 1$ (with $n=1$ being \eqref{KdV}).
\begin{theorem}\label{wellposweighted}
Let $n\geq 1$, $r_1,\dots, r_n\in \mathbb{R}^{+}\cup\{0\}$, and $s\geq \max\left\{2r_1,\dots,2r_n,(\frac{n}{2}+1)^{+}\right\}.$ Then the Cauchy problem \eqref{ZK} is locally well-posed in $$H^{s}(\mathbb{R}^n)\cap L^2\left(\left(|x_1|^{2r_1}+\dots+|x_n|^{2r_n}\right)\, dx\right).$$
\end{theorem}
The results of Theorem \ref{wellposweighted} in dimensions $n=1$, and $n=2$ with improvements on the regularity condition have been studied before, for example, in $n=1$, see  \cite{FonsecaLinaresPonce2015}, when $n=2$, see the well-posedness results in  \cite{FonsecaPachon2016}. When the dimension $n\geq 3$, our conclusions in Theorem \ref{wellposweighted} seem to be the first result in anisotropic spaces for solutions of \eqref{ZK}. Certainly, our arguments admit an improvement on the condition $s>\frac{n}{2}+1$, which we have used for simplicity in our arguments since we do not pursue minimal regularity in this paper. For more details on the validity of our results, see Remark \ref{Remarkscont} (v) below when $n\geq 2$, and see Remark \ref{Remarkscont2} (iii) in dimension $n=1$. An interesting problem is establishing well-posedness in anisotropic spaces with the exact condition $s\geq \max\{2r_1,\dots,2r_n\}$. 

The proof of the Theorem \ref{wellposweighted} is inspired by the ideas for the fractional KdV in \cite{CunhaRiano2022}, which are based on the LWP in $H^s(\mathbb{R}^n)$, $s>\frac{n}{2}+1$, and an approximation argument using smooth solutions in anisotropic spaces. A fundamental tool in the proof of Theorem \ref{wellposweighted} is the interpolation estimates in Lemma \ref{InterpLemma}, which are a consequence of the pseudo-differential calculus developed below. Because this paper focuses on the propagation of weighted decay, the proof of Theorem \ref{InterpLemma} is given in Appendix  \ref{AppendixWellposs}.
\begin{remark}    
Let us review the decay condition \eqref{MainTHMeq3} in Theorem \ref{mainTHM}.  The condition $\big(1-\frac{s}{\lfloor 2r \rfloor}\big)r$ in \eqref{MainTHMeq3} is a consequence of the interpolation Lemma \ref{InterpLemma}, which is obtained by interpolating the maximum gain of regularity $\lfloor 2r \rfloor$, and the maximum decay $r$ given by the initial condition. The condition $r-\frac{\lceil s\rceil}{2}$ is a consequence of the proof of Theorem \ref{mainTHM}, which is based on an iterative method where the size of the weight decreases in proportion to the gain of regularity. More precisely, when there is no regularity, the decay is maximum, after one iteration of our argument, we find that any derivative of order one of the solution $u$ of ZK has a decay of order $r-\frac{1}{2}$ (here we use Lemma \ref{lemmapropaform} to obtained that for every regularity $0<s\leq 1$, $J^su$ decays of order $r-\frac{1}{2}$). Consequently, we keep iterating, until we reach localized regularity of order $\lfloor 2r \rfloor$, for which the maximum decay is of order $r-\frac{\lfloor 2r\rfloor}{2}$. Figure \ref{fig:P1} shows a representation of the argument just described.
\begin{figure}[ht]
\resizebox{.65\textwidth}{!}{
\begin{tikzpicture}
 \draw[->] (0,0) -- (7,0) node[right] {Regularity};
 \draw[->] (0,0) -- (0,3.8) node[above] {Decay };
\fill[blue] (0,3.3) circle (2pt) node[right,xshift=0.1cm] {$r$}; 
\draw[-,blue] (0.1,2.5) -- (1.5,2.5) node[above,xshift=-0.7cm] {$r-\frac{1}{2}$};
\node at (0,2.5) {$\circ$};
\draw[-] (1.5,-0.1) -- (1.5,0.1) node[below,yshift=-0.2cm] {1};
\draw[-,blue] (1.53,1.7) -- (3,1.7) node[above,xshift=-0.7cm] {$r-1$};
\node at (1.5,1.7) {$\circ$};
\draw[-] (3,-0.1) -- (3,0.1) node[below,yshift=-0.2cm] {2};
\draw[dotted] (3,0.95) -- (4.5,0.95);
\draw[-,blue] (4.53,0.3) -- (6,0.3) node[above,xshift=-0.7cm] {$r-\frac{\lfloor 2r \rfloor}{2}$};
\node at (4.5,0.3) {$\circ$};
\draw[-] (6,-0.1) -- (6,0.1) node[below,yshift=-0.2cm] {$\lfloor 2r \rfloor$};
\end{tikzpicture}}
\caption{Distribution between decay and regularity condition $r-\frac{\lceil s \rceil}{2}$ in \eqref{MainTHMeq3}} \label{fig:P1}
\end{figure}
\end{remark}
\begin{remark}\label{Remarkscont}
\begin{itemize}[leftmargin=15pt]
\item[(i)] Theorem \ref{TheorArg} implies that extra regularity on the initial data is propagated by the solution flow of \eqref{ZK}. In contrast, Theorem \ref{mainTHM} establishes that localized decay on the initial data \eqref{inicond} is propagated by the corresponding solution of \eqref{ZK}, see \eqref{MainTHMeq1}, but in addition, we obtain a gain of extra localized regularity, which is determined by the magnitude of $r$ according to $\lfloor 2r \rfloor$, see \eqref{MainTHMeq2} and \eqref{MainTHMeq3}. When $r>0$ is an integer, \eqref{MainTHMeq2} shows that propagation of a weight of order $r$ produces regularity of order up to $2r$. We formally expect such a relation between decay and regularity by checking the following commutator identity \footnote{We denote the \emph{commutator operator} between two operators $A$ and $B$  by  $[A, B]:=AB-BA$, and $\delta_{1,j}$ denotes the \emph{ Kronecker's delta}  function.}
    \begin{equation}\label{eqremark1}
        \begin{aligned}
        \big[S(t),x_j \big]=-\delta_{1,j}\Delta-2\partial_{x_1}\partial_{x_j},
        \end{aligned}
    \end{equation}
where $S(t)$ denotes the group of liner operators associated with solutions of the linear equation in \eqref{ZK}, that is, $\partial_t u+\partial_{x_1}\Delta u=0$. Notice that \eqref{eqremark1} determines the balance between decay and regularity.
\item[(ii)] Let $\sigma=e_1$, where $e_1$ denotes the canonical vector in the first component. In this case, Theorem \ref{mainTHM} shows the propagation of decay on the first variable, i.e., with weighted function $\langle x_1 \rangle$ on the region $\{x_1\geq \kappa\}$. Thus, in the integer case, the case $\sigma=e_1$ yields results compatible with the one-dimensional propagation of weights established for the KdV equation, see \cite[Theorem 1.2]{IsazaLinaresPonce2015}. In this sense, Theorem \ref{mainTHM} is an extension to the higher dimensions of the results in \cite{IsazaLinaresPonce2015}.

\item[(iv)] Given $r>0$ with $2r\notin \mathbb{Z}^{+},$ then  $r-\frac{(\lfloor 2r\rfloor+1)}{2}<0$, which means that \eqref{MainTHMeq2} provides a type of smoothing effect, which, regarding the technique used, it shows the maximum regularity transferred by the polynomial decay of the initial data. In contrast, \eqref{MainTHMeq2.1} also establishes the maximum gain of regularity propagated without incorporating any weight, but this holds on a ``smaller" domain.  
\item[(v)] Setting $n=2,3$ in Theorem \ref{mainTHM}, and $n\geq 1$ arbitrary in Theorem \ref{wellposweighted}, the results in these theorems can be extended to initial conditions $u_0\in H^s(\mathbb{R}^n)$ for some $s\geq 0$, provided that there exists a local theory (i.e., existence, uniqueness, and continuous dependence) in $H^s(\mathbb{R}^n)$ for which the solutions $u$ of ZK satisfy
\begin{equation}\label{l1cond}
    u \in L^{1}([0,T];W^{1,\infty}(\mathbb{R}^n)).
\end{equation}
In contrast, when $n\geq 4$, by technical reasons detailed in Remark \ref{remarkl1condt} below, the proof of Theorem Theorem \ref{mainTHM} requires the regularity assumptions $H^{\big(\frac{n+2}{2}\big)^{+}}(\mathbb{R}^n)$ for solutions of ZK. We believe that such a restriction on the regularity could be improved, for example, with a more refined expansion of the commutator estimate in Lemma \ref{conmKP} when $n\geq 4$. 
\end{itemize}
\end{remark}
We present our third main result.
\begin{theorem}\label{expodecay}
Let $n\geq 2$ be fixed, $u_0 \in H^{\big(\frac{n+2}{2}\big)^{+}}(\mathbb{R}^n)$, and $\sigma=(\sigma_1,\sigma_2,\dots,\sigma_n)\in \mathbb{R}^n$, satisfying condition \eqref{sigmacond}. If for some $b>0$, and $\kappa\in \mathbb{R}$
\begin{equation}\label{decayiniexp}
\|e^{b\sigma\cdot x} u_0\|_{L^2(\mathcal{H}_{\{\sigma,\kappa\}})}<\infty,
\end{equation}
then the corresponding solution $u=u(x,t)\in C([0,T];H^{\big(\frac{n+2}{2}^{+}\big)}(\mathbb{R}^n))$ of the IVP \eqref{ZK} satisfies: For any $\nu > 0$, $\epsilon>0$, 
\begin{equation}\label{expresult1}
\sup_{0\leq t \leq T} \int_{\mathcal{H}_{\{\sigma,\kappa+\epsilon-\nu t\}}} \big(e^{b \sigma\cdot x} u(x,t) \big)^2\,  \mathrm{d}x < \infty.
\end{equation}
Moreover, for any multi-index $\beta$, any $\nu> 0$, $\delta>0$ $\epsilon>0$, 
\begin{equation}\label{expresult2}
\sup_{\delta\leq t \leq T}  \int_{\mathcal{H}_{\{\sigma,\kappa+\epsilon-\nu t\}}}  \big(e^{b \sigma \cdot x}\partial^{\beta} u(x,t) \big)^2  \, \mathrm{d}x <\infty.
\end{equation}
\end{theorem}
\begin{remark}
\begin{itemize}[leftmargin=10pt]
\item[(i)] The condition \eqref{sigmacond} in Theorems \ref{mainTHM} and \ref{expodecay} also appears in other studies on the dynamics of the ZK equation. To see this, setting  $\rho=\sqrt{\sigma_{2}^{2}+\dots+\sigma_{n}^{2}}$, it follows from \eqref{sigmacond}  $$\left\{(\sigma_1,\rho):\sigma_1>0,\, \rho\geq 0,\,  \sqrt{3}\sigma_1>\rho\right\}\subset \left\{(\sigma_1,\rho):\sigma_1>0,\, \rho\geq 0, \, \, \tan\Big(\frac{\rho}{\sigma_1}\Big)\leq \frac{\pi}{3}\right\}.$$ 
The angle $\frac{\pi}{3}$ also determines the region $\mathcal{H}_{\{\sigma,\kappa\}}$. This same angle and similar regions also appear in the study of asymptotic stability of solutions of \eqref{ZK}: In the $n=2$ ZK equation, see the work of C\^{o}te, Mu\~{n}oz, Pilod, and Simpson \cite{CoteMunozPilodSimpson2016}, and when $n=3$, see also Farah, Holmer, Roudenko and Yang \cite{FarahHolmerRoudYang2023}. The condition \eqref{sigmacond} also appears in the well-posedness result for the exponential initial condition in the work of Del\'eage and Linares \cite{DeleageLinares2023}. In general, it seems that conditions similar to \eqref{sigmacond} are also needed in some studies of dispersive generalizations of the ZK equation, see the numerical investigations in \cite{RianoRoudenkoYang2022} and the propagation of regularity results in \cite{Argenis2023}.

We believe that the region determined by  \eqref{sigmacond} may be optimal in the propagation of extra regularity and decay. Nevertheless, we do not address such a problem in this work.
    \item[(ii)] In the case of localized exponential decay on the initial data (i.e., condition \eqref{decayiniexp}), our arguments show that the corresponding solution of ZK is smooth on the domain $\mathcal{H}_{\{\sigma,k+\epsilon-\nu t\}}$, with $t\in[0, T]$. This contrasts with the polynomial case where polynomial decay on the initial data of order $r$ (i.e., \eqref{inicond}), yields a gain of maximum regularity of order $\lfloor 2r \rfloor$, see \eqref{MainTHMeq3}. An interesting problem is to study whether solutions of ZK with initial data satisfying \eqref{decayiniexp} are analytic in the domain $\mathcal{H}_{\{\sigma,k+\epsilon-\nu t\}}$ with $t\in[0, T]$. This question is beyond the scope of this manuscript.
\item[(iii)] When $n=2,3$, the results of Theorem \ref{expodecay} are also valid for regularity $H^s(\mathbb{R}^n)$ for some $s\geq 0$ provided that there exists a local well-posedness result for solution of \eqref{ZK} for which the solutions generated also belong to the class \eqref{l1cond}.
\end{itemize}
\end{remark}
As an outcome of the proof of Theorems \ref{mainTHM} and \ref{expodecay}, we deduce propagation of fractional polynomial and exponential weights for solutions of the KdV equation. We will use the local well-posedness theory in $H^{\frac{3}{4}^{+}}(\mathbb{R})$ deduced by Kenig, Ponce, and Vega in \cite{KenigPonceVega1993}. 
\begin{theorem}\label{mainTHMKdV}
Let $u_0 \in H^{\frac{3}{4}^{+}}(\mathbb{R})$.
\begin{itemize}[leftmargin=15pt]
    \item[(i)] \underline{{\bf Propagation of polynomial weights}}:
     If for some $\kappa\in \mathbb{R}$, and $r\in \mathbb{R}^{+}$,
\begin{equation}\label{condMainKdV}
\|\langle x \rangle^r u_0\|_{L^2((\kappa,\infty))}<\infty,
\end{equation}
then the corresponding solution $u=u(x,t)\in C([0,T];H^{\frac{3}{4}^{+}}(\mathbb{R}))$ of the IVP \eqref{KdV} with initial condition $u_0$ satisfies: For any $\nu > 0$, $\epsilon>0$, 
\begin{equation}\label{KdVMain1}
\sup_{0\leq t \leq T} \int_{\kappa}^{\infty} \big(\langle x \rangle^r u(x,t) \big)^2 \mathrm{d}x \leq c^{\ast}
\end{equation}
with $c^{\ast}=c^{\ast}\big(\epsilon,\sigma,T,\nu,\|u_0\|_{H^{\frac{3}{4}^{+}}},\|\langle x \rangle^r u_0\|_{L^2((\kappa, \infty))}\big)$. Furthermore, for any $\nu> 0$, $\delta>0$ $\epsilon>0$, and $\tau \geq 5\epsilon$,
\begin{equation*}
\int_{\delta}^T \int_{\kappa+\epsilon-\nu t}^{\kappa+\tau-\nu t} \big(J^{\lfloor 2r \rfloor+1}u\big)^2 \mathrm{d}x \, \mathrm{d}t< c^{\ast},
\end{equation*}
and if in addition $2r\notin \mathbb{Z}^{+}$, 
\begin{equation*}
\int_{\delta}^{T}\int_{\kappa+\epsilon-\nu t}^{\infty} \big(\partial_x J^{\lfloor 2r\rfloor} u(x,t) \big)^2 \langle x \rangle^{2\big(r-\frac{(\lfloor 2r\rfloor+1)}{2}\big)} \, \mathrm{d}x \, \mathrm{d}t
 \leq c^{\ast}.
\end{equation*}
Moreover, if $r\geq \frac{1}{2}$,
\begin{equation}\label{KdVfracprop}
\sup_{\delta\leq t \leq T}  \int_{\kappa+\epsilon-\nu t}^{\infty} \big(J^{s} u(x,t) \big)^2 \langle  x \rangle^{2r_s} \, \mathrm{d}x,
 \leq c^{\ast},
\end{equation}
where $$r_s=\max\left\{\left(1-\frac{s}{\lfloor 2r \rfloor}\right)r,r-\frac{\lceil s \rceil}{2}\right\}$$ and    $0\leq s \leq \lfloor 2r \rfloor.$ 
\item[(ii)] \underline{{\bf Propagation of exponential weights}:} If for some $b>0$, $\kappa\in \mathbb{R}$,
\begin{equation}\label{condo1dexp}
\|e^{b x} u_0\|_{L^2((\kappa,\infty))}<\infty,
\end{equation}
then the corresponding solution $u=u(x,t)\in C([0,T];H^{\frac{3}{4}^{+}}(\mathbb{R}))$ of the IVP \eqref{KdV} with initial condition $u_0$ satisfies: For any $\nu > 0$, $\epsilon>0$, 
\begin{equation*}
\sup_{0\leq t \leq T} \int_{\kappa}^{\infty} \big(e^{b x } u(x,t) \big)^2\,  \mathrm{d}x < \infty.
\end{equation*}
Moreover, for any multi-index $\beta$, any $\nu> 0$, $\delta>0$ $\epsilon>0$, 
\begin{equation*}
\sup_{\delta\leq t \leq T}  \int_{\kappa+\epsilon-\nu t}^{\infty}  \big(e^{b x}\partial^{\beta}_{x} u(x,t) \big)^2  \, \mathrm{d}x <\infty.
\end{equation*}
\end{itemize}
\end{theorem}
\begin{remark}\label{Remarkscont2}
\begin{itemize}[leftmargin=15pt]
\item[(i)] Theorems \ref{mainTHM} and \ref{mainTHMKdV} (i) seem to be the first results in the propagation of localized fractional polynomial decay. Similarly, Theorems \ref{expodecay} and \ref{mainTHMKdV} (ii) is to our knowledge the first results concerning the propagation of exponential weights. We believe that the arguments developed in this manuscript can be adapted in the study of localized propagation of weights for other partial differential equations of dispersive type. 
\item[(ii)] Theorem \ref{mainTHMKdV} generalizes to fractional weights the results of Isaza, Linares, and Ponce \cite[Theorem 1.2]{IsazaLinaresPonce2015}. To see this, let $m\in \mathbb{Z}^{+}$, taking $r=\frac{m}{2}$ and $s=l_1\leq m$ in  \eqref{KdVfracprop} yield
\begin{equation}\label{IsazaLinaresPOnceres1}
\sup_{\delta\leq t \leq T}  \int_{\kappa+\epsilon-\nu t}^{\infty} \big(J^{l_1} u(x,t) \big)^2 \langle  x \rangle^{m-l_1} \, \mathrm{d}x<\infty,
\end{equation}
which in particular implies when $m\geq 2$,
\begin{equation}\label{IsazaLinaresPOnceres2}
\int_{\delta}^T \int_{\kappa+\epsilon-\nu t}^{\infty} \big(J^{l'_1+1} u(x,t) \big)^2 \langle  x \rangle^{m-l'_1-1} \, \mathrm{d}x\, \mathrm{d}t
<\infty,
\end{equation}
where $1\leq l'_1\leq m-1$. Notice that by Lemma \ref{lemmapropaform}, we can replace $J^{l_1}$, and $J^{l_1'+1}$ in the above expressions by any derivative of order $l_1$ and $l_1'+1$, respectively. Hence, \eqref{IsazaLinaresPOnceres1} and \eqref{IsazaLinaresPOnceres2} correspond to the propagation of integer weights deduced in \cite[Theorem 1.2]{IsazaLinaresPonce2015}. Nevertheless, Theorem \ref{mainTHMKdV} shows an explicit connection between the propagation of fractional regularity and integer decay, thus enlarging the study initiated in \cite{IsazaLinaresPonce2015}. 
\item[(iii)] The proof of Theorem \ref{mainTHMKdV} is based on energy estimates which also require solutions of \eqref{KdV} to be in the class
\begin{equation*}
    u \in L^{1}([0,T];W^{1,\infty}(\mathbb{R})).
\end{equation*}
Notice that condition above is given by the local theory in $H^{\frac{3}{4}^+}(\mathbb{R})$ due to Kenig, Ponce, and Vega in \cite{KenigPonceVega1993}, which assures the existence of solutions $u$ of KdV such that
\begin{equation*}
    u \in C([0,T];H^{\frac{3}{4}^{+}}(\mathbb{R})),\, \text{ and } \, \partial_x u \in L^4([0,T];L^{\infty}(\mathbb{R})).
\end{equation*}
However, the existence of solutions for the Cauchy problem \eqref{KdV} is known in $H^{s}(\mathbb{R})$, $s\geq -1$ (see, \cite{RowanVisan2019, KenigPonceVega1993,BourgainI1993, Guo2009,KenigPonceVega1996, ColKeStaTakaTao2003,ChristColliTao2003}). An interesting problem would be to investigate the fractional propagation of weights when $u_0\in H^s(\mathbb{R})\cap L^2((\kappa,\infty); \langle x \rangle^{2r}\, dx)$, $s\leq \frac{3}{4}.$ 

Furthermore, the previous discussion and our argument imply that when $n=1$, the results of Theorem \ref{wellposweighted} establish well-posedness results in the space $H^{s}(\mathbb{R})\cap L^2(|x|^{2r}\, dx)$ with $r>0$, $s\geq\left\{2r,\frac{3}{4}^{+}\right\}$. 
\end{itemize}
\end{remark}
\subsection{Organization}
This paper is organized as follows: In Section \ref{pseudoSection}, we introduce and develop our results for the class of pseudo-differential operators in Definition \ref{pesudodeff}. In this same section, we also present some key consequences of pseudo-differential calculus such as propagation and interpolation formulas. We conclude Section \ref{pseudoSection} by introducing the weighted approximations that will be used along with our arguments. Section \ref{sectionMain} concerns the deduction of the propagation of fractional polynomial decay in Theorem \ref{mainTHM}. In Section \ref{ExpodecaySect}, we deduce the propagation of localized exponential decay stated in Theorem \ref{expodecay}. Next, in Section \ref{sectionKdV}, we establish our results for the KdV equation, i.e., we prove Theorem \ref{mainTHMKdV}. Finally, we present an appendix where in the first part, Appendix \ref{AppendixWellposs}, we prove the well-posedness results in anisotropic weighted Sobolev spaces detailed in Lemma \ref{wellposweighted}, and we conclude with Appendix \ref{Appendix}, where we establish continuity and asymptotic expansion for the class of pseudo-differential operators introduced in Section \ref{pseudoSection}.

\subsection{Acknowledgments}

The authors wish to extend their gratitude to Prof. Felipe Linares for the careful reading and helpful comments towards improving the manuscript. O. R. acknowledges financial support from Universidad Nacional de Colombia, Bogot\'a.

\subsection{Notation and preliminaries}
Given two positive quantities $a$ and $b$,  $a\lesssim b$ means that there exists a positive constant $c>0$ such that $a\leq c b$. We write $a\sim b$ to symbolize that $a\lesssim b$ and $b\lesssim a$. We will denote by $\mathbb{N}_{0}=\mathbb{N}\cup\{0\}$. Given $m\in \mathbb{R}$, the \emph{Bessel potential} of order $-m$ is denoted by $J^m=(1-\Delta)^{\frac{m}{2}}$. We will use the standard \emph{Lebesgue spaces} $L^p(\mathbb{R}^n)$, $1\leq p\leq \infty$ with norm $\|f\|_{L^p}$. $\mathcal{S}(\mathbb{R}^{n})$ denotes the \emph{Schwartz class} of functions. $H^{s}(\mathbb{R}^n)$ denotes the \emph{Sobolev space} of square integrable tempered distributions $f$ for which $\|J^s f\|_{L^2}<\infty$. 
\section{Preliminary results and Pseudo-differential Calculus}\label{pseudoSection}
In this section, we introduce a new class of pseudo-differential operators with polynomial decay. 
\begin{defi}\label{pesudodeff}
Let $\sigma \in \mathbb{R}^n$ and $m,\omega, q\in \mathbb{R}$. We define $\mathbb{S}_{\sigma,\omega}^{m,q}(\mathbb{R}^n\times \mathbb{R}^n)$ as the set  of functions $a \in C^{\infty}(\mathbb{R}^n\times \mathbb{R}^n)$ such that for all multi-index $\alpha, \beta\in \mathbb{N}_{0}^{n} $, there exists a constant $c>0,$ such that
\begin{equation*}
|\partial^{\alpha}_x\partial^{\beta}_{\xi}a(x,\xi)|\leq c\langle \sigma \cdot x+\omega \rangle^{q-|\alpha|}\langle \xi \rangle^{m-|\beta|},\quad \mbox{for all}\quad x,\xi\in\mathbb{R}^{n}.
\end{equation*}
We call an element  $a\in \mathbb{S}_{\sigma,\omega}^{m,q}(\mathbb{R}^n\times \mathbb{R}^n)$  a symbol  of order $(m,q).$
\end{defi}
\begin{remark}
In dimension $n=1$, the class $\mathbb{S}_{\sigma,\omega}^{m,q}$ has been used to describe properties of solutions of the nonlinear Schr\"odinger equation, e.g.,  \cite{KenigPonceRolvungVega2005}. Furthermore, in the one-dimensional setting, an equivalent class has been used to describe properties of the critical dispersive generalized Benjamin-Ono equation in \cite{KenigMartelRobbiano2011}. However, up to our knowledge, the case $n\geq 2$ in Definition \ref{pesudodeff} has not been considered before. 
\end{remark}
\begin{example}\label{examppseudo}
In the  case $\sigma=0,$ the class $ \mathbb{S}_{0,\omega}^{m,q}$ agrees with the  classical class of symbols $\mathbb{S}^{m}$ defined by Kohn and Nirenberg in  \cite{KohnNirenberg1965}. More precisely,  for $m\in\mathbb{R}$, we have $a\in\mathbb{S}^{m}$, if  $a \in C^{\infty}(\mathbb{R}^n\times \mathbb{R}^n)$, and  for all $\alpha,\beta\in\mathbb{N}_{0}^{n},$ there exists a constant $c>0,$ such that 
	\begin{equation}\label{kn}
		|\partial_{x}^{\beta}\partial_{\xi}^{\alpha}a(x,\xi)|\leq c\langle \xi \rangle^{m-|\alpha|}\quad\mbox{for all}\quad x,\xi\in\mathbb{R}^{n}.
	\end{equation}
\end{example}
\begin{example}
 Let $m,\omega, q\in \mathbb{R}.$ For $\sigma\neq 0,$ the function
	\begin{equation*}
		a(x,\xi)=\langle \sigma\cdot x+\omega \rangle^{q}\langle \xi\rangle ^{m} \quad\mbox{for}\quad x,\xi \in\mathbb{R}^{n}
	\end{equation*}
defines a symbol in the class $\mathbb{S}_{\sigma,\omega}^{m,q}.$
\end{example}
\begin{example}
Let $e_j\in \mathbb{R}^n$ be the canonical vector in the $j$ direction for some $1\leq j \leq n$. Setting $\omega=0$, we have that $a\in \mathcal{S}^{m,q}_{e_j,0}$, if for all $\alpha,\beta\in\mathbb{N}_{0}^{n},$ there exists $c>0$ such that
	\begin{equation}\label{extraExm1}
		|\partial_{x}^{\beta}\partial_{\xi}^{\alpha}a(x,\xi)|\leq c\langle x_j \rangle^{q-|\beta|}\langle \xi\rangle^{m-|\alpha|}\quad\mbox{for all}\quad x,\xi\in\mathbb{R}^{n}.
	\end{equation}
Thus, the class $a\in \mathcal{S}^{m,q}_{e_j,0}$ only considers certain polynomial decay on the $j$-variable. In this regard, the previous class can also be seen as a higher dimensional extension of the pseudo-operators studied in \cite{KenigMartelRobbiano2011, KenigPonceRolvungVega2005}.
\end{example}
\begin{example}
A variant of the symbol defined above is the following:	Let $\chi\in \C^{\infty}_{0}(\mathbb{R})$  and $m,\omega,q\in \mathbb{R},$ then the function
	\begin{equation*}
		a(x,\xi)=\langle \sigma\cdot x+\omega \rangle^{q}\chi(\sigma\cdot x+\omega)\langle \xi\rangle ^{m} \quad\mbox{for}\quad x,\xi \in\mathbb{R}^{n}
	\end{equation*}
	defines a symbol in $\mathbb{S}_{\sigma,\omega}^{m,q}.$ 
\end{example}
Before proceeding with the description of some properties of the class of symbols in Definition \ref{pesudodeff}, we will fix some further notations.
\begin{remark}
We will suppress the dependence on the domain when we indicate a symbol in the class $\mathbb{S}_{\sigma,\omega}^{m,q}(\mathbb{R}^n\times \mathbb{R}^n)$, that is, to refer to a symbol $a$ in the class determined by $m,q,\sigma,\omega$,  we write  $a\in  \mathbb{S}_{\sigma,\omega}^{m,q}.$
\end{remark}
	\begin{remark}
			Given a symbol $a\in \mathbb{S}_{\sigma,\omega}^{m,q}$, we write  the \emph{pseudo-differential operator associated with the symbol $a$} as    
		\begin{equation*}
			\Psi_a f(x)=\int_{\mathbb{R}^{n}} a(x,\xi) \widehat{f}(\xi)\,e^{2\pi i x\cdot \xi} \, d\xi,\, x\in\mathbb{R}^{n}.
		\end{equation*}
	\end{remark}
\begin{defi}
	If   $\Sigma$   is any class of symbols  and $a(x,\xi)\in \Sigma$,  we say that  $\Psi_{a}\in \mathrm{OP}\Sigma.$ 
\end{defi}
One of the main properties that underlie the study of pseudo-differential operators is the continuity in Lebesgue spaces. The description of our results depends strongly on the following lemma.
\begin{lemma}\label{contiprop}
 Let $\sigma \in \mathbb{R}^n$ and $m,\omega,q \in \mathbb{R}$ be fixed. Consider $1<p<\infty$ and $\Psi_{a}\in \mathrm{OP} \mathbb{S}_{\sigma,\omega}^{m,q}$. Then
\begin{equation}\label{continuity}
\|\Psi_a f\|_{L^p}\lesssim \|\langle \sigma \cdot x+\omega \rangle^{q}J^{m}f\|_{L^p}.
\end{equation}
\end{lemma}
\begin{proof}
The proof is detailed in Appendix \ref{Appendix} below.
\end{proof}
We will use the following asymptotic analysis to derive some estimates connecting localized weights and regularity.
\begin{prop}\label{PO1}
Let $\sigma \in \mathbb{R}^n$ and $m_{1},m_{2}, \omega,q_{1},q_{2} \in \mathbb{R}$. Consider $a\in \mathbb{S}_{\sigma,\omega}^{m_1,q_1}$ and $b\in\mathbb{S}_{\sigma,\omega}^{m_2,q_2}$. Then, there exists $c \in \mathbb{S}^{m_1+m_2-1,q_1+q_2-1}_{\sigma,\omega}$ such that 
\begin{equation*}
     [\Psi_a,\Psi_b]=\Psi_a \Psi_b-\Psi_b \Psi_a=:\Psi_c.
\end{equation*}
Moreover, 
\begin{equation*}
    c \approx \sum_{|\beta|>0 }\frac{1}{(2\pi i)^{|\beta|} \beta !}\Big((\partial^{\beta}_{\xi}a)\cdot(\partial^{\beta}_{x}b)-(\partial^{\beta}_{\xi}b)\cdot(\partial^{\beta}_{x}a) \Big),
\end{equation*}
in the sense that 
\begin{equation*}
    c - \sum_{0<|\beta| < N}\frac{1}{(2\pi i)^{|\beta|}\beta !}\Big((\partial^{\beta}_{\xi}a)\cdot(\partial^{\beta}_{x}b)-(\partial^{\beta}_{\xi}b)\cdot(\partial^{\beta}_{x}a) \Big) \in \mathbb{S}^{m_1+m_2-N,q_{1} + q_{2}- N}_{\sigma, \omega},
\end{equation*}
for all $N \geq 2$.
\end{prop}
\begin{proof}
We refer to Appendix \ref{Appendix} below.
\end{proof}
\subsection{Propagation Formulas and Interpolation}
In this part, we deduce some consequences of the pseudo-differential calculus introduced above. We emphasize that the results in this subsection are of independent interest and they can be applied in different situations.
The following Lemma is fundamental in the proof of Theorem \ref{mainTHM}.
\begin{lemma}\label{lemmadecom}
Let $\sigma \in \mathbb{R}^n$ and $\omega \in \mathbb{R}$. Let  $g(x,\xi)=g(x)\in \mathbb{S}^{0,q_1}_{\sigma, \omega}$, and $a\in \mathbb{S}^{m_2,q_2}_{\sigma, \omega}$ for some $ m_2,q_1,q_2\in \mathbb{R}$. Then, for each $N\in \mathbb{N},$  there exist pseudo-differential operators $\Psi_{a_{k}}\in \mathrm{OP} \mathbb{S}^{m_2-k,q_2}_{\sigma, \omega}$, with $k=1,\dots,N$ and $\mathcal{K}_N\in  \mathrm{OP}\mathbb{S}^{m_2-N-1,q_1+q_2-N-1}_{\sigma, \omega},$ such that the following factorization holds true
\begin{equation}\label{decomident}
\begin{aligned}
g\Psi_a f=\Psi_a(g f)+\sum_{k=1}^N \sum_{|\beta|=k} \Psi_{a_k}(\partial^{\beta}g f)+\mathcal{K}_{N}(f).
\end{aligned}
\end{equation}
For simplicity in the notation, we are using $\partial^{\beta}g$ for both the symbol and the pseudo-differential operator determined by it. Moreover, for any $N\geq 1$ there exist pseudo-differential operators  $\widetilde{\Psi}_{a_{k}}\in \mathrm{OP}\mathbb{S}^{m_2-k,q_2}_{\sigma, \omega}$, $k=1,\dots,N$ and $\widetilde{\mathcal{K}}_N\in\mathrm{OP} \mathbb{S}^{m_2-N-1,q_1+q_2-N-1}_{\sigma, \omega}$ such that
\begin{equation}\label{decomident2}
\begin{aligned}
\Psi_a (g f)=g\Psi_a( f)+\sum_{k=1}^N \sum_{|\beta|=k}\partial^{\beta} g \widetilde{\Psi}_{a_k}(f)+\widetilde{\mathcal{K}}_N(f).
\end{aligned}
\end{equation}
\end{lemma}
\begin{proof}
We will only deduce \eqref{decomident} as the proof of \eqref{decomident2} follows by a similar reasoning. Writing
\begin{equation*}
g\Psi_a f=[g,\Psi_a]f+\Psi_a(gf),
\end{equation*}
we are reduced to decompose the commutator term in the above identity. By Proposition \ref{PO1}, there exist pseudo-differential operators of order $\Psi_{a,m_2-|\beta|}$ with symbol in the class $\mathbb{S}^{m_2-|\beta|,q_2}_{\sigma, \omega}$ for each multi-index $0<|\beta|\leq N$, and $K$ with symbol in $\mathbb{S}^{m_2-N-1,q_1+q_2-N-1}_{\sigma, \omega}$ such that
\begin{equation*}
[g,\Psi_a]f=\sum_{0<|\beta|\leq N} \partial^{\beta}g \Psi_{a,m_2-|\beta|}(f)+K(f),
\end{equation*}
where we have used that $g$ depends only on the $x$ variable. We observe that $\partial^{\beta}g \Psi_{a,m_2-|\beta|}$ determines a pseudo-differential operator in the class $ \mathbb{S}^{m_2-|\beta|,q_1+q_2-|\beta|}_{\sigma, \omega}$, $0<|\beta|\leq N$. This suggests that 
$$ \partial^{\beta}g \Psi_{a,m_2-|\beta|}(f)=[\partial^{\beta}g,\Psi_{a,m_2-|\beta|}]f+\Psi_{a,m_2-|\beta|}(\partial^{\beta}g f),$$
and by noticing that the last term of the above identity is part of  \eqref{decomident}, we can repeat the above argument, expanding the commutator in the identity above to lower the class of the operator until we deduce \eqref{decomident}. More precisely, fixing a multi-index $\beta$ with $0<|\beta|\leq N$, by Proposition \ref{PO1} we get
\begin{equation*}
\begin{aligned}
\left[\partial^{\beta}g, \Psi_{a,m_2-|\beta|}\right]f=&\sum_{0<|\beta_1|\leq N} \partial^{\beta+\beta_1}g \Psi_{a,m_2-|\beta|-|\beta_1|}(f)+K_2(f),
\end{aligned}
\end{equation*}
where $ \Psi_{a,m_2-|\beta|-|\beta_1|}\in\mathrm{OP}\mathbb{S}^{m_2-|\beta|-|\beta_1|,q_2}$ and $K_2\in\mathrm{OP}\mathbb{S}^{m_2-|\beta|-N-1,q_1+q_2-|\beta|-N-1}_{\sigma, \omega}$. Since $\partial^{\beta+\beta_1}g \Psi_{a,m_2-|\beta|-|\beta_1|}\in \mathrm{OP}  \mathbb{S}^{m_2-|\beta|-|\beta_1|,q_1+q_2-|\beta|-|\beta_1|}_{\sigma, \omega}$, we can repeat the above argument a finite number of times until we arrive at identity \eqref{decomident}. The proof is complete. 
\end{proof}
The previous lemma establishes the following corollary, which can be regarded as an extension of \cite[Lemma 4.1]{Mendez2024} to the setting of the class of operators introduced in Definition \ref{pesudodeff}. 
\begin{cor}\label{supportsepareted3}
Let $\sigma \in \mathbb{R}^n,$ $\omega \in \mathbb{R},$ and $m_2,q_1, q_2\in \mathbb{R}$. Let  $g(x,\xi)=g(x)\in \mathbb{S}^{0,q_1}_{\sigma, \omega}$ and $\Psi_a\in \mathrm{OP} \mathbb{S}^{m_2,q_2}_{\sigma, \omega}$.  If $f \in L^{p}(\mathbb{R}^d)$  with $p\in(1,\infty)$ is such that
    \begin{equation*}
    \dist(\supp(f),\supp(g)) \geq \delta > 0, 
    \end{equation*}
then
\begin{equation*}
\|g\Psi_a(f)\|_{L^p}\lesssim \|f\|_{L^p}.
\end{equation*}
\end{cor}
\begin{proof}
We consider a fixed integer $N \geq  \max\{1,m_2-1,q_1+q_2-1\}$. Since $g$ and $f$ have separated support, $f \partial^{\beta}g=0$ for all multi-index $\beta$, then an application of Lemma \ref{lemmadecom} shows
\begin{equation*}
g\Psi_a(f)=R_N(f),
\end{equation*}
where $R_N\in  \mathrm{OP}\mathbb{S}^{m_2-N-1,q_1+q_2-N-1}_{\sigma, \omega}\subset \mathrm{OP}\mathbb{S}^{0,0}_{\sigma, \omega}$. At this point, the desired result is a consequence of the continuity of pseudo-differential operators in Lemma \ref{contiprop}.
\end{proof}
By noticing that if $\Psi_a\in \mathrm{OP}\mathbb{S}^{m,0}_{\sigma,\omega}$, then $J^{-m}\Psi_a \in \mathrm{OP}\mathbb{S}^{0,0}_{\sigma,\omega}$, we can use multiple applications of Lemma \ref{lemmadecom} to obtain the following result.
\begin{cor}\label{remlemmadecom}
	Let $s>0$ and $g(x,\xi)=g(x)\in \mathbb{S}^{0,q_1}_{\sigma, \omega}$. Consider $N \geq\max\{s,q_1\}$ be integer. Then there exist pseudo-differential operators $\Psi_0,\Psi_1,\widetilde{\Psi}_0,\widetilde{\Psi}_1 K_N, \widetilde{K}_N\in \mathrm{OP}\mathbb{S}^{0,0}_{\sigma,\omega}$ such that:
	\begin{equation*}
		\begin{aligned}
			g J^s f=&J^s(g f)+\Psi_0 \Big(\sum_{1\leq k < s}\sum_{|\beta|=k} J^{s-k}(\partial^{\beta}g f)\Big)\\
			&+\Psi_1 \Big(\sum_{s \leq k < N}\sum_{|\beta|=k} \partial^{\beta}g f \Big)+K_N(f),
		\end{aligned}
	\end{equation*}
	and
	\begin{equation*}
		\begin{aligned}
			J^s (g f)=&g J^s f+\widetilde{\Psi}_0\Big(\sum_{1\leq k <s} \sum_{|\beta|=k}\partial^{\beta} g J^{s-k}f\Big)\\
			&+\widetilde{\Psi}_1\Big(\sum_{s\leq k <N} \sum_{|\beta|=k}\partial^{\beta} g f\Big)+\widetilde{K}_N(f),
		\end{aligned}
	\end{equation*}
	where we assume the zero convention for the empty summation\footnote{ We assume the convention  $\sum_{1\leq k <s}(\dots)=0$, if $0< s \leq 1$.}.
\end{cor}
Next, we present some consequences of the theory of pseudo-differential operators. We remark that similar results were obtained in \cite{Mendez2024} for the class of pseudo-differential operator presented in Example \ref{examppseudo}, but here we extend these results to incorporate the weight $\langle \sigma\cdot x+\omega \rangle$.
\begin{lemma}\label{lemmapropaform}
Let $\sigma \in \mathbb{R}^n$ and $\omega \in \mathbb{R}$ be fixed, $1<p<\infty$, and $f\in L^{p}(\mathbb{R}^n)$. Consider $\theta_1,\theta_2 \in C^{\infty}(\mathbb{R}^n)$ such that $0\leq \theta_1, \theta_2 \leq 1$ and satisfying that $\partial^{\alpha}\theta_j \in L^{\infty}(\mathbb{R}^n)$ for all multi-index $\alpha$, $j=1,2$. Additionally, assume that there exists $\delta>0,$ such that
\begin{equation}\label{suppsep}
\dist(\supp(1-\theta_1),\supp(\theta_2))\geq \delta.
\end{equation}
\begin{itemize}[leftmargin=15pt]
\item[(I)] Let $m\in \mathbb{Z}^{+}$ and $r\geq 0$. Assume that 
\begin{equation}\label{hip1}
\langle \sigma\cdot x+\omega \rangle^r \theta_j\in \mathbb{S}^{0,r}_{\sigma,\omega},
\end{equation}
for all $j=1,2$. If $\langle \sigma\cdot x+\omega \rangle^{r} \theta_1 f, \langle \sigma\cdot x+\omega \rangle^{r}\theta_1 \partial^{\beta}f \in L^p(\mathbb{R}^n)$ for all multi-index $\beta$ of order $|\beta|=m$, then 
\begin{equation*}
\begin{aligned}
\|\langle \sigma\cdot x+\omega \rangle^{r}\theta_2 J^m f\|_{L^p}\lesssim & \|\langle \sigma\cdot x+\omega \rangle^{r} \theta_1 f\|_{L^p}\\
&+\sum_{|\beta|=m}\|\langle \sigma\cdot x+\omega \rangle^{r} \theta_1 \partial^{\beta} f\|_{L^p}+\|f\|_{L^p}.
\end{aligned}
\end{equation*}
In other words, $\langle \sigma\cdot x+\omega \rangle^r\theta_2J^{m}f \in L^{p}(\mathbb{R}^n)$. 
\item[(II)] Let $m\in \mathbb{Z}^{+}$ and $r\geq 0$. Assume that $\theta_1$ and $\theta_2$ satisfy the condition \eqref{hip1}. If $\langle \sigma\cdot x+\omega \rangle^r \theta_1 J^m f \in L^{p}(\mathbb{R}^n)$, then 
\begin{equation*}
\|\langle \sigma\cdot x+\omega \rangle^r\theta_2A_mf\|_{L^{p}} \lesssim \|\langle \sigma\cdot x+\omega \rangle^r \theta_1 J^m f\|_{L^p}+\|f\|_{L^p},
\end{equation*}
where $A_m$ is either the operator $\partial^{\beta}$ with $0\leq |\beta|\leq m$, or $J^s$ with $0\leq s \leq m$. That is to say, $\langle \sigma\cdot x+\omega \rangle^r \theta_2 \partial^{\beta} f \in L^{p}(\mathbb{R}^n)$ for each multi-index of order $0\leq |\beta|\leq m$, and $\langle \sigma\cdot x+\omega \rangle^r \theta_2 J^s f \in L^{p}(\mathbb{R}^n)$ for any $0\leq s \leq m$.
\end{itemize}
\end{lemma}
\begin{proof}
We first deal with (I). Since $J^{2m}=(1-\Delta)^{m}$, by writing $J^mf=J^{-m}J^{2m}f$, there exists some constants $c_{\gamma}\geq 0$ for each multi-index $|\gamma|\leq 2m$ such that 
\begin{equation}\label{eqsuppsep0}
\begin{aligned}
\langle \sigma \cdot x +\omega \rangle^{r}\theta_2 J^m f=&\sum_{0\leq |\gamma| \leq m} c_{\gamma} \, \langle \sigma \cdot x +\omega \rangle^{r}\theta_2 J^{-m}\partial^{\gamma}f \\
&+ \sum_{m<|\gamma| \leq 2m} c_{\gamma} \langle \sigma \cdot x +\omega \rangle^{r} \theta_2J^{-m}\partial^{\gamma} f \\
=:&\mathcal{I}_1+\mathcal{I}_2.
\end{aligned}
\end{equation}
Notice that if $r=0$, by using that $f\in L^{p}(\mathbb{R}^n)$ and $ J^{-m}\partial^{\gamma}$ has associated symbol in the class $\mathbb{S}^{0,0}_{\sigma, \omega}$ for each $|\gamma|\leq m$, at once we get 
\begin{equation*}
\begin{aligned}
\|\mathcal{I}_1\|_{L^p} \lesssim \sum_{0\leq |\gamma| \leq m} \|\theta_2\|_{L^{\infty}}\| J^{-m}\partial^{\gamma}f\|_{L^p} \lesssim \|f\|_{L^p}.
\end{aligned}
\end{equation*}
Now if $r>0$, let $N\geq \max\{m,r\}$ integer. For each $0\leq |\gamma| \leq m$, we apply Lemma \ref{lemmadecom} to find pseudo-differential operators $\Psi_k$, $k=1, \dots N$ of order $\mathbb{S}^{0,0}_{\sigma, \omega}$ and $K_N$ of order $\mathbb{S}^{-N-1,r-N-1}_{\sigma,\omega}$ such that
\begin{equation}\label{eqsuppsep0.1}
\begin{aligned}
\langle \sigma \cdot x +\omega \rangle^{r} \theta_2 J^{-m}\partial^{\gamma} f=& J^{-m}\partial^{\gamma}(\langle \sigma \cdot x +\omega \rangle^r \theta_2 f)\\
&+\sum_{k=1}^N \sum_{|\beta|=k}\Psi_{k}\big(\partial^{\beta}(\langle \sigma \cdot x +\omega \rangle^{r} \theta_2) f \big)+K_N(f).
\end{aligned}
\end{equation}
By support considerations, we see that 
\begin{equation}\label{suppppro}
|\partial^{\beta}(\langle \sigma \cdot x+\omega \rangle^{r} \theta_2)|\lesssim |\langle \sigma \cdot x+\omega \rangle^{r} \theta_1|\end{equation}
for any multi-index $\beta$. This remark and identity \eqref{eqsuppsep0.1} yield
\begin{equation}\label{eqsuppsep0.2}
\begin{aligned}
\|\langle \sigma \cdot x+\omega \rangle^{r} \theta_2 J^{-m}\partial^{\gamma} f\|_{L^p}\lesssim  \|\langle \sigma \cdot x+\omega \rangle^{r} \theta_1 f\|_{L^p}+\|f\|_{L^p},
\end{aligned}
\end{equation}
for each $0\leq |\gamma|\leq m$. Summing over all multi-index $\gamma$ of order $0\leq |\gamma|\leq m$, up to some modification of the constant, we obtain the same upper bound in \eqref{eqsuppsep0.2} for the complete sum in $\mathcal{I}_1$.

Let us now deal with $\mathcal{I}_2$. For each  $m<|\gamma|\leq 2m$, we consider a fixed multi-index $\gamma_m$ of order $m$ such that $\gamma=\gamma_m+\gamma_{m}'$, with $0<|\gamma_{m}'|\leq m$. Thus, we write
\begin{equation}\label{eqsuppsep1}
\begin{aligned}
\theta_2J^{-m}\partial^{\gamma} f=\theta_2J^{-m}\partial^{\gamma_{m}'}(\theta_1\partial^{\gamma_{m}} f)+\theta_2J^{-m}\partial^{\gamma_{m}'}((1-\theta_1)\partial^{\gamma_{m}} f).
\end{aligned}
\end{equation}
To estimate the second term on the right-hand side of \eqref{eqsuppsep1}, we write
\begin{equation}\label{eqsuppsep1.1}
\begin{aligned}
&\langle \sigma \cdot x+\omega  \rangle^{r}\theta_2J^{-m}\partial^{\gamma_{m}'}\big((1-\theta_1)\partial^{\gamma_{m}} f\big)\\
&\, \, =\sum_{\gamma_{m,1}+\gamma_{m,2}=\gamma_m} c_{\gamma_{m,1},\gamma_{m_2}} \langle \sigma \cdot x+\omega  \rangle^{r}\theta_2 J^{-m}\partial^{\gamma_{m}'}\partial^{\gamma_{m,1}}\big(\partial^{\gamma_{m,2}}(1-\theta_1) f\big),
\end{aligned}
\end{equation}
for some constants $c_{\gamma_{m,1},\gamma_{m_2}}$. Now, since $\theta_2$ and the derivatives of any order of $1-\theta_1$ have separate supports, we apply Corollary \ref{supportsepareted3} to the identity \eqref{eqsuppsep1.1} to get 
\begin{equation}\label{eqsuppsep1.2}
\begin{aligned}
\|\langle \sigma \cdot x+\omega  \rangle^{r}\theta_2J^{-m}\partial^{\gamma_{m}'}((1-\theta_1)\partial^{\gamma_{m}} f)\|_{L^p} \lesssim \|f\|_{L^p}.
\end{aligned}
\end{equation} 
On the other hand, since $J^{-m}\partial^{\gamma_{m}'}$ determines a symbol in the class $S^{0,0}_{\sigma,\omega}$ for each $0<|\gamma_{m}'|\leq m$, to estimate the first term on the right-hand side of \eqref{eqsuppsep1}, we can repeat the same arguments developed in \eqref{eqsuppsep0.1} and \eqref{eqsuppsep0.2}. This line of arguments mostly depends on the identity in Lemma \ref{lemmadecom}. To avoid falling into repetition, we omit these computations. However, we emphasize that the following bound holds 
\begin{equation}\label{eqsuppsep1.3}
\begin{aligned}
\|\langle \sigma \cdot x+\omega \rangle^{r}\theta_2J^{-m}&\partial^{\gamma_{m}'}(\theta_1\partial^{\gamma_{m}} f)\|_{L^p} \\
&\lesssim \sum_{|\beta|=m}\|\langle \sigma \cdot x+\omega \rangle^{r}\theta_1\partial^\beta f\|_{L^p}+\|f\|_{L^p}.
\end{aligned}
\end{equation}
Collecting \eqref{eqsuppsep1.2} and \eqref{eqsuppsep1.3}, we control the $L^2$-norm of $\langle x\rangle^{\theta}\theta_2J^{-m}\partial^{\gamma} f$ for each multi-index $m<|\gamma|\leq 2m$. Thus, summing over these multi-indexes, we complete the estimate for $\mathcal{I}_2$ and in turn the proof of (I).

The proof of (II) follows by an application of Lemma \ref{lemmadecom}. Indeed, let us consider first the local operator $\partial^{\beta}$. Let $\beta$ be a multi-index of order $|\beta|\leq m$, and $N_1\geq \{m,r\}$ be a fixed integer. By  Lemma \ref{lemmadecom}, there exist $\widetilde{\Psi}_k$, $k=1,\dots,N_1$  pseudo-differential operators in the class $\mathbb{S}^{0,0}_{\sigma, \omega}$, and $\widetilde{K}_{N_1}$ in $\mathbb{S}^{-N_1-1,r-N_1-1}_{\sigma,\omega}$ such that
\begin{equation*}
\begin{aligned}
\langle \sigma\cdot x+\omega \rangle^{r}\theta_2 \partial^{\beta} f=& \langle \sigma \cdot x +\omega\rangle^{r}\theta_2 \partial^{\beta}J^{-m}\big(J^m f \big) \\
=& \partial^{\beta}J^{-m}\big(\langle \sigma\cdot x+\omega \rangle^{r}\theta_2 J^m f \big)+\sum_{k=1}^{N_1}\sum_{|\beta|=k}\widetilde{\Psi}_{k}\big(\partial^{\beta}(\langle \sigma\cdot x+\omega \rangle^{r} \theta_2) J^mf \big)\\
&+\widetilde{K}_{N_1}(J^m f).
\end{aligned}
\end{equation*}
Consequently, property \eqref{suppppro} and the above identity show
\begin{equation*}
\|\langle \sigma\cdot x+\omega \rangle^{r}\theta_2 \partial^{\beta} f\|_{L^p}\lesssim \|\langle \sigma\cdot x+\omega \rangle^{r}\theta_1 J^m f\|_{L^p} +\|f\|_{L^p},
\end{equation*}
which holds true for any multi-index $\beta$ of order $|\beta|\leq m$. By replacing the role of $\partial^{\beta}$ by $J^s$, the same arguments above also yield the desired conclusion for the non-homogeneous derivatives $J^s$ with $0\leq s \leq m$.  The proof of (II) is complete. 
\end{proof}
To control some estimates involving derivatives and weights, we require the following interpolation inequality whose proof follows similar arguments in \cite[Lemma 4]{NahasPonce2009}.
\begin{lemma}\label{InterpLemma}
Let $\sigma \in \mathbb{R}^n$ and $\omega \in \mathbb{R}$ be fixed, $a, b>0$ and $1<p<\infty$. Assume that $J^af \in L^{p}(\mathbb{R}^n)$ and $\langle \sigma \cdot x+\omega \rangle^{b} f \in L^{p}(\mathbb{R}^n)$. Then for any $\theta \in (0,1)$
\begin{equation}\label{eqsuppsep2}
\|\langle \sigma \cdot x+\omega \rangle^{\theta b}J^{(1-\theta)a}f\|_{L^p} \lesssim \|\langle \sigma \cdot x+\omega \rangle^{b}f\|_{L^{p}}^{\theta}\|J^a f\|_{L^p}^{1-\theta}.
\end{equation} 
Moreover,
\begin{equation}\label{eqsuppsep3}
\|J^{(1-\theta) a}\big(\langle \sigma \cdot x+\omega \rangle^{\theta b}f)\|_{L^p} \lesssim \|\langle \sigma \cdot x+\omega \rangle^{b}f\|_{L^{p}}+\|J^a f\|_{L^p}.
\end{equation}
\end{lemma}
\begin{proof}
We define
\begin{equation*}
F(z)=e^{z^2-1}\int_{\mathbb{R}^n} \langle \sigma \cdot x+\omega \rangle^{bz} J^{a(1-z)}f(x)\overline{g(x)} \, dx,
\end{equation*}
with $g\in \mathcal{S}(\mathbb{R}^n)$, and $\|g\|_{L^{p'}}=1$. Then $F$ is continuous in $\{z=\eta+iy\, :\, 0\leq \eta \leq 1\}$ and analytic in its interior. Moreover,
\begin{equation*}
|F(0+iy)|\leq e^{-(y^2+1)}c(y)\|J^a f\|_{L^p},
\end{equation*}
 where $c_1(y)>0$ is a constant of polynomial order in the $y$-variable. Now, since $\langle \sigma \cdot x+\omega \rangle^{b}J^{-iay}$ determines a pseudo-differential operator with symbol the class $\mathbb{S}_{\sigma,\omega}^{0,b}$, the continuity of pseudo-differential operators, Proposition \ref{contiprop}, yields 
\begin{equation*}
|F(1+iy)|\leq e^{-y^2}\|\langle \sigma\cdot x+\omega \rangle^b J^{-iay }f\|_{L^p} \leq e^{-y^2}c_2(y)\|\langle \sigma \cdot x+\omega \rangle^{b}f\|_{L^p},
\end{equation*}
where $c_2(y)$ is a constant of polynomial order in the $y$-variable. Gathering the previous estimates, \eqref{eqsuppsep2} is now a consequence of Hadamard's three lines Theorem.

Let $\theta \in (0,1)$ be fixed. Let $N \geq \max\{ a,b\}$ be fixed. We claim that there exist pseudo-differential operators $\Psi_{\theta}$ and $K$ both with symbols in the class $\mathbb{S}^{0,0}_{\sigma,\omega}$ such that
\begin{equation}\label{eqsuppsep1.4}
\begin{aligned}
J^{(1-\theta)a}(\langle \sigma \cdot x+\omega \rangle^{\theta b} f)=& \Psi_{\theta}\Big(\sum_{0\leq k < N} \langle \sigma \cdot x+\omega \rangle^{\theta b-k} J^{(1-\theta)a-k}f\Big)\\
&+K(f),
\end{aligned}
\end{equation}
The above identity is a consequence of Lemma \ref{lemmadecom}, together with the fact that $J^{-m}\Psi J^m \in \mathrm{OPS}^{m,0}_{\sigma,\omega}$  for any operator $\Psi \in \mathrm{OPS}^{m,0}_{\sigma,\omega}$, and that $\partial_{x}^{\beta}(\langle \sigma \cdot x+\omega \rangle^r)\langle \sigma \cdot x+\omega \rangle^{|\beta|-r}$ determines a symbol in the class $\mathbb{S}^{0,0}_{\sigma,\omega}$. We will assume $\theta b-k>0$, and $(1-\theta)a-k>0$ for all $0\leq k <N$, otherwise, the desired inequality \eqref{eqsuppsep3} holds true.  Hence, we get
\begin{equation*}
\begin{aligned}
\|J^{(1-\theta)a}&(\langle \sigma \cdot x+\omega \rangle^{\theta b} f)\|_{L^p}\\
\lesssim & \sum_{0\leq k < N} \|\langle \sigma \cdot x+\omega \rangle^{\theta b-k} J^{(1-\theta)a-k}f\|_{L^p}+\|f\|_{L^p} \\
\lesssim & \sum_{0\leq k < N} \|\langle \sigma \cdot x+\omega \rangle^{b}f\|_{L^{p}}^{\theta-\frac{k}{b}}\|J^{\frac{((1-\theta)a-k)b}{(1-\theta)b+k}}f\|_{L^p}^{1-\theta+\frac{k}{b}}+\|f\|_{L^p} \\
\lesssim & \|\langle \sigma \cdot x+\omega \rangle^{b}f\|_{L^{p}}+\|J^a f\|_{L^p},
\end{aligned}
\end{equation*}
where we have used \eqref{eqsuppsep2},  $\|J^{\frac{((1-\theta)a-k)b}{(1-\theta)b+k}}f\|_{L^p}\lesssim \|J^{a}f\|_{L^{p}}$ and Young's inequality. The proof is complete.
\end{proof}
Finally, we also recall the following commutator estimate deduced in \cite{KatoPonce1988, Gustavo1990}.
\begin{lemma}\label{conmKP}
	If $s>1$ and $1<p<\infty$, then
	\begin{equation*}
		\|[J^s,f]g\|_{L^p} \lesssim  \|\nabla f\|_{L^{p_1}} \|J^{s-1}g\|_{L^{p_2}}+ \|J^s f\|_{L^{p_3}} \|g\|_{L^{p_4}},
	\end{equation*}
	where $p_1,p_{2},p_{3},p_4\in(1,\infty]$ are such that 
	\begin{equation*}
		\frac{1}{p}=\frac{1}{p_1}+\frac{1}{p_2}=\frac{1}{p_3}+\frac{1}{p_4}.
	\end{equation*}
\end{lemma}
\subsection{Smooth approximations}\label{smoothapprosec}
In this part,  we introduce the cutoff functions to be used along with our arguments. For a more detailed discussion on these approximations, see the results in \cite{IsazaLinaresPonce2015}. 

Given $\epsilon>0$ and $\tau\geq 5\epsilon$, we consider the family of functions
\begin{equation*}
    \chi_{\epsilon,\tau}, \phi_{\epsilon,\tau}, \psi_{\epsilon}\in C^{\infty}(\mathbb{R}),
\end{equation*}
satisfying the following properties:
\begin{itemize}
    \item[(i)] $\chi_{\epsilon,\tau}' \geq 0$,
    \item[(ii)] 
    $\chi_{\epsilon,\tau}(x)=\left\{\begin{aligned}
     0, \, \, x\leq \epsilon \\ 1, \, \, x\geq \tau,
    \end{aligned}
     \right.$,
    \item[(iii)] $\chi_{\epsilon,\tau}'(x) \geq \frac{1}{10(b-\epsilon)}\mathbbm{1}_{[2\epsilon,\tau-2\epsilon]}(x)$,
    \item[(iv)] $\chi_{\epsilon,\tau}(x) \geq \frac{1}{2} \frac{\epsilon}{b-3\epsilon}$, whenever $x\in (3\epsilon,\infty)$,
    \item[(v)] $\supp(\chi'_{\epsilon,\tau})\subset [\epsilon,\tau]$,
    \item[(vi)] $\supp(\phi_{\epsilon,\tau})\subset [\frac{\epsilon}{4},\tau]$,
    \item[(vii)] $\phi_{\epsilon}(x)=\widetilde{\phi}_{\epsilon,\tau}(x)=1, x\in [\frac{\epsilon}{2},\epsilon]$,
    \item[(viii)] $\supp(\psi_{\epsilon})\subset (-\infty,\frac{\epsilon}{2}]$.
    \item[(ix)]  Given $x\in \mathbb{R}$, we have the following partitions of unity
    \begin{equation}\label{phidecomp}
        \chi_{\epsilon,\tau}(x)+\phi_{\epsilon,\tau}(x)+\psi_{\epsilon}(x)=1.
    \end{equation}
\end{itemize}
We will use the following notation
\begin{equation}\label{defiweight}
\begin{aligned}
\chi_{r,\epsilon,\tau,\sigma}(x,t)=\langle \sigma \cdot x+\nu t+\kappa \rangle^{r}\chi_{\epsilon,\tau}(\sigma \cdot x+\nu t+\kappa), \, \, x\in \mathbb{R}^n, \, t\geq 0. 
\end{aligned}
\end{equation}
We will also use the notation
\begin{equation*}
\begin{aligned}
\chi_{\epsilon,\tau,\sigma}(x,t)=\chi_{0,\epsilon,\tau,\sigma}(x,t)=\chi_{\epsilon,\tau}(\sigma \cdot x+\nu t+\kappa), \, \, x\in \mathbb{R}^n, \, t\geq 0. 
\end{aligned}
\end{equation*}
Similarly, we set $\phi_{\epsilon,\tau,\sigma}(x)$ and $\psi_{\epsilon,\sigma}(x)$. Notice that the above functions are an extension to $n$-dimensions of the one dimensional functions $\chi_{\epsilon,\tau}$, $\phi_{\epsilon,\tau}$, and $\psi_{\epsilon}$. For example, given a multi-index $\beta$, we have 
\begin{equation}\label{notationderchi}
    \partial^{\beta}(\chi_{\epsilon,\tau,\sigma}(x,t))=\sigma^{\beta}\chi_{\epsilon,\tau,\sigma}^{(|\beta|)}(x,t),
\end{equation}
where $\chi_{\epsilon,\tau}^{(j)}(x)$ denotes the $j$th derivative of $\chi_{\epsilon,\tau}(x)$, $x\in \mathbb{R}$.
\begin{remark}
Since any derivative of $\chi_{\epsilon,\tau}$ is compactly supported, we have that $\chi_{r,\epsilon,\tau,\sigma} \in \mathbb{S}^{0,r}_{\sigma, \nu t+\kappa}$. Similarly, $\phi_{\epsilon,\tau,\sigma},\psi_{\epsilon,\sigma}\in  \mathbb{S}^{0,0}_{\sigma, \nu t+\kappa}$.
\end{remark}
\begin{remark}
Restricted to the weighted functions $\chi_{r,\epsilon,\tau,\sigma}$ , $\phi_{\epsilon,\tau,\sigma}$ and $\psi_{\epsilon,\sigma}$,  the constants in Lemmas \ref{lemmadecom}, \ref{lemmapropaform}, \ref{InterpLemma}, and Corollaries \ref{supportsepareted3}, \ref{remlemmadecom} can be taken independent of $\nu t+\kappa$, provided that $|t|\leq M$. For example, this can be motivated by Definition \ref{pesudodeff}, and the fact that the constants in it for the functions $\chi_{r,\epsilon,\tau,\sigma}$, $\phi_{\epsilon,\tau,\sigma}$ and $\psi_{\epsilon,\sigma}$ can be taken independent of the translation $\nu t+\kappa$, $t\geq 0$. 
\end{remark}
\section{Propagation of polynomial decay: Proof of Theorem \ref{mainTHM}}\label{sectionMain}
By translation, we may set $\kappa=0$. The proof of Theorem \ref{mainTHM} follows from weighted energy estimates in the equation in \eqref{ZK}. By approximation by smooth functions and taking the limit in our arguments, we will consider solutions of \eqref{ZK} with sufficient regularity and decay to perform all the estimates detailed below. This assumption can be justified at the end of Section 3 in  \cite{IsazaLinaresPonce2015}. We remark that such an approximation argument is granted thanks to our local well-posedness theory in weighted spaces deduced in Theorem \ref{wellposweighted}.

Since our argument depends on several energy estimates, we will state a general identity for an arbitrary differential operator $\mathcal{B}^s$, which will be specified in each step of the proof of Theorem \ref{mainTHM}. Hence, by applying $\mathcal{B}^s$ to the equation in \eqref{ZK}, multiplying then by $(\mathcal{B}^s u )\chi_{r,\epsilon,\tau,\sigma}^2$ and integrating in space the resulting expression, we get
\begin{equation}\label{differineq}
\begin{aligned}
\frac{1}{2}\frac{d}{dt}\int &(\mathcal{B}^su(x,t))^2 \chi_{r,\epsilon,\tau,\sigma}^2(x,t)\, \mathrm{d}x- \frac{1}{2}\underbrace{ \int (\mathcal{B}^s u(x,t))^2\frac{\partial}{\partial t}\chi_{r,\epsilon,\tau,\sigma}^2(x,t)\, \mathrm{d}x}_{A_1} \\
&+\underbrace{\int \partial_{x_1}\Delta \mathcal{B}^s u(x,t) \mathcal{B}^su(x,t)\chi_{r,\epsilon,\tau,\sigma}^2(x,t)\, \mathrm{d}x}_{A_2} \\
&+\underbrace{\int \mathcal{B}^s\big(u\partial_{x_1}u(x,t)) \mathcal{B}^s  u(x,t)\chi_{r,\epsilon,\tau,\sigma}^2(x,t)\, \mathrm{d}x}_{A_3}\\
&=0.
\end{aligned}
\end{equation}
Considering that the estimates for $A_1$ and $A_2$ follow by rather general considerations independent of $\mathcal{B}^s$, we summarize them in the following lemma.
\begin{lemma}\label{linearEstlemma1} 
Let $n\geq 2$,  $A_1$ and $A_2$ defined in \eqref{differineq}, then
\begin{equation}\label{estA1}
\begin{aligned}
|A_1|\lesssim &\int (\mathcal{B}^su(x,t))^2\chi_{r,\epsilon,\tau,\sigma}^2(x,t)\, \mathrm{d}x \\
&+\Big|\int (\mathcal{B}^su(x,t))^2 \langle \sigma \cdot x+\nu t \rangle^{2r}\chi_{\epsilon,\tau,\sigma}\chi_{\epsilon,\tau,\sigma}'(x,t)\, \mathrm{d}x\Big|.
\end{aligned}
\end{equation}
Moreover, there exists some constant $c_{\sigma, r}>0$ such that
\begin{equation}\label{estA2}
\begin{aligned}
A_2= & c_{\sigma,r}\int |\nabla \mathcal{B}^s u|^2 \chi_{\epsilon,\tau,\sigma}\big(\langle \sigma \cdot x+\nu t\rangle^{2r-2}(\sigma \cdot x+\nu t)\chi_{\epsilon,\tau,\sigma}\\
&\hspace{2cm}+\langle\sigma \cdot x+\nu t\rangle^{2r}\chi_{\epsilon,\tau,\sigma}' \big) \, \mathrm{d}x \\
&+R(\mathcal{B}^su,t),
\end{aligned}
\end{equation}
where 
\begin{equation}\label{estR}
\begin{aligned}
|R(\mathcal{B}^su,t)| \lesssim & \int (\mathcal{B}^su(x,t))^2\chi_{r,\epsilon,\tau,\sigma}^2(x,t)\,\mathrm{d}x \\
&+ \int (\mathcal{B}^s u(x,t))^2 \big(|\chi_{\epsilon,\tau,\sigma}\chi_{\epsilon,\tau,\sigma}'|+ |\chi_{\epsilon,\tau,\sigma}'|^2 +|\chi_{\epsilon,\tau,\sigma}\chi_{\epsilon,\tau,\sigma}''|\\
&\hspace{2cm}+|\chi_{\epsilon,\tau,\sigma}'\chi_{\epsilon,\tau,\sigma}''|+|\chi_{\epsilon,\tau,\sigma}\chi_{\epsilon,\tau,\sigma}'''|\big)\,\mathrm{d}x.
\end{aligned}
\end{equation}
\end{lemma}
\begin{proof}
The estimate \eqref{estA1} is obtained by computing $\partial_{t}\left(\chi_{r,\epsilon,\tau,\sigma}^2(x,t)\right)$. To deduce \eqref{estA2}, we integrate by parts to write
\begin{equation*}
\begin{aligned}
A_2=&\frac{1}{2}\int |\nabla \mathcal{B}^s u|^2 \partial_{x_1}(\chi_{r,\epsilon,\tau,\sigma}^2) \, \mathrm{d}x +\int \partial_{x_1}\mathcal{B}^s u\nabla \mathcal{B}^s u \cdot \nabla(\chi_{r,\epsilon,\tau,\sigma}^2) \, \mathrm{d}x \\
&-\frac{1}{2}\int (\mathcal{B}^su)^2 \partial_{x_1}\Delta (\chi_{r,\epsilon,\tau,\sigma}^2) \, \mathrm{d}x  \\
=:& A_{2,1}+A_{2,2}+R(\mathcal{B}^s u,t).
\end{aligned}
\end{equation*}
We notice that
\begin{equation*}
\begin{aligned}
 A_{2,1}+ A_{2,2}=&\sum_{m=1}^n \sigma_1 \big(\int (\partial_{x_m} \mathcal{B}^s u)^2 \langle \sigma\cdot x+\nu t \rangle^{2r}\chi_{\epsilon,\tau,\sigma}\chi_{\epsilon,\tau,\sigma}' \, \mathrm{d}x \\
&\hspace{0.4cm}+r\int (\partial_{x_m} \mathcal{B}^s u)^2 \langle \sigma\cdot x+\nu t \rangle^{2r-2}(\sigma\cdot x+\nu t)\chi_{\epsilon,\tau,\sigma}^2 \, \mathrm{d}x\big) \\
&+2\sum_{m=1}^n\sigma_{m}\big(\int \partial_{x_1}\mathcal{B}^s u\partial_{x_m}\mathcal{B}^s u \langle \sigma\cdot x+\nu t \rangle^{2r}\chi_{\epsilon,\tau,\sigma}\chi_{\epsilon,\tau,\sigma}' \, \mathrm{d}x \\
& \hspace{0.4cm} +r\int \partial_{x_1}\mathcal{B}^s u\partial_{x_m} \mathcal{B}^s u \langle \sigma\cdot x+\nu t \rangle^{2r-2}(\sigma\cdot x+\nu t)\chi_{\epsilon,\tau,\sigma}^2 \, \mathrm{d}x \big)\\
=& \int \langle \nabla \mathcal{B}^s u, \nabla \mathcal{B}^s u \rangle_{M} \langle \sigma\cdot x+\nu t \rangle^{2r}\chi_{\epsilon,\tau,\sigma}\chi_{\epsilon,\tau,\sigma}' \, \mathrm{d}x\\
&+ r\int \langle \nabla \mathcal{B}^s u, \nabla \mathcal{B}^s u \rangle_M \langle \sigma\cdot x+\nu t \rangle^{2r-2}(\sigma\cdot x+\nu t)\chi_{\epsilon,\tau,\sigma}^2 \, \mathrm{d}x,
\end{aligned}
\end{equation*}
where for $x,y \in \mathbb{R}^n$, the product $\langle x,y \rangle_M=x^{T}My$ is defined via the symmetric matrix  
\begin{equation*}
		M=\begin{pmatrix}
			3\sigma_{1} & \sigma_{2}&\sigma_{3}&\cdots &\sigma_{n}\\
			\sigma_{2} & \sigma_{1}&0&\cdots &0\\
			\sigma_{3}&0 &\sigma_{1}&\cdots&0\\
			\vdots&\vdots&\vdots&\ddots&\vdots \\
			\sigma_{n}& 0&\vdots&\cdots &\sigma_{1}
		\end{pmatrix}.
	\end{equation*} 
Since $\sigma_1>0$ and $\sqrt{3}\sigma_1 > \sqrt{\sigma_2^2+\dots+ \sigma_n^2}$, it turns out that $M$ is positive definite, thus, $\langle x,x \rangle_M \sim |x|^2$ for all $x\in \mathbb{R}^n$ (this was observed before in \cite{LinaresPonceZK2018,Mendez2024}). Then
\begin{equation*}
\begin{aligned}
A_{2,1}+A_{2,2}\sim_{\sigma,r} \int |\nabla \mathcal{B}^s u|^2 \chi_{\epsilon,\tau,\sigma}\big(&\langle \sigma \cdot x+\nu t\rangle^{2r-2}(\sigma \cdot x+\nu t)\chi_{\epsilon,\tau,\sigma}\\
&+\langle\sigma \cdot x+\nu t\rangle^{2r}\chi_{\epsilon,\tau,\sigma}' \big) \, \mathrm{d}x.
\end{aligned}
\end{equation*}
On the other hand, using the identity
\begin{equation*}
\begin{aligned}
\Delta(\chi^2_{r,\epsilon,\tau,\sigma})=&\Delta (\langle\sigma\cdot x+\nu t\rangle^{2r})\chi_{\epsilon,\tau,\sigma}^2(x,t)+2\nabla(\langle\sigma \cdot x+\nu t\rangle^{2r})\cdot\nabla(\chi_{\epsilon,\tau,\sigma}^2)(x,t) \\
& +\langle\sigma \cdot x+\nu t\rangle^{2r}\Delta(\chi_{\epsilon,\tau,\sigma}^2)(x,t),
\end{aligned}
\end{equation*}
we deduce
\begin{equation*}
\begin{aligned}
|\partial_{x_1}\Delta(\chi^2_{r,\epsilon,\tau,\sigma})|\lesssim & \langle\sigma \cdot x+\nu t\rangle^{2r-3}\chi_{\epsilon,\tau,\sigma}^2(x,t)+\langle\sigma \cdot x+\nu t \rangle^{2r-2}\chi_{\epsilon,\tau,\sigma}\chi_{\epsilon,\tau,\sigma}'(x,t) \\
&+\langle\sigma \cdot x+\nu t\rangle^{2r-1}\big((\chi_{\epsilon,\tau,\sigma}')^2+\chi_{\epsilon,\tau,\sigma}|\chi_{\epsilon,\tau,\sigma}''|\big)(x,t)\\
&+\langle\sigma \cdot x+\nu t\rangle^{2r}\big(\chi_{\epsilon,\tau,\sigma}'|\chi_{\epsilon,\tau,\sigma}''|+\chi_{\epsilon,\tau,\sigma}|\chi_{\epsilon,\tau,\sigma}'''|\big)(x,t).
\end{aligned}
\end{equation*}
Consequently, the above inequality and the fact that any derivative of $\chi_{\epsilon,\tau, \sigma}$ is compactly supported yield \eqref{estR}.
\end{proof}
Next, we present the key estimates for $A_3(t)$ in \eqref{differineq}.
\begin{lemma}\label{NonlinearEstlemma1}
Let $\epsilon>0,$ $\tau \geq 5 \epsilon$. Assume $\mathcal{B}^s=\partial^{\beta}$ for some multi-index $|\beta|\leq 2$, then
\begin{equation}\label{estnonT}
\begin{aligned}
|A_3|\lesssim & \big(\|u\|_{L^{\infty}}+\|\nabla u\|_{L^{\infty}}\big)\sum_{|\gamma|=|\beta|}\Big(\int  (\partial^{\gamma} u(x,t))^2 \chi_{r,\epsilon,\tau,\sigma}^2 \, \mathrm{d}x\\
&+\int  (\partial^{\gamma} u(x,t))^2 \chi_{r,\epsilon,\tau,\sigma}\chi_{r,\epsilon,\tau,\sigma}' \, \mathrm{d}x\Big).
\end{aligned}
\end{equation}
Moreover, assume $\mathcal{B}^s=J^s$, then
\begin{equation}\label{estnonT2}
|A_3|\lesssim |\widetilde{A}_3|\|\chi_{r,\epsilon,\tau,\sigma}J^s u \|_{L^2},
\end{equation}
where $\widetilde{A}_3$ is bounded according to the following cases:
\begin{itemize}[leftmargin=20pt]
    \item[(I)] Assume $3\leq s\leq \lfloor\frac{n}{2}+1\rfloor$ with $n\geq 4$, and $u\in H^{(\frac{n}{2}+1)^{+}}(\mathbb{R}^n)$, then
\begin{equation}\label{extranonlEsteq}
\begin{aligned}
|\widetilde{A}_3|\lesssim & \Big(\|u\|_{L^{\infty}}+\|\nabla u\|_{L^{\infty}} \Big)\Big(\|u\|_{L^2}+\|\chi_{r,\epsilon',\tau',\sigma} u\|_{L^2}\Big)\\
&+\Big(\|J^{(\frac{n}{2}+1)^{+}}u\|_{L^2}+\|u\|_{L^{\infty}}+\|\partial_{x_1}u\|_{L^{\infty}}\Big)\\
&\, \times\Big(\|\chi_{r,\epsilon,\tau,\sigma}J^s u\|_{L^2}+\|\chi_{r-1,\epsilon',\tau',\sigma}J^s u\|_{L^2}+\|\chi_{r,\epsilon',\tau',\sigma} u\|_{L^2}+\|u\|_{L^2} \Big).
\end{aligned}
\end{equation}    
\item[(II)] Assume $s\geq 3$ with $s>\frac{n}{2}+1$. Then
    \begin{equation}\label{nonlEsteq}
\begin{aligned}
|\widetilde{A}_3|\lesssim & \Big(\|u\|_{L^{\infty}}+\|\nabla u\|_{L^{\infty}} \Big)\Big(\|u\|_{L^2}+\|\chi_{r,\epsilon',\tau',\sigma} u\|_{L^2}\Big)\\
&+\Big(\|\chi_{r,\epsilon,\tau,\sigma}J^s u\|_{L^2}+\|\chi_{r-1,\epsilon',\tau',\sigma}J^s u\|_{L^2}+\|\chi_{r,\epsilon',\tau',\sigma} u\|_{L^2}+\|u\|_{L^2} \Big)\\
& \hspace{0.5cm} \times \Big(\|\chi_{\epsilon',\tau',\sigma}J^s u\|_{L^2}+\|\widetilde{\phi}_{\epsilon,\tau,\sigma}J^s u\|_{L^2}+\|u\|_{L^{\infty}}+\|u\|_{L^2}\Big).
\end{aligned}
\end{equation}
\end{itemize}
Where $\epsilon'>0$, $\tau'\geq 5\epsilon'$, $\tau' \leq \epsilon$ and $\widetilde{\phi}_{\epsilon,\tau,\sigma}(x)=\widetilde{\phi}_{\epsilon,\tau}( \sigma\cdot x+vt)$, with $\widetilde{\phi}_{\epsilon,\tau} \in C^{\infty}_c(\mathbb{R})$ be such that $\widetilde{\phi}_{\epsilon,\tau}\phi_{\epsilon,\tau}=\phi_{\epsilon,\tau}$, and $\phi_{\epsilon,\tau, \sigma}$ be defined in Subsection \ref{smoothapprosec}.
\end{lemma}
\begin{remark}\label{remarkl1condt}
In the cases $n=2,3$, the results of Lemma \ref{NonlinearEstlemma1} parts \eqref{estnonT} and (II), in conjunction with the argument developed below, show that Theorem \ref{mainTHM} can be extended to initial condition $u_0\in H^{s'}(\mathbb{R}^n)$ with $s'\leq \frac{n}{2}+1$ provided that there exists a local theory for which solutions of \eqref{ZK} are in the class \eqref{l1cond}. This is not the case when the dimension $n\geq 4$, where we need solutions of \eqref{ZK} in the space $H^{(\frac{n}{2}+1)^{+}}(\mathbb{R}^n)$ (this is stated in Lemma \ref{NonlinearEstlemma1} (I)).
    
The reason for such difference is that when applying the Kato-Ponce inequality, Lemma \ref{conmKP}, we have some difficulties keeping under control the regularity needed to bound factors such as $\|\nabla (\chi_{r,\epsilon,\tau,\sigma}u)\|_{L^p}$, $p\neq 2$. This can be seen for instance in \eqref{extraeqdecomp} and \eqref{eqdecom1} below, where we use different Sobolev embedding to control $\nabla (\chi_{r,\epsilon,\tau,\sigma}u)$. Nevertheless, we believe that this is a technical limitation that can be overcome with a more detailed expansion of Lemma \ref{conmKP} when $n\geq 4$. However, since our results for $n=1,2,3$ extend to regularity $s \leq \frac{n}{2}+1$ provided \eqref{l1cond}, we decided not to pursue the more involved lower regularity case when $n\geq 4$.
\end{remark}
\begin{proof}[Proof of Lemma \ref{NonlinearEstlemma1}]
Estimate \eqref{estnonT} can be easily deduced via integration by parts. Let us deal with the deduction of \eqref{estnonT2}. Let $N \geq \max\{r,s\}$, by Corollary \ref{remlemmadecom}, we write
\begin{equation}\label{eqdecom0}
\begin{aligned}
\chi_{r,\epsilon,\tau,\sigma} & J^s(u\partial_{x_1}u)\\
=& J^s\big(\chi_{r,\epsilon,\tau,\sigma}u\partial_{x_1}u \big)+\Psi_0\Big(\sum_{1\leq k<s} \sum_{|\beta|=k} J^{s-k}(\partial^{\beta}(\chi_{r,\epsilon,\tau,\sigma})u\partial_{x_1}u)\Big)\\
&+\Psi_1\Big(\sum_{s\leq k<r} \sum_{|\beta|=k} \partial^{\beta}(\chi_{r,\epsilon,\tau,\sigma})u\partial_{x_1}u\Big)+K_N(u\partial_x u),
\end{aligned}
\end{equation}
for some pseudo-differential operators $\Psi_0, \Psi_1$ and $K_N$ in $\mathrm{OP}\mathbb{S}^{0,0}_{\sigma, \omega}$. Above we have also used the zero convention for the empty summation. In view that
\begin{equation}\label{nonlEsteq1}
\|K_N(u\partial_x u)\|_{L^2}\lesssim \|\nabla u\|_{L^{\infty}}\| u\|_{L^{2}},
\end{equation}
and 
\begin{equation}\label{nonlEsteq2}
\begin{aligned}
\|\Psi_1\Big(\sum_{s\leq k<r} \sum_{|\beta|=k} \partial^{\beta}(\chi_{r,\epsilon,\tau,\sigma})u\partial_{x_1}u\Big)\|_{L^2}\lesssim \|\partial_{x_1}u\|_{L^{\infty}}\|\chi_{r,\epsilon',\tau',\sigma} u\|_{L^2},
\end{aligned}
\end{equation}
where $\epsilon'>0$, $\tau'\geq 5\epsilon'$, $\tau'\leq \epsilon$, we are reduced to control the first and second term on the right-hand side of \eqref{eqdecom0}. For the former term, by using the decomposition \eqref{phidecomp}, that is, $\chi_{\epsilon,\tau,\sigma}+\phi_{\epsilon,\tau,\sigma}+\psi_{\epsilon,\sigma}=1$, the fact that $\chi_{\epsilon,\tau,\sigma}$ and $\psi_{\epsilon,\sigma}$ have separated supports, we write
\begin{equation}\label{eqdecom}
\begin{aligned}
J^s\big(\chi_{r,\epsilon,\tau,\sigma}u\partial_{x_1}u \big)=&J^s\big(\chi_{r,\epsilon,\tau,\sigma}u\partial_{x_1}(\chi_{\epsilon,\tau,\sigma} u )\big)+J^s\big(\chi_{r,\epsilon,\tau,\sigma}u\partial_{x_1}(\phi_{\epsilon,\tau,\sigma} u )\big) \\
=& [J^s,\chi_{r,\epsilon,\tau,\sigma}u]\partial_{x_1}(\chi_{\epsilon,\tau,\sigma} u)+ [J^s,\chi_{r,\epsilon,\tau,\sigma}u]\partial_{x_1}(\phi_{\epsilon,\tau,\sigma}u)\\
&+\chi_{r,\epsilon,\tau,\sigma}u \partial_{x_1}J^s(\chi_{\epsilon,\tau,\sigma} u)+\chi_{r,\epsilon,\tau,\sigma}u \partial_{x_1}J^s(\phi_{\epsilon,\tau,\sigma} u)\\
=& [J^s,\chi_{r,\epsilon,\tau,\sigma}u]\partial_{x_1}(\chi_{\epsilon,\tau,\sigma} u)+ [J^s,\chi_{r,\epsilon,\tau,\sigma}u]\partial_{x_1}(\phi_{\epsilon,\tau,\sigma} u)\\
&-\chi_{r,\epsilon,\tau,\sigma}u \partial_{x_1}J^s(\psi_{\epsilon,\sigma} u)+\chi_{r,\epsilon,\tau,\sigma}u \partial_{x_1}J^su.
\end{aligned}
\end{equation}
The last term on the right-hand side of the above equality is estimated by going back to the integral defining $A_{3}$ and integrating by parts. Now, we consider the cases (I) and (II), i.e., $3\leq s\leq \lfloor \frac{n}{2}+1\rfloor$ with $n\geq 4$, and $s>\frac{n}{2}+1$ with $s\geq 3$.

\underline{Proof of (I): Case $3\leq s\leq \lfloor\frac{n}{2}+1\rfloor$ with $n\geq 4$}. Since $u\in H^{(\frac{n}{2}+1)^{+}}(\mathbb{R}^n)$, let $\delta>0$ be small enough such that $\delta<s-1$, and $u\in H^{\frac{n}{2}+1+\delta}(\mathbb{R}^n)$. We apply Lemma \ref{conmKP}, and Hardy-Littlewood-Sobolev inequality to get
\begin{equation}\label{extraeqdecomp}
\begin{aligned}
\|[J^s,\chi_{r,\epsilon,\tau,\sigma}u]\partial_{x_1}(\chi_{\epsilon,\tau,\sigma}u) \|_{L^2}\lesssim & \|\nabla(\chi_{r,\epsilon,\tau,\sigma}u)\|_{L^{p_1}}\|J^{s-1}\partial_{x_1}( \chi_{\epsilon,\tau,\sigma} u)\|_{L^{p_2}}\\
&+\|J^s(\chi_{r,\epsilon,\tau,\sigma}u)\|_{L^{2}}\|\partial_{x_1}(\chi_{\epsilon,\tau,\sigma} u)\|_{L^{\infty}}\\
\lesssim &\big(\|J^{\frac{n}{2}+1+\delta}(\chi_{\epsilon,\tau,\sigma} u)\|_{L^{2}}+\|\partial_{x_1}(\chi_{\epsilon,\tau,\sigma} u)\|_{L^{\infty}}\big)\\
& \times\|J^s(\chi_{r,\epsilon,\tau,\sigma}u)\|_{L^{2}},
\end{aligned}
\end{equation}
where
\begin{equation*}
\left\{\begin{aligned}
&\frac{1}{p_1}+\frac{1}{p_2}=\frac{1}{2}, \qquad \frac{1}{p_1}=\frac{1}{2}-\frac{(s-1-\delta)}{n}, \\
& \frac{1}{p_2}=\frac{1}{2}-\frac{(\frac{n}{2}+1+\delta-s)}{n}.
\end{aligned}\right.
\end{equation*}
By Corollary \ref{remlemmadecom} and \eqref{extraeqdecomp}, we get
\begin{equation}\label{extraeqdecomp1}
\begin{aligned}
\|[J^s,\chi_{r,\epsilon,\tau,\sigma}u]\partial_{x_1}(\chi_{\epsilon,\tau,\sigma}u) \|_{L^2}\lesssim &\big(\|J^{\frac{n}{2}+1+\delta}u\|_{L^{2}}+\|u\|_{L^{\infty}}+\|\partial_{x_1}u\|_{L^{\infty}}\big)\\
& \times\|J^s(\chi_{r,\epsilon,\tau,\sigma}u)\|_{L^{2}}.
\end{aligned}
\end{equation}
A further application of Corollary \ref{remlemmadecom} shows
\begin{equation}\label{extranonlEsteq2}
\begin{aligned}
\|J^{s}(\chi_{r,\epsilon,\tau,\sigma}u)\|_{L^{2}} 
\lesssim & \|\chi_{r,\epsilon,\tau,\sigma}J^{s} u\|_{L^2}+ \sum_{1\leq k<s} \|\chi_{r-k,\epsilon',\tau',\sigma}J^{s-k} u\|_{L^2}\\
&+\|\chi_{r,\epsilon',\tau',\sigma}u\|_{L^2}+\|u\|_{L^2},
\end{aligned}
\end{equation}
where $\epsilon'>0$, $\tau'\geq 5\epsilon'$, $\tau' \leq \epsilon$. Summarizing, we have deduced
\begin{equation}\label{extraeqdecomp1.1}
\begin{aligned}
\|[J^s,\chi_{r,\epsilon,\tau,\sigma}u]&\partial_{x_1}(\chi_{\epsilon,\tau,\sigma}u) \|_{L^2}\\
\lesssim &\big(\|J^{\frac{n}{2}+1+\delta}u\|_{L^{2}}+\|u\|_{L^{\infty}}+\|\partial_{x_1}u\|_{L^{\infty}}\big)\\
& \times\Big(\|\chi_{r,\epsilon,\tau,\sigma}J^{s} u\|_{L^2}+ \sum_{1\leq k<s} \|\chi_{r-k,\epsilon',\tau',\sigma}J^{s-k} u\|_{L^2}\\
&+\|\chi_{r,\epsilon',\tau',\sigma}u\|_{L^2}+\|u\|_{L^2}\Big).
\end{aligned}
\end{equation}
Similarly, it follows
\begin{equation}\label{extraeqdecomp2}
\begin{aligned}
\|[J^s,\chi_{r,\epsilon,\tau,\sigma}u]&\partial_{x_1}(\phi_{\epsilon,\tau,\sigma}u) \|_{L^2}\\
\lesssim &\big(\|J^{\frac{n}{2}+1+\delta}u\|_{L^{2}}+\|u\|_{L^{\infty}}+\|\partial_{x_1}u\|_{L^{\infty}}\big)\\
& \times\Big(\|\chi_{r,\epsilon,\tau,\sigma}J^{s} u\|_{L^2}+ \sum_{1\leq k<s} \|\chi_{r-k,\epsilon',\tau',\sigma}J^{s-k} u\|_{L^2}\\
&+\|\chi_{r,\epsilon',\tau',\sigma}u\|_{L^2}+\|u\|_{L^2}\Big).
\end{aligned}
\end{equation}
To estimate the remaining term in \eqref{eqdecom}, since $\chi_{r,\epsilon,\tau,\sigma}$ and $\psi_{\epsilon,\sigma}$ have separated supports, we use Corollary \ref{supportsepareted3} to get
\begin{equation}\label{extraeqdecomp3}
\begin{aligned}
\|\chi_{r,\epsilon,\tau,\sigma} u\partial_{x_1}J^s(\psi_{\epsilon,\sigma} u)\|_{L^2}&\lesssim \|u\|_{L^{\infty}}\|\chi_{r,\epsilon,\tau,\sigma}\partial_{x_1}J^s(\psi_{\epsilon,\sigma} u)\|_{L^2}\\
&\lesssim \|u\|_{L^{\infty}}\|u\|_{L^2}.
\end{aligned}
\end{equation}   
This completes the study of \eqref{eqdecom}. Finally, we deal with the second term on the right-hand side of \eqref{eqdecom0}. For fixed $1\leq k<s$, and $|\beta|=k$, we write
\begin{equation}\label{nonlEsteq7}
\begin{aligned}
 J^{s-k}(\partial^{\beta}(\chi_{r,\epsilon,\tau,\sigma})u\partial_{x_1}u)=&\frac{1}{2}\partial_{x_1}J^{s-k}(\partial^{\beta}(\chi_{r,\epsilon,\tau,\sigma})u^2)\\
 &-\frac{1}{2}J^{s-k}(\partial_{x_1}\partial^{\beta}(\chi_{r,\epsilon,\tau,\sigma})u^2)\\
 =&\frac{1}{2}\partial_{x_1}J^{-1}\Big(J^{s+1-k}(\partial^{\beta}(\chi_{r,\epsilon,\tau,\sigma})u^2)\Big)\\
 &-\frac{1}{2}J^{s-k}(\partial_{x_1}\partial^{\beta}(\chi_{r,\epsilon,\tau,\sigma})u^2).
\end{aligned}
\end{equation}
Given that $k\geq 1$, the above identity yields
\begin{equation*}
\begin{aligned}
\|J^{s-k}(\partial^{\beta}(\chi_{r,\epsilon,\tau,\sigma})u\partial_{x_1}u)\|_{L^2}\lesssim \|J^s(\partial^{\beta}(\chi_{r,\epsilon,\tau,\sigma})u^2)\|_{L^2}+\|J^s(\partial_{x_1}\partial^{\beta}(\chi_{r,\epsilon,\tau,\sigma})u^2)\|_{L^2}.
\end{aligned}
\end{equation*}
To complete the estimate of the above inequality, we write
\begin{equation*}
\begin{aligned}
J^s(\partial^{\beta}(\chi_{r,\epsilon,\tau,\sigma})u^2)=[J^s,\partial^{\beta}(\chi_{r,\epsilon,\tau,\sigma})u]u+\partial^{\beta}(\chi_{r,\epsilon,\tau,\sigma})u J^s u.
\end{aligned}    
\end{equation*}
Since for each multi-index $\gamma$, $|\partial^{\gamma}(\chi_{r,\epsilon,\tau,\sigma})|\lesssim \chi_{r-|\gamma|,\epsilon',\tau',\sigma}$, where $\epsilon'>0$, $\tau'\geq 5\epsilon'$, $\tau'\leq \epsilon$, we use the same ideas in \eqref{extraeqdecomp}-\eqref{extranonlEsteq2}, and the identity above to deduce
\begin{equation}\label{extraeqdecomp4}
\begin{aligned}
\|J^s (\partial^{\beta}&(\chi_{r,\epsilon,\tau,\sigma})u^2)\|_{L^2}\\
\lesssim & \big( \|J^{\frac{n}{2}+1+\delta}u\|_{L^2}+\|u\|_{L^{\infty}}+\|\partial_{x_1}u\|_{L^{\infty}}\big)\|J^s(\partial^{\beta}(\chi_{r,\epsilon,\tau,\sigma}) u)\|_{L^2}\\
&+\|u\|_{L^{\infty}}\|\chi_{r-k,\epsilon',\tau',\sigma}J^s u\|_{L^2}\\
\lesssim & \big( \|J^{\frac{n}{2}+1+\delta}u\|_{L^2}+\|u\|_{L^{\infty}}+\|\partial_{x_1}u\|_{L^{\infty}}\big)\Big(\|\chi_{r-k,\epsilon,\tau,\sigma}J^{s} u\|_{L^2}\\
&+ \sum_{1\leq k'<s} \|\chi_{r-k-k',\epsilon',\tau',\sigma}J^{s-k'} u\|_{L^2}+\|\chi_{r,\epsilon',\tau',\sigma}u\|_{L^2}+\|u\|_{L^2}\Big)\\
&+\|u\|_{L^{\infty}}\|\chi_{r-k,\epsilon',\tau',\sigma}J^s u\|_{L^2}.
\end{aligned}    
\end{equation}
By the same argument just described, changing $\partial^{\beta}(\chi_{r,\epsilon,\tau,\sigma})$ by $\partial_{x_1}\partial^{\beta}(\chi_{r,\epsilon,\tau,\sigma})$, we deduce
\begin{equation}\label{extraeqdecomp5}
\begin{aligned}
\|J^s(\partial_{x_1}\partial^{\beta}&(\chi_{r,\epsilon,\tau,\sigma})u^2)\|_{L^2}\\
\lesssim & \big( \|J^{\frac{n}{2}+1+\delta}u\|_{L^2}+\|u\|_{L^{\infty}}+\|\partial_{x_1}u\|_{L^{\infty}}\big)\Big(\|\chi_{r-k-1,\epsilon,\tau,\sigma}J^{s} u\|_{L^2}\\
&+ \sum_{1\leq k'<s} \|\chi_{r-k-1-k',\epsilon',\tau',\sigma}J^{s-k'} u\|_{L^2}+\|\chi_{r,\epsilon',\tau',\sigma}u\|_{L^2}+\|u\|_{L^2}\Big)\\
&+\|u\|_{L^{\infty}}\|\chi_{r-k-1,\epsilon',\tau',\sigma}J^s u\|_{L^2}.
\end{aligned}    
\end{equation}
Collecting \eqref{extraeqdecomp1}-\eqref{extraeqdecomp5}, we almost get \eqref{extranonlEsteq}, however, we need a further simplification using Lemma \ref{lemmadecom}.  For example, writing 
\begin{equation}\label{exampleestim}
    \chi_{r-k,\epsilon',\tau',\sigma'}J^{s-k}u=[\chi_{r-k,\epsilon',\tau',\sigma'},J^{-k}]J^s u+J^{-k}\big(\chi_{r-k,\epsilon',\tau',\sigma'}J^s u\big),
\end{equation}
we can use Lemma \ref{lemmadecom}, the fact that $\chi_{r-k,\epsilon',\tau',\sigma'}\lesssim \chi_{r-1,\epsilon',\tau',\sigma'}$, and similar considerations as above to get the desired result.

\underline{Proof of (II): Case $s> \frac{n}{2}+1$ with $s\geq 3$}. This case shares some similarities to the previous one, however, as the regularity can be arbitrary, we need to keep localization in each of our estimates. We will use the same identity \eqref{eqdecom0}, thus it only remains to study the first and second factors on the right-hand side of \eqref{eqdecom0}. For the former, we use again the same identity \eqref{eqdecom}, and we proceed with the study of the right-hand side of \eqref{eqdecom}. We apply Lemma \ref{conmKP}, Sobolev embedding $H^{(\frac{n}{2})^{+}}(\mathbb{R}^n)\hookrightarrow L^{\infty}(\mathbb{R}^n)$ and the fact that $s>\frac{n}{2}+1$ to get
\begin{equation}\label{eqdecom1}
\begin{aligned}
\|[J^s,\chi_{r,\epsilon,\tau,\sigma}u]\partial_{x_1}(\chi_{\epsilon,\tau,\sigma}u) \|_{L^2}\lesssim & \|\nabla(\chi_{r,\epsilon,\tau,\sigma}u)\|_{L^{\infty}}\|J^{s}( \chi_{\epsilon,\tau,\sigma} u)\|_{L^2}\\
&+\|J^s(\chi_{r,\epsilon,\tau,\sigma}u)\|_{L^{2}}\|\partial_{x_1}(\chi_{\epsilon,\tau,\sigma} u)\|_{L^{\infty}}\\
\lesssim & \|J^{s}(\chi_{r,\epsilon,\tau,\sigma}u)\|_{L^{2}}\|J^{s}(\chi_{\epsilon,\tau,\sigma} u)\|_{L^2}.
\end{aligned}
\end{equation}
By Corollary \ref{remlemmadecom}, we deduce
\begin{equation}\label{nonlEsteq4}
\begin{aligned}
\|J^{s}(\chi_{\epsilon,\tau,\sigma}u)\|_{L^{2}}\lesssim \|\chi_{\epsilon,\tau,\sigma}J^{s} u\|_{L^2}+ \sum_{1\leq k <s} \|\chi_{\epsilon',\tau',\sigma}J^{s-k} u\|_{L^2}+\|u\|_{L^2}.
\end{aligned} 
\end{equation}
Gathering the estimate above and \eqref{extranonlEsteq2}, we have
\begin{equation}\label{extraneweq1}
\begin{aligned}
\|[J^s &,\chi_{r,\epsilon,\tau,\sigma}u]\partial_{x_1}(\chi_{\epsilon,\tau,\sigma}u) \|_{L^2}\\
\lesssim & \Big(\|\chi_{r,\epsilon,\tau,\sigma}J^{s} u\|_{L^2}+ \sum_{1\leq k <s} \|\chi_{r-k,\epsilon',\tau',\sigma}J^{s-k} u\|_{L^2}+\|\chi_{r,\epsilon',\tau',\sigma}u\|_{L^2}+\|u\|_{L^2}\Big)\\
& \times \Big(\|\chi_{\epsilon,\tau,\sigma}J^{s} u\|_{L^2}+ \sum_{1\leq k <s} \|\chi_{\epsilon',\tau',\sigma}J^{s-k} u\|_{L^2}+\|u\|_{L^2}\Big).
\end{aligned}
\end{equation}
Notice that in \eqref{extraneweq1}, we have used the functions $\chi_{\epsilon,\tau,\sigma}$, $\chi_{\epsilon',\tau',\sigma}$ to keep the localization of the operators $J^s$ and $J^{s-k}$. In contrast, in \eqref{extraeqdecomp1}, we used the fact that $u\in H^{(\frac{n}{2}+1)^{+}}(\mathbb{R}^n)$. This last condition is crucial when $n\geq 4$. By a similar argument, it follows 
\begin{equation}\label{nonlEsteq5}
\begin{aligned}
\|[J^s &,\chi_{r,\epsilon,\tau,\sigma}u] \partial_{x_1}(\phi_{\epsilon,\tau,\sigma} u)\|_{L^2}\\
\lesssim & \Big(\|\chi_{r,\epsilon,\tau,\sigma}J^{s} u\|_{L^2}+ \sum_{1\leq k <s} \|\chi_{r-k,\epsilon',\tau',\sigma}J^{s-k} u\|_{L^2}+\|\chi_{r,\epsilon',\tau',\sigma}u\|_{L^2}+\|u\|_{L^2}\Big)\\
&\times \Big(\|\widetilde{\phi}_{\epsilon,\tau,\sigma} J^{s} u\|_{L^2}+ \sum_{1\leq k <s} \|\widetilde{\phi}_{\epsilon,\tau,\sigma} J^{s-k} u\|_{L^2}+\|u\|_{L^2}\Big).
\end{aligned}
\end{equation}
where $\widetilde{\phi}_{\epsilon,\tau,\sigma}(x)=\widetilde{\phi}_{\epsilon,\tau}( \sigma\cdot x+vt)$, with $\widetilde{\phi}_{\epsilon,\tau} \in C^{\infty}_c(\mathbb{R})$ be such that $\widetilde{\phi}_{\epsilon,\tau}\phi_{\epsilon,\tau}=\phi_{\epsilon,\tau}$. 

Next, the estimate of the third term on the right-hand side of \eqref{eqdecom} is the same in \eqref{extraeqdecomp3}. This completes the estimate of \eqref{eqdecom}. Let us focus now on the second factor on the right-hand side of \eqref{eqdecom0}.

For fixed $1\leq k <s$, $|\beta|=k$, we use the same decomposition in \eqref{nonlEsteq7}. Each of the two factors on the right-hand side of \eqref{nonlEsteq7} can be estimated as we did in \eqref{eqdecom}, that is, using commutator estimates and Lemma \ref{conmKP}. We omit these computations to avoid repetitions. However, we remark that we end up with
\begin{equation}\label{nonlEsteq8}
\begin{aligned}
\|J^{s-k}(\partial^{\beta}(\chi_{r,\epsilon,\tau,\sigma})u\partial_{x_1}u)\|_{L^2}\lesssim \mathcal{E}_1+\mathcal{E}_2,
\end{aligned}
\end{equation}
where $\mathcal{E}_1$  and $\mathcal{E}_2$ are the estimates for the two terms on the right-hand side of \eqref{nonlEsteq7}, respectively, which satisfy
\begin{equation*}
    \begin{aligned}
\mathcal{E}_1=&\Big(\|\chi_{r-k,\epsilon',\tau',\sigma}J^{s} u\|_{L^2}+ \sum_{1\leq k'<s} \|\chi_{r-k-k',\epsilon',\tau',\sigma}J^{s-k'} u\|_{L^2}+\|\chi_{r,\epsilon',\tau',\sigma}u\|_{L^2}+\|u\|_{L^2}\Big)\\
&\times \Big(\|\chi_{\epsilon,\tau,\sigma}J^{s-k} u\|_{L^2}+\|\phi_{\epsilon,\tau,\sigma}J^{s-k} u\|_{L^2}\\
&\hspace{0.2cm}+ \sum_{1\leq k'<s-k}\big( \|\chi_{\epsilon',\tau',\sigma}J^{s-k-k'}u\|_{L^2}+\|\widetilde{\phi}_{\epsilon,\tau,\sigma}J^{s-k-k'}u\|_{L^2}\big)+\|\chi_{\epsilon',\tau',\sigma}u\|_{L^2}+\|u\|_{L^2}\Big)\\
 &+\Big(\|\chi_{r-k,\epsilon',\tau',\sigma}J^{s+1-k} u\|_{L^2}+ \sum_{1\leq k'<s+1-k} \|\chi_{r-k-k',\epsilon',\tau',\sigma}J^{s+1-k-k'} u\|_{L^2}\\
 &\hspace{1cm}+\|\chi_{r,\epsilon',\tau',\sigma}u\|_{L^2}+\|u\|_{L^2}\Big)\times \Big(\|\chi_{\epsilon,\tau,\sigma}J^{s} u\|_{L^2}+\|\phi_{\epsilon,\tau,\sigma}J^{s} u\|_{L^2}+ \\
&\sum_{1\leq k'<s}\big( \|\chi_{\epsilon',\tau',\sigma}J^{s-k'}u\|_{L^2}+\|\widetilde{\phi}_{\epsilon,\tau,\sigma}J^{s-k'}u\|_{L^2}\big)+\|\chi_{\epsilon',\tau',\sigma}u\|_{L^2}+\|u\|_{L^2}+\|u\|_{L^{\infty}}\Big),
    \end{aligned}
\end{equation*}
and
\begin{equation*}
    \begin{aligned}
\mathcal{E}_2=&\Big(\|\chi_{r-k-1,\epsilon',\tau',\sigma}J^{s} u\|_{L^2}+ \sum_{1\leq k'<s} \|\chi_{r-k-1-k',\epsilon',\tau',\sigma}J^{s-k'} u\|_{L^2}+\|\chi_{r,\epsilon',\tau',\sigma}u\|_{L^2}+\|u\|_{L^2}\Big)\\
&\times \Big(\|\chi_{\epsilon,\tau,\sigma}J^{s-k} u\|_{L^2}+\|\phi_{\epsilon,\tau,\sigma}J^{s-k} u\|_{L^2}\\
&\hspace{0.2cm}+ \sum_{1\leq k'<s-k}\big( \|\chi_{\epsilon',\tau',\sigma}J^{s-k-k'}u\|_{L^2}+\|\widetilde{\phi}_{\epsilon,\tau,\sigma}J^{s-k-k'}u\|_{L^2}\big)+\|\chi_{\epsilon',\tau',\sigma}u\|_{L^2}+\|u\|_{L^2}\Big)\\
 &+\Big(\|\chi_{r-k-1,\epsilon',\tau',\sigma}J^{s-k} u\|_{L^2}+ \sum_{1\leq k'<s-k} \|\chi_{r-k-1-k',\epsilon',\tau',\sigma}J^{s-k-k'} u\|_{L^2}\\
 &\hspace{1cm}+\|\chi_{r,\epsilon',\tau',\sigma}u\|_{L^2}+\|u\|_{L^2}\Big)\times \Big(\|\chi_{\epsilon,\tau,\sigma}J^{s} u\|_{L^2}+\|\phi_{\epsilon,\tau,\sigma}J^{s} u\|_{L^2}+ \\
&\sum_{1\leq k'<s}\big( \|\chi_{\epsilon',\tau',\sigma}J^{s-k'}u\|_{L^2}+\|\widetilde{\phi}_{\epsilon,\tau,\sigma}J^{s-k'}u\|_{L^2}\big)+\|\chi_{\epsilon',\tau',\sigma}u\|_{L^2}+\|u\|_{L^2}+\|u\|_{L^{\infty}}\Big),
    \end{aligned}
\end{equation*}
where  $1\leq k<s$, and $|\beta|=k$. Thus, to get \eqref{nonlEsteq}, we can use the same idea in \eqref{exampleestim}, which consists of applying Lemma \ref{lemmadecom} to further simplify \eqref{eqdecom1}-\eqref{nonlEsteq8}. The proof is complete. 
\end{proof}
Now that we have estimated $A_1$, $A_2$, and $A_3$, we proceed with the deduction of Theorem \ref{mainTHM}. Setting $\mathcal{B}^s=Id$ the identity operator in \eqref{differineq}, by using Lemmas \ref{linearEstlemma1} and \ref{NonlinearEstlemma1}, together with Gronwall's inequality, we deduce the inequality
\begin{equation}\label{eq3}
\begin{aligned}
\sup_{t\in[0,T]}\int &u^2(x,t)\chi_{r,\epsilon,\tau,\sigma}^2(x,t)\, \mathrm{d}x \\
&+\int_0^T \int |\nabla u|^2(x,t)\big(\langle\sigma \cdot x+\nu t\rangle^{2r-2}(\sigma \cdot x+\nu t)\chi_{\epsilon,\tau,\sigma}^2 \\
& \hspace{3cm} +\langle \sigma \cdot x+\nu t\rangle^{2r}\chi_{\epsilon,\tau,\sigma} \chi_{\epsilon,\tau,\sigma}' \big)\, \mathrm{d}x \, \mathrm{d} t \\
&\leq c(\|u\|_{H^s},\|u\|_{L^1_T W^{1,\infty}},\|\langle \sigma \cdot x \rangle^{r}u_0\|_{L^2(\mathcal{H}_{\sigma,0})}).
\end{aligned}
\end{equation}
Therefore, the above estimate establishes the propagation of polynomial decay
\begin{equation}\label{eqconclu1}
\begin{aligned}
    \chi_{r,\epsilon,\tau,\sigma} u(x,t)\in L^{\infty}([0,T];L^2(\mathbb{R}^n)),
\end{aligned}
\end{equation}
a type of regularity and decay localized estimate
\begin{equation}\label{eqconclu2}
\begin{aligned}
\chi_{\epsilon,\tau,\sigma}^2\langle\sigma \cdot x+\nu t\rangle^{2r-2}(\sigma \cdot x+\nu t) |\nabla u(x,t)|^2\in L^{1}((0,T);L^1(\mathbb{R}^n)),
\end{aligned}
\end{equation}
and the local smoothing effect
\begin{equation}\label{eqconclu2.1}
\begin{aligned}
\big(\chi_{\epsilon,\tau,\sigma}\chi'_{\epsilon,\tau,\sigma}\big)^{\frac{1}{2}} \partial^{\gamma} u(x,t)\in L^{2}((0,T);L^2(\mathbb{R}^n)), \, \, \text{ for all } \, |\gamma|=1.
\end{aligned}
\end{equation}
Letting $2\epsilon<\widetilde{\epsilon}<\widetilde{\tau}<\tau-2\epsilon$, by property (iii) in Subsection \ref{smoothapprosec}, Lemma \ref{lemmapropaform} and \eqref{eqconclu2.1}, we get
\begin{equation}\label{eqconclu2.2}
\begin{aligned}
 \Theta^{m} u(x,t)\in L^{2}((0,T);L^2(\mathcal{Q}_{\{\sigma,\widetilde{\epsilon}-\nu t,\widetilde{\tau}-\nu t\}})),
\end{aligned}
\end{equation}
for each $0<m\leq 1$, and $\Theta_m$ is either $J^m$ or $\partial^{\beta}$ with $|\beta|=m$, if $m=1$ (also recall the definition of $\mathcal{Q}_{\{\sigma,\widetilde{\epsilon}-\nu t,\widetilde{\tau}-\nu t\}}$ in  \eqref{defQ}).  Although \eqref{eqconclu2.1} and \eqref{eqconclu2.2} hold without localization as $u\in C([0, T]; H^{(\frac{n}{2}+1)^{+}}(\mathbb{R}^n))$, the previous results are relevant to scheme the proof of Theorem \ref{mainTHM}. On the other hand, notice that \eqref{eqconclu2} involves a gain of decay and regularity for solutions of \eqref{ZK}. To exploit this fact, we will use an inductive argument on the size of the weight $r>0$. 

\underline{\bf Case $0\leq r$}. We will further split this case as follows:
\begin{itemize}
\item[(a)] Case $0<r<\frac{1}{2}$.
\item[(b)] Case $r=\frac{1}{2}$.
\item[(c)] Case $\frac{1}{2}<r$. 
\end{itemize}
\underline{\bf Case (a)}. If $0<r<\frac{1}{2}$, we are unable to obtain further properties from \eqref{eqconclu2}. The proof of Theorem \ref{mainTHM} is complete in this case. 

Let us assume $r\geq \frac{1}{2}$. Before we deal with the cases (b) and (c) separately, we set $r_1=r-\frac{1}{2}=\frac{2r-1}{2}$, since $\sigma \cdot x +\nu t\geq \epsilon$ in the support of $\chi_{\epsilon,\tau,\sigma}$, it is seen that 
\begin{equation*}
\begin{aligned}
\langle\sigma \cdot x+\nu t\rangle^{2r-2}(\sigma \cdot x+\nu t)\chi_{\epsilon,\tau,\sigma}^2 \gtrsim_{\epsilon,r} \chi_{r_1, \epsilon,\tau,\sigma}^2.
\end{aligned}
\end{equation*}
Hence, \eqref{eqconclu2} and an application of Lemma \ref{lemmapropaform} allow us to deduce
\begin{equation}\label{eqconclu2.3}
    \chi_{r_1,\widetilde{\epsilon},\widetilde{\tau},\sigma}\Theta^m u(x,t)\in L^2((0,T);L^2(\mathbb{R}^n)),
\end{equation}
whenever $\widetilde{\epsilon}>\tau$, $\widetilde{\tau} \geq 5\widetilde{\epsilon}$, and either $\Theta^m=J^m$, $0\leq m \leq 1$, or $\Theta^m=\partial^{\beta}$ with $\beta=m=1$. Moreover, \eqref{eqconclu2} establishes that for all $\delta>0$, there exists $t_1\in (0,\delta)$ such that
\begin{equation}\label{newinicond}
\begin{aligned}
\int |\nabla u|^2(x,t_1)\chi_{r_1,\epsilon,\tau,\sigma}^2(x,t_1) \, \mathrm{d}x <\infty.
\end{aligned}
\end{equation}
Consequently,  Lemma \ref{lemmapropaform} yields
\begin{equation}\label{newinicond0.1}
\begin{aligned}
\int |\Theta^m u|^2(x,t_1)\chi_{r_1,\widetilde{\epsilon},\widetilde{\tau},\sigma}^2(x,t_1)\, \mathrm{d}x<\infty,
\end{aligned}
\end{equation}
whenever $\widetilde{\epsilon}>\tau$, $\widetilde{\tau}\geq 5\widetilde{\epsilon}$, and $\Theta^m$ be given as in \eqref{eqconclu2.3}. In particular, when $r\geq \frac{1}{2}$, we get
\begin{equation}\label{newinicond1}
\begin{aligned}
\int |\Theta^m u|^2(x,t_1)\chi_{\widetilde{\epsilon},\widetilde{\tau},\sigma}^2(x,t_1)\, \mathrm{d}x<\infty.
\end{aligned}
\end{equation}
The previous two equations suggest that at later times $t\geq t_1$, the solution of the ZK equation \eqref{ZK} propagates extra regularity and decay. This will be confirmed according to cases (b) and (c) below.

\underline{\bf Case (b)}. Here we assume $r=\frac{1}{2}$. We fix $\widetilde{\epsilon}, \widetilde{\tau}>0$ such that \eqref{eqconclu2.3}-\eqref{newinicond0.1} are valid. By \eqref{newinicond1}, we can apply the propagation of regularity result in  Theorem \ref{TheorArg} with initial data $u(\cdot,t_1)$ to get
\begin{equation}\label{eqconclu3}
  \chi_{\epsilon_1, \tau_1,\sigma}(x,t)  J^{s_1} u(x,t)\in L^{\infty}([t_1,T];L^2({\mathbb{R}^n})).
\end{equation}
for fixed $\epsilon_1>\widetilde{\tau}$, $\tau_1 \geq 5\epsilon_1$, and any $s_1\in (0,1]$. Moreover, up to a further modification of $\epsilon_1$, and $\tau_1$, Theorem \ref{TheorArg} and Lemma \ref{lemmapropaform} (II) yield 
\begin{equation*}
\Theta^{m}u\in L^{2}((t_1,T);\mathcal{Q}_{\{\sigma,\epsilon_1-\nu t,\tau_1-\nu t\}}),
\end{equation*}
where $0\leq m\leq 2$, and $\Theta^m=J^m$, or $\Theta^m=\partial^{\beta}$, $0\leq \beta \leq m$. Let us now use \eqref{eqconclu1} and \eqref{eqconclu3} to obtain a relation between decay and regularity for fractional derivatives of order $0<s<1$. An application of Lemma \ref{InterpLemma} reveals  
\begin{equation*}
\begin{aligned}
\|J^{s}\big(\chi_{\frac{1-s}{2},\epsilon_1, \tau_1,\sigma} u\big)\|_{L^2} \lesssim \|\chi_{\frac{1}{2},\epsilon_1, \tau_1,\sigma} u\|_{L^2}+\|J(\chi_{\epsilon_1, \tau_1,\sigma} u)\|_{L^2}.
\end{aligned}
\end{equation*}
The first term on the right-hand side of the above identity is controlled by \eqref{eqconclu1}, while the second term is estimated by Corollary \ref{remlemmadecom} and the propagation result \eqref{eqconclu3}. Therefore, up to choosing new parameters $\widetilde{\epsilon}_1>\tau_1$, and $\widetilde{\tau}_1 \geq 5\widetilde{\epsilon}_1$, a second application of Corollary \ref{remlemmadecom} and Lemma \ref{lemmapropaform} show 
\begin{equation}\label{interpoldecay}
\sup_{t_1 \leq t \leq T}\int  \big(J^s u(x,t) \big)^2 \langle \sigma \cdot x+\nu t \rangle^{(1-s)} \chi_{\epsilon_1,\tau_1,\sigma}^2(x,t)\, \mathrm{d}x <\infty,
\end{equation}
for each $0<s\leq 1$. This completes the proof of Case (b).

\underline{\bf Case (c)}. Since in this case, $\frac{1}{2}<r$, \eqref{newinicond0.1} shows that starting from $t_1$, one expects additional localized decay and regularity 
for solutions $u$ of \eqref{ZK}. Thus, we study the propagation of such regularity by using \eqref{differineq} with $\mathcal{B}^s=\partial^{\gamma}$, $|\gamma|=1$, $\chi_{r_1,\epsilon_1,\tau_2,\sigma}$ with $r_1=r-\frac{1}{2}$, and the restrictions $\epsilon_1>\widetilde{\tau}$, $\tau_1\geq 5\epsilon_1$, where $\widetilde{\epsilon}, \widetilde{\tau}$ are fixed positive numbers such that \eqref{eqconclu2.3}-\eqref{newinicond0.1} hold. The idea is that under the present restrictions, we will derive bounds for $A_1$, $A_2$, and $A_3$ that combined with Gronwall's inequality yield the desired conclusion.

The estimate for $A_1$ is developed in \eqref{estA1}. The first term on the right-hand side of \eqref{estA1} is the quantity to be estimated using Gronwall's inequality. Noticing that $\langle \sigma\cdot x+\nu t \rangle^{2r}\chi_{\epsilon_1,\tau_1,\sigma}\chi_{\epsilon_1,\tau_1,\sigma}'\lesssim \chi_{\widetilde{\epsilon},\widetilde{\tau},\sigma}^2$, the second term on the right-hand side of \eqref{estA1} is estimated by integrating on the time variable between $(t_1,T)$ and using \eqref{eqconclu2.3}. (The integration of the time variable is justified by Gronwall's inequality). Similarly, we control $R(\partial^{\gamma}u,t)$ in \eqref{estA2}, and the remaining part of the expansion for $A_2$ in \eqref{estA2} provides a positive quantity.

Finally, the estimate for $A_3$ follows from the fact that $u\in L^{1}([0,T];L^2(\mathbb{R}^n))$, which is a consequence of Sobolev embedding $H^{(\frac{n}{2}+1)^{+}}(\mathbb{R}^n)\hookrightarrow W^{1,\infty}(\mathbb{R}^n)$. Thus, by the discussion above and summing over all multi-index $\gamma$ of order one, Gronwall's inequality yields the estimate
\begin{equation}\label{differineq2}
\begin{aligned}
\sup_{t\in[t_1,T]}&\sum_{|\gamma|=1}\int \big(\partial^{\gamma}u(x,t)\big)^2\chi_{r_1,\epsilon_1,\tau_1,\sigma}^2(x,t)\, \mathrm{d}x \\
&+\sum_{|\gamma|=1}\int_{t_1}^T \int |\nabla \partial^{\gamma}u|^2(x,t)\big(\langle\sigma \cdot x+\nu t\rangle^{2r_1-2}(\sigma \cdot x+\nu t)\chi_{\epsilon_1,\tau_1,\sigma}^2 \\
& \hspace{3cm} +\langle \sigma \cdot x+\nu t\rangle^{2r_1}\chi_{\epsilon_1,\tau_1,\sigma} \chi_{\epsilon_1,\tau_1,\sigma}' \big)\, \mathrm{d}x \, \mathrm{d} t \\
&\leq c_1,
\end{aligned}
\end{equation}
where the constant $c_1>0$ depends on
\begin{equation*}
\|u\|_{H^s},\|u\|_{L^1_T W^{1,\infty}},\|\nabla u (x,t_1)\chi_{r_1,\epsilon_1,\tau_1,\sigma}\|_{L^2},\|\chi_{r_1,\widetilde{\epsilon},\widetilde{\tau},\sigma}\nabla u\|_{L^2((t_1,T);L^2(\mathbb{R}^n))}.    
\end{equation*}
In particular, we find
\begin{equation}\label{eqconclu4}
    \chi_{r_1,\epsilon_1,\tau_1,\sigma}\nabla u(x,t)\in L^{\infty}([t_1,T];L^2(\mathbb{R}^n)),
\end{equation}
\begin{equation}\label{eqconclu4.0}
\begin{aligned}
\chi_{\epsilon_1,\tau_1,\sigma}^2\langle\sigma \cdot x+\nu t\rangle^{2r_l-2}(\sigma \cdot x+\nu t) |\nabla \partial^{\gamma} u(x,t)|^2\in L^{1}((t_1,T);L^1(\mathbb{R}^n)),
\end{aligned}
\end{equation}
and
\begin{equation}\label{eqconclu4.01}
    \big(\chi_{\epsilon_1,\tau_1,\sigma}\chi_{\epsilon_1,\tau_1,\sigma}'\big)^{\frac{1}{2}}\nabla \partial^{\gamma} u(x,t)\in L^{2}([t_1,T];L^2(\mathbb{R}^n)),
\end{equation}
for all multi-index $|\gamma|=1$. Hence, when $\epsilon_1>\widetilde{\tau}$ and $\tau_1 \geq 5\epsilon_1$, at later times $t\geq t_1$, \eqref{eqconclu4} establishes a stronger propagation of regularity effect than that obtained in  \eqref{eqconclu2.3}. In particular, Lemma \ref{lemmapropaform} and \eqref{eqconclu4} yield
\begin{equation}\label{eqconclu4.1}
    \chi_{r_1,\widetilde{\epsilon}_1,\widetilde{\tau}_1,\sigma}\Theta^m u(x,t)\in L^{\infty}([t_1,T];L^2(\mathbb{R}^n)),
\end{equation}
provided that $\widetilde{\epsilon}_1>\tau_1$, and $\widetilde{\tau}_1\geq 5 \widetilde{\epsilon}_1$, where either $\Theta^m=J^{m}$, $0\leq m \leq 1$, or $\theta^{m}=\partial^{\beta}$, with $|\beta|=1$ if $m=1$. Finally, by using interpolation Lemma \ref{InterpLemma} with the fact
\begin{equation}\label{eqconclu4.2}
\chi_{\epsilon_1,\tau_1,\sigma}Ju,\chi_{r,\epsilon,\tau,\sigma}u \in L^{\infty}([t_1,T];L^{2}(\mathbb{R}^n)),   
\end{equation} 
together with \eqref{eqconclu4.1}, it follows
\begin{equation}\label{eqconclu5}
\sup_{t_1 \leq t \leq T}\int  \big(J^{s_1} u(x,t) \big)^2 \langle \sigma \cdot x+\nu t \rangle^{2r^{\ast}_1} \chi_{\widetilde{\epsilon}_1,\widetilde{\tau}_1,\sigma}^2(x,t)\, \mathrm{d}x <\infty,
\end{equation}
where $r_1^{\ast}=\max\{(1-s_1)r, r-\frac{1}{2}\}$, $0\leq s_1\leq 1$. We remark that in order to apply the interpolation Lemma \ref{InterpLemma} with \eqref{eqconclu4.2}, we have used first Corollary \ref{remlemmadecom} to write $\chi_{\epsilon_1,\tau_1,\sigma}Ju$ in terms of $J(\chi_{\epsilon_1,\tau_1,\sigma}u)$ plus a remainder.  Besides, notice that if $0<r_1<\frac{1}{2}$ (i.e., $\frac{1}{2}<r<1$), \eqref{eqconclu4.0} does not provide extra decay to repeat the process described above. 

The argument now follows by iteration of the previous cases, where the number of repetitions is determined by the weight parameter $r$. For the general case, we will argue by induction on some lower bound of the weight parameter $r>0$. 

\underline{Case weight $\frac{l}{2}\leq  r$ with $l\in \mathbb{N}_0.$} By the cases considered previously, we will restrict our considerations to $l\geq 2$. In virtue of the inductive hypothesis, we will assume that the case $l-1$ holds, that is for $\delta>0$, with $\delta<T$, which was fixed before, there exists a time $t_{l-1}\in (0,\delta)$ such that
\begin{equation}\label{propaInd}
    \chi_{r_{l-1},\epsilon_{l-1},\tau_{l-1},\sigma}\Theta^m u(x,t)\in L^{\infty}([t_{l-1},T];L^2(\mathbb{R}^n)),
\end{equation}
and
\begin{equation}
\begin{aligned}\label{propaInd2}
\chi_{\epsilon_{l-1},\tau_{l-1},\sigma}^2\langle\sigma \cdot x+\nu t\rangle^{2r_{l-1}-2}&(\sigma \cdot x+\nu t) |\nabla \widetilde{\Theta}^{l-1} u(x,t)|^2\\
&\in L^{1}((t_{l-1},T);L^1(\mathbb{R}^n)),
\end{aligned}
\end{equation}
with $r_{l-1}=r-\frac{(l-1)}{2}$, $\epsilon_{l-1}>0$, $\tau_{l-1} \geq 5 \epsilon_{l-1}$ \footnote{The inductive argument shows that $\epsilon_{l-1}$ and $\tau_{l-1}$ are constructed recursively by considering $\{\epsilon_j\}_{j=1}^{l-1}$, $\{\tau_{j}\}_{j=1}^{l-1}$, with the  condition $\epsilon_{j+1}> \tau_j$, $j=1\dots,l-2$, and $\tau_j \geq 5\epsilon_j$ for all $j$.}, where either $\Theta^m=J^m$ with $0\leq m \leq l-1$, or $\Theta^m=\partial^{\beta}$ with $0\leq |\beta|\leq m$. Moreover,  $\widetilde{\Theta}^{l-1}=J^{l-1}$ if $l\geq 4$, and $\widetilde{\Theta}^{l-1}=\partial^{\beta}$, $|\beta|=l-1$ if $l\leq 3$. Notice that $\widetilde{\Theta}^{l-1}$ is motivated by the cases \eqref{estnonT} and \eqref{estnonT2} in Lemma \ref{NonlinearEstlemma1}, in other words, when the regularity index $l-1\leq 2$, we use partial derivatives in \eqref{propaInd2}, and when it is greater or equal than $3$, we use the operator $J^{l-1}$ instead. We will also assume by induction
\begin{equation}\label{propaInd3}
\sup_{t_{l-1} \leq t \leq T}\int  \big(J^{s_{l-1}} u(x,t) \big)^2 \langle \sigma \cdot x+\nu t \rangle^{2r^{\ast}_{l-1}} \chi_{\epsilon_{l-1},\tau_{l-1},\sigma}^2(x,t)\, \mathrm{d}x <\infty,
\end{equation}
where $r^{\ast}_{l-1}=\max\left\{\big(1-\frac{s_{l-1}}{l-1}\big)r,r-\frac{\lceil s_{l-1} \rceil}{2}\right\}$, and $s_{l-1}\in[0,l-1]$. We remark that $l-1=1$ in \eqref{propaInd}, \eqref{propaInd2} and \eqref{propaInd3} correspond to  \eqref{eqconclu4.1}, \eqref{eqconclu4.0} and \eqref{eqconclu5}, respectively.

Since $r\geq \frac{l}{2}$, by support considerations of the function $\chi_{\epsilon_{l-1},\tau_{l-1},\sigma}$, it is seen that
\begin{equation*}
    \chi_{\epsilon_{l-1},\tau_{l-1},\sigma}^2\langle\sigma \cdot x+\nu t\rangle^{2r_{l-1}-2}(\sigma\cdot x+\nu t)\gtrsim \chi_{\epsilon_{l-1},\tau_{l-1},\sigma}^2\langle\sigma \cdot x+\nu t\rangle^{2r_{l-1}-1}.
\end{equation*}
Therefore, \eqref{propaInd2} provides a smoothing effect, which combined with Lemma \ref{lemmapropaform} implies
\begin{equation}\label{L2smoothing}
    \begin{aligned}
\chi_{r_{l},\widetilde{\epsilon}_{l-1}, \widetilde{\tau}_{l-1},\sigma}\Theta^m_1 u (x,t)\in L^{2}((t_{l-1},T);L^2(\mathbb{R}^n))
    \end{aligned}
\end{equation}
with $r_{l}=r-\frac{l}{2}$, and either $\Theta^{m}_1=J^m$, $1\leq m\leq l$, or $\Theta^{m}_1=\partial^{\beta}$, $0\leq |\beta|\leq m$, and $\widetilde{\epsilon}_{l-1}> \tau_{l-1}$, $\widetilde{\tau}_{l-1}\geq 5\widetilde{\epsilon}_{l-1}$. On the other hand, since $t_{l-1}\in (0,\delta)$, by \eqref{propaInd2} and Lemma \ref{lemmapropaform} there exists $t_{l}\in (t_{l-1},\delta)$ such that
\begin{equation}\label{newinitialcond2}
    \int |\Theta^{m}_1 u(x,t_{l})|^{2}\chi_{r_{l},\widetilde{\epsilon}_{l-1}, \widetilde{\tau}_{l-1},\sigma}^2(x,t_{l}) \, \mathrm{d}x<\infty.
\end{equation}
In particular, since $\langle x \rangle^{r_l}\gtrsim 1$, we have $\Theta^{m}_1 u(x,t_{l})\chi_{\widetilde{\epsilon}_{l-1}, \widetilde{\tau}_{l-1},\sigma}(x,t_{l})\in L^{2}(\mathbb{R}^n)$. Thus, we can apply Theorem \ref{TheorArg} on the interval $[t_{l},T]$ to deduce
\begin{equation}\label{eqconclu6}
    \chi_{\widetilde{\epsilon}_l,\widetilde{\tau}_l,\sigma}(x,t)J^{s_l}u \in L^{\infty}([t_l,T];L^2(\mathbb{R}^n)), 
\end{equation}
and
\begin{equation*}
J^{l+1}u\in L^{2}((t_1,T);\mathcal{Q}_{\{\sigma,\widetilde{\epsilon}_l-\nu t,\widetilde{\tau}_l-\nu t\}}),
\end{equation*}
with $\widetilde{\epsilon}_l> \widetilde{\tau}_{l-1}$, $\widetilde{\tau}_l\geq 5\widetilde{\epsilon}_l$, and any $s_l\in (0,l]$. Consequently, \eqref{eqconclu6} shows that at later times $t \geq t_l$, the solution $u$ of the ZK equation propagates extra regularity localized in the support of $\chi_{\widetilde{\epsilon}_l,\widetilde{\tau}_l,\sigma}$. Thus, when $r>\frac{l}{2}$, we can use identity \eqref{differineq} and Gronwall's inequality on the interval $(t_l, T)$ to obtain extra decay and regularity. To check this remark, we further divide our arguments into the following cases:
\begin{itemize}
    \item[(d)] case $r=\frac{l}{2}$.
    \item[(e)] case $\frac{l}{2}<r$.
\end{itemize}
\underline{\bf Case (d)}. If $r=\frac{l}{2}$, we cannot obtain further regularity and polynomial decay properties from \eqref{propaInd2}. Thus, by using Lemma \ref{InterpLemma}, we interpolate \eqref{propaInd3} with $s_{l-1}=0$, and \eqref{eqconclu6} to deduce
\begin{equation*}
\sup_{t_{l} \leq t \leq T}\int  \big(J^{s_{l}} u(x,t) \big)^2 \langle \sigma \cdot x+\nu t \rangle^{2r^{\ast}_{l}} \chi_{\epsilon_1,\tau_1,\sigma}^2(x,t)\, \mathrm{d}x <\infty,
\end{equation*}
where $r^{\ast}_{l}=\max\{\big(1-\frac{s_{l}}{l}\big)\frac{l}{2},r-\frac{\lceil s_l \rceil}{2}\}$, $s_{l}\in[0,l]$, $\epsilon_l>\widetilde{\tau}_l$, $\tau_l\geq 5\epsilon_l$. Note that given the induction hypothesis, in this case, the proof of Theorem \ref{mainTHM} is complete.

\underline{\bf Case (e)}. We set $\frac{l}{2}<r$. In this case, we can use \eqref{newinitialcond2} and the differential equation \eqref{differineq} to obtain extra propagation of regularity for the given solution $u$ of the ZK equation. If $l=2$, we consider \eqref{differineq} with $\mathcal{B}^s=\partial^{\beta}$, $|\beta|=2$. However, by Lemma \ref{NonlinearEstlemma1}, this occurrence is similar to that treated in Case (c) above, thus we omit its analysis. We will assume that $l\geq 3$, and let $\epsilon_l> 4\widetilde{\tau}_{l}$, $\tau_l\geq 5\epsilon_l$, where we emphasize that $\widetilde{\epsilon}_{l}$, and $\widetilde{\tau}_{l}$ were chosen according to \eqref{eqconclu6}. We consider \eqref{differineq} with $\mathcal{B}^s=J^{l}$, and $\chi_{r_l, \epsilon_l,\tau_l,\sigma}$, $r_l=r-\frac{l}{2}$. Consequently, we proceed with the study of $A_1$, $A_2$ and $A_3$ when $\mathcal{B}^s=J^{l}$, and with the weighted function $\chi_{r_l, \epsilon_l,\tau_l,\sigma}$.

By our choice of parameters, and the fact that $1\lesssim \langle x\rangle^{r_l}$, we have
\begin{equation*}
    \begin{aligned}
    \chi_{\epsilon_l,\tau_l,\sigma}^{(j)}\chi_{\epsilon_l,\tau_l,\sigma}^{(j')}\lesssim \chi_{r_l,\widetilde{\epsilon}_{l-1},\widetilde{\tau}_{l-1},\sigma}^2,
    \end{aligned}
\end{equation*}
for any integer $j,j' \geq 0$. Hence, the above estimate, Lemma \ref{linearEstlemma1}, and Gronwall's inequality applied to \eqref{differineq} between $(t_l, T)$ allow us to use \eqref{L2smoothing} to obtain the desired estimate for $A_1$ and $A_2$. (Notice that the required integration in time is justified by applying Gronwall's inequality, and the fact $t_{l-1}\leq t_l$ by construction). Furthermore, after Gronwall's inequality, \eqref{estA2} establishes
\begin{equation}\label{L2smoothing2}
\begin{aligned}
\chi_{\epsilon_{l},\tau_{l},\sigma}^2\langle\sigma \cdot x+\nu t\rangle^{2r_{l}-2}(\sigma \cdot x+\nu t) |\nabla J^{l} u(x,t)|^2\in L^{1}((t_{l},T);L^1(\mathbb{R}^n)).
\end{aligned}
\end{equation}
Next, we turn to the estimate for $A_3$. We will only estimate the most difficult case of $A_3$ given by the bound \eqref{nonlEsteq} as that from \eqref{extranonlEsteq} follows from the same ideas. We use \eqref{nonlEsteq} in Lemma \ref{NonlinearEstlemma1} to deduce
\begin{equation}\label{eqconclu7}
\begin{aligned}
|A_3|\lesssim & \Big(\|u\|_{L^{\infty}}+\|\nabla u\|_{L^{\infty}} \Big)\Big( \|u\|_{L^2}+\|\chi_{r_l,\widetilde{\epsilon}_{l},\widetilde{\tau}_{l},\sigma} u\|_{L^2}\Big)\|\chi_{r_l,\epsilon_l,\tau_l,\sigma}J^l u\|_{L^2}\\
&+\Big(\|\chi_{r_l,\widetilde{\epsilon}_{l},\widetilde{\tau}_{l},\sigma}J^l u\|_{L^2}+\|\chi_{r_l,\widetilde{\epsilon}_{l},\widetilde{\tau}_{l},\sigma} u\|_{L^2}+\|u\|_{L^2}\Big)\\
& \qquad \times\Big(\|\chi_{\widetilde{\epsilon}_{l},\widetilde{\tau}_{l},\sigma}J^l u\|_{L^2}+\|u\|_{L^{\infty}}+\|u\|_{L^2}\Big)\|\chi_{r_l,\epsilon_l,\tau_l,\sigma}J^l u\|_{L^2}\\
&+\Big(1+\|u\|_{L^{\infty}}+\|\nabla u\|_{L^{\infty}}\Big)\|\chi_{r_l,\epsilon_l,\tau_l,\sigma}J^l u\|_{L^2}^{2}\\
=:& A_{3,1}\|\chi_{r_l,\epsilon_l,\tau_l,\sigma}J^l u\|_{L^2}+A_{3,2}\|\chi_{r_l,\epsilon_l,\tau_l,\sigma}J^l u\|_{L^2}\\
&+A_{3,3}\|\chi_{r_l,\epsilon_l,\tau_l,\sigma}J^l u\|_{L^2}^2,
\end{aligned}
\end{equation}
where recalling the choice of $\widetilde{\epsilon}_{l}$, $\widetilde{\tau}_l$ in \eqref{eqconclu6}, we have used the fact $\widetilde{\tau}_{l} \geq 5\widetilde{\epsilon}_l$, and $\epsilon_l>4\widetilde{\tau}_{l}$. As a matter of fact, $\epsilon_l>4\widetilde{\tau}_{l}$ implies that $\phi_{\epsilon_l,\tau_l}\lesssim \chi_{\widetilde{\epsilon}_{l},\widetilde{\tau}_{l},\sigma}$. Therefore, when we apply Gronwall's inequality, after integration on $(t_l, T)$, the estimate for \eqref{eqconclu7} will be finite provided that $A_{3,j}\in L^1_t(t_1, T)$, $j=1,2,3$. Indeed, by applying Cauchy Schwarz inequality in time, we have that  $A_{3,1}\in L^1_t(t_1, T)$ requires
\begin{equation}\label{Condieq1}
    u, \nabla u  \in L^{1}((t_l,T);L^{\infty}(\mathbb{R}^n)),
\end{equation}
and
\begin{equation}\label{Condieq2}
    u, \chi_{r_l,\widetilde{\epsilon}_l,\widetilde{\tau}_l,\sigma} u \in L^{\infty}([t_l,T];L^{2}(\mathbb{R}^n))
\end{equation}
Considering that the solution $u$ of the Cauchy problem \eqref{ZK} determines a continous curve in $H^{\big(\frac{n+2}{2}\big)^{+}}(\mathbb{R}^n)$, an application of Sobolev embedding $H^{\frac{n}{2}^{+}}(\mathbb{R}^n)\hookrightarrow L^{\infty}(\mathbb{R}^n)$ justify the validity of \eqref{Condieq1}, while \eqref{Condieq2} holds true by the induction hypothesis \eqref{propaInd3}. Note that our definition of parameters $\widetilde{\epsilon}_l$, $\widetilde{\tau}_l$ also plays a crucial role to keep the localization properties of $u$. Similarly, \eqref{Condieq1} implies that $A_{3,3}\in L^1_t(t_1,T)$. On the other hand, to have that $A_{3,2}\in L^1_t(t_1, T)$, an application of Cauchy Schwarz inequality on the time variable reduces us to prove
\begin{equation}\label{Condieq3}
    \chi_{r_l,\widetilde{\epsilon}_l,\widetilde{\tau}_l,\sigma}u, \chi_{r_l,\widetilde{\epsilon}_l,\widetilde{\tau}_l,\sigma}J^l u, \chi_{\widetilde{\epsilon}_l,\widetilde{\tau}_l,\sigma}J^l u  \in L^{2}((t_l,T);L^{2}(\mathbb{R}^n)).
\end{equation}
However, \eqref{Condieq3} is a consequence of \eqref{propaInd3} (with $s_{l-1}=0$), \eqref{L2smoothing}, \eqref{eqconclu6}, and the fact that $t_{l-1}<t_l$. Therefore, the previous results complete the estimate for $A_3$ in \eqref{differineq}. Collecting the estimates for $A_1$, $A_2$ and $A_3$ deduced above, we can apply Gronwall's inequality as we did in  \eqref{differineq2} to conclude
\begin{equation}\label{eqconclu8}
    \begin{aligned}
\chi_{r_{l},\epsilon_{l}, \tau_{l},\sigma}J^l u (x,t)\in L^{\infty}([t_{l},T];L^2(\mathbb{R}^n)).
    \end{aligned}
\end{equation}
Setting $\widetilde{\epsilon}_l>\tau_l$, $\widetilde{\tau}_l\geq 5 \widetilde{\epsilon}_l$, we use \eqref{eqconclu8} and Lemma \ref{lemmapropaform} to get
\begin{equation}\label{eqconclu9}
    \begin{aligned}
\chi_{r_{l},\widetilde{\epsilon}_{l}, \widetilde{\tau}_{l},\sigma}\Theta^m_1 u (x,t)\in L^{\infty}([t_{l},T];L^2(\mathbb{R}^n)),
    \end{aligned}
\end{equation}
with $r_{l}=r-\frac{l}{2}$, $\Theta^{m}_1=J^m$, either $1\leq m\leq l$, or $\Theta^{m}_1=\partial^{\beta}$, $0\leq |\beta|\leq m$, and $\widetilde{\epsilon}_{l}> \tau_{l}$, $\widetilde{\tau}_{l}\geq 5\widetilde{\epsilon}_{l}$. Next, by \eqref{eqconclu9}, and using Lemma \ref{InterpLemma} to interpolate \eqref{propaInd3} with $s_{l-1}=0$, and \eqref{eqconclu6}, we deduce
\begin{equation}\label{interpconclus} 
\sup_{t_{l} \leq t \leq T}\int  \big(J^{s_{l}} u(x,t) \big)^2 \langle \sigma \cdot x+\nu t\rangle^{2r^{\ast}_{l}} \chi_{\overline{\epsilon}_1,\overline{\tau}_1,\sigma}^2(x,t)\, \mathrm{d}x <\infty,
\end{equation}
where $r^{\ast}_{l}=\max\{\big(1-\frac{s_{l}}{l}\big)r,r-\frac{\lceil s_l \rceil}{2}\}$, $s_{l}\in[0,l]$, $\overline{\epsilon}_l>\widetilde{\tau}_l$, $\overline{\tau}_l\geq 5\overline{\epsilon}_l$. Gathering \eqref{L2smoothing2} \eqref{eqconclu9}, \eqref{interpconclus} and using the indicator function $\chi_{\overline{\epsilon}_1,\overline{\tau}_1,\sigma}$ in each case, we conclude the inductive step. Observe that the process established above can be iterated an integer number of times $l\in \mathbb{N}_0$ as long as the condition $0\leq r-\frac{l}{2}<\frac{1}{2}$ holds, in other words, when $l=\lfloor 2r\rfloor$. In particular, this implies that when the initial condition has localized polynomial decay of order $r$,  $\lfloor 2r\rfloor$ is the maximum localized regularity propagated by solutions of ZK. The proof of Theorem \ref{mainTHM} is complete.
\section{Propagation exponential decay: Proof of Theorem \ref{expodecay}}\label{ExpodecaySect}
This section concerns the deduction of Theorem \ref{expodecay}. As pointed out by Kato \cite{Kato1983}, the exponential function $e^{bx}$ is highly unbounded, which makes it more difficult to work with it in the energy estimates as we did with the polynomial weights in previous sections. To avoid this problem, we follow the approach used in \cite{Kato1983} in the context of the KdV equation,  that is,  we approximate the function $e^{bx}$ by  a bounded  weighted function
\begin{equation*}
	q_{\eta}(x):=e^{bx}\left(1+\eta e^{2bx}\right)^{-\frac{1}{2}}\, x\in\mathbb{R},
\end{equation*}
depending on the parameter $\eta>0.$  Additionally, it will be convenient in our analysis to define the following auxiliary functions of the variable $x\in\mathbb{R},$ namely,
\begin{equation*}
	\rho_{\eta}(x):=e^{bx}\left(1+\eta e^{2bx}\right)^{-1},\quad p_{\eta}(x):=q_{\eta}^{2}(x).
\end{equation*}
Both functions $\rho_{\eta}$ and $q_{\eta}$ are monotonous and converge to $e^{bx}$ when $\eta\rightarrow 0^{+}$. We also emphasize that this approach was used more recently in \cite{DeleageLinares2023} to study the dispersive blow-up of the solutions to the ZK equation.

Next, we show some properties of the weighted functions introduced above. For any $\eta>0$ and $x\in\mathbb{R}$ the following inequality holds
\begin{equation*}
	0\leq \rho_{\eta}(x)\leq q_{\eta}
	(x)\leq e^{bx},
\end{equation*}
\begin{equation*}
	p_{\eta}'(x)=2b\rho_{\eta}^{2}(x),\quad |p''(x)|\leq 4b^{2}\rho_{\eta}^{2}(x),
\end{equation*}
and 
\begin{equation*}
	|p'''_{\eta}(x)|\lesssim \rho_{\eta}^{2}(x),\quad |\rho'_{\eta}(x)|\leq b\rho_{\eta}(x),\quad \forall x\in\mathbb{R}.
\end{equation*}
We also introduce some additional notation to be used in our arguments in this section. Let $\sigma\in \mathbb{R}^n$ such that \eqref{sigmacond} holds, $\kappa\in \mathbb{R}$, $\epsilon>0$ and $\tau\geq 5\epsilon$, we consider the function
\begin{equation}\label{chiexpdef}
\begin{aligned}
\chi_{\epsilon,\tau,\sigma,\eta}(x,t):=&q_{\eta}(\sigma\cdot x+\nu t+\kappa)\chi_{\epsilon,\tau}(\sigma\cdot x+\nu t+\kappa),\\
=&q_{\epsilon,\tau,\sigma,\eta}(x,t)\chi_{\epsilon,\tau,\sigma}(x,t),
\end{aligned}
\end{equation}
$x\in\mathbb{R}^{n}$, $t\geq 0$, where $\chi_{\epsilon,\tau}$, and $\chi_{\epsilon,\tau,\sigma}$ are defined in Subsection \ref{smoothapprosec}. Similarly, we define $p_{\epsilon,\tau,\sigma,\eta}$ and $\rho_{\epsilon,\tau,\sigma,\eta}$. Using the properties of $q_{\eta}$, $p_{\eta}$, and $\rho_{\eta}$,  the the function $\chi_{\epsilon,\tau,\sigma,\eta}$ satisfies the following properties:
\begin{itemize}
    \item[(i)] Given $x\in \mathbb{R}^n$, it follows
	\begin{equation}\label{p1}
		\begin{split}
		    \partial_{x_j} &(\chi_{\epsilon,\tau,\sigma,\eta}^{2}(x,t))\\
			&=\sigma_j p_{\epsilon,\tau,\sigma,\eta}'(x,t)\chi_{\epsilon,\tau,\sigma}^{2}(x,t)+2\sigma_j p_{\epsilon,\tau,\sigma,\eta}(x)\chi_{\epsilon,\tau,\sigma}(x,t)\chi_{\epsilon,\tau,\sigma}'(x)\\
			&\geq \sigma_j p_{\epsilon,\tau,\sigma,\eta}'(x,t)\chi_{\epsilon,\tau,\sigma}^{2}(x)=2b\sigma_j \rho_{\epsilon,\tau,\sigma,\eta}^{2}(x,t)\chi_{\epsilon,\tau,\sigma}^{2}(x).
		\end{split}
	\end{equation}
	\item[(ii)] Since $p_{\epsilon,\tau,\sigma,\eta}(x,t)\chi'_{\epsilon,\tau,\sigma}(x,t)\leq e^{2b \tau} \chi'_{\epsilon,\tau,\sigma}$,  
	\begin{equation}\label{p2}
		\begin{split}
			\partial_{x_j}(\chi_{\epsilon,\tau,\sigma,\eta}^{2}(x,t))\leq &  2b \sigma_j p_{\epsilon,\tau,\sigma,\eta}(x,t)\chi_{\epsilon,\tau,\sigma}^{2}(x)\\
			&+2\sigma_je^{2b\tau}\chi_{\epsilon,\tau,\sigma}(x,t)\chi_{\epsilon,\tau,\sigma}'(x,t).
		\end{split}
	\end{equation}
	\item[(iii)] Let $\beta$ be a multi-index  with  $2\leq |\beta|\leq 3$, then
	\begin{equation}\label{p3}
	\begin{aligned}
	|\partial^{\beta}\big(\chi_{\epsilon,\tau,\sigma,\eta}^{2}(x,t)\big)|\lesssim & p_{\epsilon,\tau,\sigma,\eta}(x,t)\big((\chi_{\epsilon,\tau,\sigma}(x,t))^2+\chi_{\epsilon,\tau,\sigma}'(x,t)\chi_{\epsilon,\tau,\sigma}'(x,t)\\
	&+\chi_{\epsilon,\tau,\sigma}(x,t)|\chi_{\epsilon,\tau,\sigma}''(x,t)|+\chi_{\epsilon,\tau,\sigma}'(x,t)|\chi_{\epsilon,\tau,\sigma}''(x,t)|\\
	&+\chi_{\epsilon,\tau,\sigma}(x,t)|\chi_{\epsilon,\tau,\sigma}'''(x,t)|\big).
	\end{aligned}
	\end{equation}
\end{itemize}
Similar to the propagation polynomial decay studied above, Theorem \ref{expodecay} follows by using energy estimates in the equation in \eqref{ZK}. More precisely, let $\beta$ be a multi-index, by standard arguments, we get
\begin{equation}\label{differineqexp}
\begin{aligned}
\frac{1}{2}\frac{d}{dt}\int &(\partial^{\beta}u(x,t))^2 \chi_{\epsilon,\tau,\sigma,\eta}^2(x,t)\, \mathrm{d}x- \frac{1}{2}\underbrace{ \int (\partial^{\beta} u(x,t))^2\frac{\partial}{\partial t}\chi_{\epsilon,\tau,\sigma,\eta}^2(x,t)\, \mathrm{d}x}_{\mathcal{C}_1} \\
&+\underbrace{\int \partial_{x_1}\Delta \partial^{\beta} u(x,t) \partial^{\beta} u(x,t)\chi_{\epsilon,\tau,\sigma,\eta}^2(x,t)\, \mathrm{d}x}_{\mathcal{C}_2} \\
&+\underbrace{\int \partial^{\beta}\big(u\partial_{x_1}u(x,t)) \partial^{\beta}  u(x,t)\chi_{\epsilon,\tau,\sigma,\eta}^2(x,t)\, \mathrm{d}x}_{\mathcal{C}_3}=0.\\
\end{aligned}
\end{equation}
Notice that the above equation is justified for sufficiently regular solutions $u$ of \eqref{ZK}, the validity, in general, can be obtained by approximation in  $H^{(\frac{n+2}{2})^{+}}(\mathbb{R}^n)$ and our estimates below. We remark that since the function $q_{\eta}\in L^{\infty}$, the approximation argument in the proof of Theorem \ref{expodecay} can be carried out in $C([0, T]; H^{(\frac{n+2}{2})^{+}}(\mathbb{R}^n))$, i.e., we do not need to use any extra exponential decay assumption for smooth solutions $u$ of ZK. This contrasts with our proof of Theorem \ref{mainTHM}, where we need Theorem \ref{wellposweighted} to justify the validity of \eqref{differineq} and the subsequent limit in our estimates for that case.
\begin{lemma}\label{expolemma}
Let $n\geq 2$,  $\mathcal{C}_1$ and $\mathcal{C}_2$ be defined as in \eqref{differineqexp}. Then
\begin{equation}\label{eqexpodecay1}
\begin{aligned}
|\mathcal{C}_1|\lesssim &\int (\partial^{\beta}u(x,t))^2\chi_{\epsilon,\tau,\sigma,\eta}^2(x,t)\, \mathrm{d}x \\
&+\Big|\int (\partial^{\beta} u(x,t))^2 \chi_{\epsilon,\tau,\sigma}\chi_{\epsilon,\tau,\sigma}'(x,t)\, \mathrm{d}x\Big|.
\end{aligned}
\end{equation}
There exists some constant $c_{\sigma,b}>0$ such that
\begin{equation}\label{eqexpodecay2}
\begin{aligned}
\mathcal{C}_2= & c_{\sigma,b}\int |\nabla \partial^{\beta} u|^2 \chi_{\epsilon,\tau,\sigma}\big(p_{\epsilon,\tau,\sigma,\eta}'\chi_{\epsilon,\tau,\sigma}+p_{\epsilon,\tau,\sigma,\eta}\chi_{\epsilon,\tau,\sigma}' \big) \, \mathrm{d}x \\
&+R(\partial^{\beta} u,t),
\end{aligned}
\end{equation}
where 
\begin{equation}\label{eqexpodecay3}
\begin{aligned}
|R(\partial^{\beta} u,t)| \lesssim & \int (\partial^{\beta }u(x,t))^2\chi_{\epsilon,\tau,\sigma,\eta}^2(x,t)\,\mathrm{d}x \\
&+ \int (\partial^{\beta} u(x,t))^2 \big(|\chi_{\epsilon,\tau,\sigma}\chi_{\epsilon,\tau,\sigma}'|+ |\chi_{\epsilon,\tau,\sigma}'|^2 +|\chi_{\epsilon,\tau,\sigma}\chi_{\epsilon,\tau,\sigma}''|\\
&\hspace{2cm}+|\chi_{\epsilon,\tau,\sigma}'\chi_{\epsilon,\tau,\sigma}''|+|\chi_{\epsilon,\tau,\sigma}\chi_{\epsilon,\tau,\sigma}'''|\big)\,\mathrm{d}x.
\end{aligned}
\end{equation}
Moreover, let $\epsilon'>0$, $\tau'\geq 5\epsilon'$, $\tau'<\epsilon$, then
\begin{equation}\label{eqexpodecay4}
\begin{aligned}
|\mathcal{C}_3| \lesssim &\Big(\sum_{|\beta_1|\leq |\beta|+1}\|\chi_{\epsilon',\tau',\sigma}\partial^{\beta_1}u\|_{L^{\infty}}\Big)\Big(\sum_{|\beta_2|\leq |\beta|}\|\chi_{\epsilon,\tau,\sigma,\eta}\partial^{\beta_2}u\|_{L^{2}}\Big)\\
&\qquad \times\|\chi_{\epsilon,\tau,\sigma,\eta}\partial^{\beta} u\|_{L^2}.
\end{aligned}
\end{equation}
All the implicit constants \eqref{eqexpodecay1}-\eqref{eqexpodecay4} above are independent of $\eta>0$.
\end{lemma}
\begin{proof}
We notice that \eqref{eqexpodecay1} follows from the ideas in the inequality \eqref{p2}. The estimate \eqref{eqexpodecay2} follows by \eqref{p2}, \eqref{p3}, and  the same argument in the deduction of \eqref{estA2}. We remark that in this part, it is fundamental to have $\sqrt{3}\sigma_1>\sqrt{\sigma_2^2+\dots+\sigma_n^2}$. Finally, we use the Cauchy-Schwarz inequality, and Leibniz rule to deduce
\begin{equation*}
\begin{aligned}
|\mathcal{C}_3|\lesssim & \|\chi_{\epsilon,\tau,\sigma,\eta}\partial^{\beta}(u\partial_{x_1}u)\|_{L^2}\|\chi_{\epsilon,\tau,\sigma,\eta}\partial^{\beta} u\|_{L^2}\\
\lesssim &\sum_{\beta_1+\beta_2=\beta}\|\chi_{\epsilon,\tau,\sigma,\eta}(\partial^{\beta_1}u)(\partial^{\beta_2}\partial_{x_1}u)\|_{L^2}\|\chi_{\epsilon,\tau,\sigma,\eta}\partial^{\beta} u\|_{L^2}\\ \lesssim &\sum_{\beta_1+\beta_2=\beta}\|\chi_{\epsilon',\tau',\sigma}\partial^{\beta_2}\partial_{x_1}u\|_{L^{\infty}}\|\chi_{\epsilon,\tau,\sigma,\eta}\partial^{\beta_1}u\|_{L^2}\|\chi_{\epsilon,\tau,\sigma,\eta}\partial^{\beta} u\|_{L^2},
\end{aligned}
\end{equation*}
where given that $\epsilon'>0$, $\tau'\geq 5\epsilon'$, and $\tau'\leq \epsilon$, we have used (see the notation in \eqref{notationderchi})
\begin{equation*}
    \chi_{\epsilon,\tau,\sigma}^{(j)}\lesssim \chi_{\epsilon',\tau',\sigma}, \text{ for all integer } j\geq 0.
\end{equation*}
This completes the deduction of \eqref{eqexpodecay4}.
\end{proof}
Now we are in the condition to deduce the propagation of exponential decay result. 
\begin{proof}[Proof of Theorem \ref{expodecay}]
The proof follows by an inductive argument on the size of the multi-index $\beta$ in \eqref{expresult2}. But first, let us obtain some consequence of \eqref{decayiniexp}. The exponential decay condition in \eqref{decayiniexp} implies that for all $r>0$
\begin{equation}\label{proofexpdecaeq1}
\|\langle \sigma\cdot x \rangle^{r} u_0\|_{L^2(\mathcal{H}_{\{\sigma,\kappa\}})}<\infty.
\end{equation}
For a given multi-index $\beta$, let $r>|\beta|+\frac{n}{2}+2$,  it follows from \eqref{proofexpdecaeq1} and Theorem \ref{mainTHM} that for any $0<\delta_0<T$ fixed, we have
\begin{equation*}
\sup_{t\in [\delta_0,T]} \int_{\mathbb{R}^d}  \big(J^{s} u(x,t) \big)^2 \chi_{r_s,\epsilon_0,\tau_0,\sigma}^2(x,t) \, \mathrm{d}x <\infty,
\end{equation*}
where $r_s\geq r-\frac{\lfloor 2r\rfloor}{2}\geq 0$, $ 0\leq s \leq \lfloor 2r\rfloor$, any $\epsilon_0$, $\tau_0$ such that $\epsilon_0>0$, $\tau_0\geq 5\epsilon_0$, and $\chi_{r_s,\epsilon,\tau,\sigma}$ be defined by \eqref{defiweight}. In particular, since $\chi_{r_s,\epsilon_0,\tau_0,\sigma}^2\geq \chi_{\epsilon_0,\tau_0,\sigma}^2$, up to some modification of $\epsilon_0$, and $\tau_0$ (for simplicity, we will keep the same notation), we can apply Lemma \ref{lemmapropaform} to deduce that there exists a constant $c_2>0$ such that 
\begin{equation}\label{proofexpdecaeq2}
\sup_{t\in [\delta_0,T]} \int_{\mathbb{R}^d}  \big(\partial^{\gamma} u(x,t) \big)^2 \chi_{\epsilon_0,\tau_0,\sigma}^2(x,t) \, \mathrm{d}x\leq c_2,
\end{equation}
for all $|\gamma|\leq |\beta|+\frac{n}{2}+2$. Let $\delta>0$ arbitrary but fixed, by taking $\delta_0$ small in the statements above, we will assume that $\delta_0<\delta$ in \eqref{proofexpdecaeq2}. Thus, fixing $\epsilon_0$ and $\tau_0$ in \eqref{proofexpdecaeq2}, we proceed with the deduction of Theorem \ref{expodecay}. We will apply an inductive argument on the magnitude of the index $\beta$ in \eqref{expresult2} with $|\beta|=l$, $l\in \mathbb{Z}^{+}$.

\underline{\bf Case $|\beta|=0$ in \eqref{expresult2}}. We use the differential equation \eqref{differineqexp} with $\beta=0$, $\epsilon>0$, and $\tau\geq 5\epsilon$. Let us estimate the factors $\mathcal{C}_1$, $\mathcal{C}_2$, and $\mathcal{C}_3$. By Lemma \ref{expolemma},  we have $|\mathcal{C}_1|\lesssim \|u(t)\|_{L^2}=\|u_0\|_{L^2}$. Now, to estimate $\mathcal{C}_2$, using \eqref{p1}, we notice that 
\begin{equation*}
\begin{aligned}
    \int |\nabla u|^2 & \chi_{\epsilon,\tau,\sigma}\big(p_{\epsilon,\tau,\sigma,\eta}'\chi_{\epsilon,\tau,\sigma}+p_{\epsilon,\tau,\sigma,\eta}\chi_{\epsilon,\tau,\sigma}' \big) \, \mathrm{d}x\\
    &\gtrsim \sum_{|\gamma|=1}\int |\partial^{\gamma} u|^2 \rho_{\epsilon,\tau,\sigma,\eta}^2\chi_{\epsilon,\tau,\sigma}^2  \, \mathrm{d}x
\end{aligned}
\end{equation*}
and by \eqref{eqexpodecay3},
\begin{equation*}
    |R(u,t)|\lesssim \int_{\mathbb{R}^{n}}u^{2}(x,t)\chi_{\epsilon,\tau,\sigma,\eta}^{2}(x,t)\,\mathrm{d}x+\|u_0\|_{L^2}.
\end{equation*}
Also from Lemma \ref{expolemma}, we get
\begin{equation*}
    \mathcal{C}_3\lesssim \|\nabla u\|_{L^{\infty}}\|\chi_{\epsilon,\tau,\sigma,\eta} u\|_{L^2}^2.
\end{equation*}
Finally, we gather all the estimates above, together with Gronwall's inequality to get 
	\begin{equation}\label{initstep}
		\begin{split}
			\sup_{t\in[0,T]}\int_{\mathbb{R}^{n}}u^{2}(x,t)\chi_{\epsilon,\tau,\sigma,\eta}^{2}(x,t)\,\mathrm{d}x
			&\leq c_1<\infty,
		\end{split}
	\end{equation}
and
\begin{equation}\label{initstep2}
    \sum_{|\gamma|=1}\int_0^T\int |\partial^{\gamma} u(x,t)|^2 \rho_{\epsilon,\tau,\sigma,\eta}^2\chi_{\epsilon,\tau,\sigma}^2(x,t)  \, \mathrm{d}x \mathrm{d}t\leq c_1,
\end{equation}
where $c_1=(\|u_0\|_{L^2},\|u\|_{L^1_T W^{1,\infty}}, \|e^{b\sigma \cdot x}u_0\|_{\mathcal{H}_{\{\sigma,\kappa\}}})$ is independent of $\eta$. Thus, taking $\eta\to 0^{+}$, we get \eqref{expresult1}, and 
\begin{equation*}
    \sum_{|\gamma|=1}\int_0^T\int |\partial^{\gamma} u(x,t)|^2 e^{b\sigma \cdot x}\chi_{\epsilon,\tau,\sigma}^2(x,t)  \, \mathrm{d}x\, \mathrm{d} t<\infty.
\end{equation*}
Above, we have also used that since $t\in [0,T]$, $\kappa\in \mathbb{R}$, $e^{b(\sigma\cdot x+vt+\kappa)}\sim_{T,\kappa} e^{b\sigma \cdot x}$. As a further consequence of the inequality above, there exists $\delta_1<t_1<\delta$ such that
\begin{equation*}
    \sum_{|\gamma|=1}\int |\partial^{\gamma} u(x,t_1)|^2 e^{b\sigma \cdot x}\chi_{\epsilon,\tau,\sigma}^2(x,t_1)  \, \mathrm{d}x<\infty.
\end{equation*}
This completes the considerations for the case $|\beta|=0$ in \eqref{expresult2}.

\underline{\bf Case $|\beta|=l$, $l\geq 1$ in \eqref{expresult2}}. We will assume by induction hypothesis that there exists $t_{l-1}\in (\delta_0,\delta)$ such that
\begin{equation}\label{induchypho}
    \sup_{t_{l-1}\leq t\leq T}\int_{\mathbb{R}^n}(\partial^{\beta'}u(x,t))^2 e^{2b\sigma \cdot x}\chi_{\epsilon_{l-1},\tau_{l-1},\sigma}^2(x,t)\, \mathrm{d}x<\infty,
\end{equation}
for all $|\beta'|\leq |\beta|-1=l-1$, $\epsilon_{l-1}>0$, $\tau_{l-1}\geq 5 \epsilon_{l-1}$, and $\epsilon_{l-1}>\tau_0$, where $\epsilon_0>0$ and $\tau_0$ satisfy \eqref{proofexpdecaeq2}. Moreover, we will also assume by induction hypothesis
\begin{equation}\label{induchypho2}
    \sum_{|\beta'|=l}\int_{t_{l-1}}^T\int_{\mathbb{R}^n}(\partial^{\beta'}u(x,t))^2 e^{2b\sigma \cdot x}\chi_{\epsilon_{l-1},\tau_{l-1},\sigma}^2(x,t)\, \mathrm{d}x\mathrm{d}t<\infty.
\end{equation}
We notice that when $l=1$, \eqref{induchypho} and \eqref{induchypho2} correspond to \eqref{initstep} and \eqref{initstep2}, respectively. Now, by  \eqref{induchypho2} there exists $\delta_0<t_{l-1}<t_l<\delta$ such that
\begin{equation}\label{induchypho3}
    \sum_{|\beta'|=l}\int |\partial^{\beta'} u(x,t_l)|^2 e^{b\sigma \cdot x}\chi_{\epsilon_{l-1},\tau_{l-1},\sigma}^2(x,t_l)  \, \mathrm{d}x<\infty.
\end{equation}
Starting from $t_l$, we will show that $\partial^{\beta}u $, $|\beta|=l$ propagates localized exponential decay. To this aim, we let $\epsilon_l>0$, $\tau_l\geq 5\epsilon_l$, and $\epsilon_l\geq \tau_{l-1}$, we consider the differential identity \eqref{differineqexp} with weighted function $\chi_{\epsilon_l,\tau_l,\sigma,\eta}$. Let us estimate $\mathcal{C}_j$, $j=1,2,3$ for the present case. Our choice of parameter $\epsilon_{0}$, $\epsilon_{l}$, $\epsilon_{l-1}$, $\tau_{0}$, $\tau_l$, and $\tau_{l-1}$ implies that for all integers $j, j'\geq 0$ 
\begin{equation}\label{supportchiprop}
    \chi^{(j')}_{\epsilon_{l},\tau_{l},\sigma}\chi_{\epsilon_{l},\tau_{l},\sigma}^{(j)}\lesssim \chi_{\epsilon_{l-1},\tau_{l-1},\sigma}^2,
\end{equation}
and
\begin{equation}\label{supportchiprop2}
    \chi^{(j')}_{\epsilon_{l-1},\tau_{l-1},\sigma}\chi_{\epsilon_{l-1},\tau_{l-1},\sigma}^{(j)}\lesssim \chi_{\epsilon_{0},\tau_{0},\sigma}^2.
\end{equation}
Hence \eqref{supportchiprop}, \eqref{supportchiprop2} and \eqref{eqexpodecay1} yield
\begin{equation*}
\begin{aligned}
|\mathcal{C}_1|\lesssim & \|\chi_{\epsilon_l,\tau_l,\sigma,\eta}\partial^{\beta}u(t)\|_{L^2}^2+\sup_{t\in[\delta_0,T]}\|\chi_{\epsilon_0,\tau_0,\sigma}\partial^{\beta}u(t)\|_{L^2}^2 \\
\lesssim & \|\chi_{\epsilon_l,\tau_l,\sigma,\eta}\partial^{\beta}u(t)\|_{L^2}^2+c_2^2,
\end{aligned}    
\end{equation*}
where we have also used \eqref{proofexpdecaeq2}. Next, we estimate $R(\partial^{\beta}u,t)$ in the representation for $\mathcal{C}_2$ in \eqref{eqexpodecay2}. We observe that the argument used in the above estimate for $\mathcal{C}_1$ also implies
\begin{equation*}
|R(\partial^{\beta} u,t)|\lesssim \|\chi_{\epsilon_l,\tau_l,\sigma,\eta}\partial^{\beta}u(t)\|_{L^2}^2+c_2^2.
\end{equation*}
Next, we deal with $\mathcal{C}_3$. By \eqref{supportchiprop} and \eqref{eqexpodecay4}, we get
\begin{equation*}
\begin{aligned}
|\mathcal{C}_3| \lesssim &\Big(\sup_{t\in [\delta_0,T]}\sum_{|\beta_1|\leq |\beta|+1}\|\chi_{\epsilon_{l-1},\tau_{l-1},\sigma}\partial^{\beta_1}u(t)\|_{L^{\infty}}\Big)\Big(\sum_{|\beta_2|\leq |\beta|}\|\chi_{\epsilon_l,\tau_l,\sigma,\eta}\partial^{\beta_2}u\|_{L^{2}}\Big)\\
&\qquad \times\|\chi_{\epsilon_l,\tau_l,\sigma,\eta}\partial^{\beta} u\|_{L^2}.
\end{aligned}
\end{equation*}
Let us estimate the right-hand side of the inequality above. Let $\beta_1$ be a multi-index of order $|\beta_1|\leq |\beta|+1$, by using Sobolev embedding $H^{\lfloor\frac{n}{2} \rfloor+1}(\mathbb{R}^n)\hookrightarrow L^{\infty}(\mathbb{R}^n)$, Leibniz rule, \eqref{supportchiprop2}, and \eqref{proofexpdecaeq2}, we get
\begin{equation*}
\begin{aligned}
\|\chi_{\epsilon_{l-1},\tau_{l-1},\sigma}\partial^{\beta_1}u(t)\|_{L^{\infty}}\lesssim &  \|\chi_{\epsilon_{l-1},\tau_{l-1},\sigma}\partial^{\beta_1}u(t)\|_{L^{2}}+\sum_{|\gamma|=\lfloor \frac{n}{2}\rfloor+1}\|\partial^{\gamma}(\chi_{\epsilon_{l-1},\tau_{l-1},\sigma}\partial^{\beta_1}u(t))\|_{L^{2}}\\
\lesssim &  \sum_{|\gamma|\leq \lfloor \frac{n}{2}\rfloor+2+|\beta|}\|\chi_{\epsilon_0,\tau_0,\sigma}\partial^{\gamma}u(t)\|_{L^{2}}\\
\lesssim &  c_2,
\end{aligned}
\end{equation*}
where $t\in [\delta_0,T]$. Now, using the induction hypothesis \eqref{induchypho}, we have
\begin{equation*}
\begin{aligned}
\sup_{t\in [t_{l-1},T]}\sum_{|\beta_2|< |\beta|=l}\|\chi_{\epsilon_l,\tau_l,\sigma,\eta}\partial^{\beta_2}u(t)\|_{L^{2}}\leq & \sup_{t\in [t_{l-1},T]}\sum_{|\beta_2|< |\beta|}\|e^{b\sigma \cdot x}\chi_{\epsilon_{l-1},\tau_{l-1},\sigma}\partial^{\beta_2}u(t)\|_{L^{2}}\\
\lesssim & 1,
\end{aligned}    
\end{equation*}
where the implicit constant is independent of $\eta$. Collecting the previous estimates, we conclude
\begin{equation*}
\begin{aligned}
|\mathcal{C}_3|\lesssim & c_2\Big(\sum_{|\beta_2|\leq |\beta|}\|\chi_{\epsilon_l,\tau_l,\sigma,\eta}\partial^{\beta_2}u\|_{L^{2}}\Big)\|\chi_{\epsilon_l,\tau_l,\sigma,\eta}\partial^{\beta} u\|_{L^2}\\
\lesssim & c_2\Big(\sum_{|\beta_2|= |\beta|}\|\chi_{\epsilon_l,\tau_l,\sigma,\eta}\partial^{\beta_2}u\|_{L^{2}}\Big)^2\\
&+c_2\Big(\sum_{|\beta_2|< |\beta|}\|\chi_{\epsilon_l,\tau_l,\sigma,\eta}\partial^{\beta_2}u\|_{L^{2}}\Big)\Big(\sum_{|\beta_2|= |\beta|}\|\chi_{\epsilon_l,\tau_l,\sigma,\eta}\partial^{\beta_2}u\|_{L^{2}}\Big)\\
\lesssim & c_2\Big(\sum_{|\beta_2|= |\beta|}\|\chi_{\epsilon_l,\tau_l,\sigma,\eta}\partial^{\beta_2}u\|_{L^{2}}\Big)^2\\
&+c_2\Big(\sum_{|\beta_2|= |\beta|}\|\chi_{\epsilon_l,\tau_l,\sigma,\eta}\partial^{\beta_2}u\|_{L^{2}}\Big).
\end{aligned}
\end{equation*}
We can apply Young's inequality and standard estimates to get
\begin{equation*}
\begin{aligned}
|\mathcal{C}_3|
\lesssim & c_2^2 +(1+c_2)\Big(\sum_{|\beta_2|= |\beta|}\|\chi_{\epsilon_l,\tau_l,\sigma,\eta}\partial^{\beta_2}u\|_{L^{2}}^2\Big).
\end{aligned}
\end{equation*}
Finally, gathering the previous estimates in \eqref{differineqexp}, and summing over all multi-index $|\gamma|=|\beta|$, there exists a constant $c_3>0$ independent of $\eta$ such that
\begin{equation*}
\begin{aligned}
\frac{d}{dt} &\Big(\sum_{|\gamma|=|\beta|}\|\chi_{\epsilon_l,\tau_l,\sigma,\eta}\partial^{\beta}u(t)\|_{L^2}^2\Big)\\
&+\sum_{|\gamma|=|\beta|}c_{\sigma,\gamma,b}\int |\nabla \partial^{\gamma} u|^2 \chi_{\epsilon_l,\tau_l,\sigma}\big(p_{\epsilon_l,\tau_l,\sigma,\eta}'\chi_{\epsilon_l,\tau_l,\sigma}+p_{\epsilon_l,\tau_l,\sigma_l,\eta}\chi_{\epsilon_l,\tau_l,\sigma}' \big)\, \mathrm{d}x\\
\leq & c_3+c_3\Big(\sum_{|\gamma|=|\beta|}\|\chi_{\epsilon_l,\tau_l,\sigma,\eta}\partial^{\beta}u(t)\|_{L^2}^2\Big).
\end{aligned}
\end{equation*}
Thus, applying Gronwall's inequality in the above differential equation over the interval $[t_l, T]$, and using \eqref{induchypho3}, we conclude that there exists $c_3^{\ast}>0$ independent of $\eta$ such that for all $|\beta|=l$,
\begin{equation*}
    \sup_{t_{l}\leq t\leq T}\int_{\mathbb{R}^n}(\partial^{\beta}u(x,t))^2 \chi_{\epsilon_{l},\tau_{l},\sigma,\eta}^2(x,t)\, \mathrm{d}x\leq c_3,
\end{equation*}
and 
\begin{equation*}
    \sum_{|\gamma|=l+1}\int_{t_{l}}^T\int_{\mathbb{R}^n}(\partial^{\beta}u(x,t))^2\, \rho_{\epsilon,\tau,\sigma,\eta}^2\chi_{\epsilon_{l},\tau_{l},\sigma}^2(x,t)\, \mathrm{d}x\mathrm{d}t\leq c_3^{\ast}.
\end{equation*}
In the inequality above, we have also used \eqref{p1}. Thus, taking $\eta\to 0^{+}$ and using that $e^{b(\sigma\cdot x+vt+\kappa)}\sim_{T,\kappa} e^{b\sigma \cdot x}$, we deduce
\begin{equation}\label{induchypho4}
    \sup_{t_{l}\leq t\leq T}\int_{\mathbb{R}^n}(\partial^{\beta}u(x,t))^2 e^{2b\sigma \cdot x}\chi_{\epsilon_{l},\tau_{l},\sigma}^2(x,t)\, \mathrm{d}x<\infty,
\end{equation}
for all $|\beta|=l$, and
\begin{equation}\label{induchypho5}
    \sum_{|\gamma|=l+1}\int_{t_{l}}^T\int_{\mathbb{R}^n}(\partial^{\beta}u(x,t))^2 e^{2b\sigma \cdot x}\chi_{\epsilon_{l},\tau_{l},\sigma}^2(x,t)\, \mathrm{d}x\mathrm{d}t<\infty.
\end{equation}
The results in \eqref{induchypho4} and \eqref{induchypho5} complete the inductive step and thus the proof of Theorem \ref{expodecay}.
\end{proof}
\section{Propagation of polynomial and exponential  weights KdV equation}\label{sectionKdV}
In this section, we briefly discuss the main ideas leading to the proof of Theorem \ref{mainTHMKdV}, in which we establish propagation of polynomial and exponential weights for solutions of the KdV equation. The deduction of such results follows as an outcome of the arguments developed in the higher dimensional case in the proof of Theorems \ref{mainTHM} and \ref{expodecay}. We first notice that formally \eqref{sigmacond} in dimension $n=1$ reduces to $\sigma_1>0$, and thus we do not need this condition in our study of the KdV equation, which in turn simplifies the energy estimates of the dispersive term in KdV.

Let $ u \in C([0,T];H^{\frac{3}{4}^{+}}(\mathbb{R}))$ be the solutions of \eqref{KdV} with initial condition $u_0\in H^{\frac{3}{4}^{+}}(\mathbb{R})$, which also satisfies \eqref{condMainKdV}. By translation, we will assume that $\kappa=0$. Since here we deal with one spatial dimension, we will denote by $\chi_{r,\epsilon,\tau}=\chi_{r,\epsilon,\tau,1}$, see the notation in \eqref{defiweight} with $\sigma=1$, and dimension $n=1$.

Let us first discuss the propagation of fractional weights, i.e., the proof of Theorem \ref{mainTHMKdV} (i). This result depends on the fractional propagation of regularity for solutions of the KdV equation in \cite[Theorem 1.1]{KenigLinaresVega2018}. Such a result is the analog of Theorem \ref{TheorArg} with $n=1$. Now, let $\mathcal{B}^s$ be a fixed differential operator. By using the equation in \eqref{KdV}, we obtain
\begin{equation}\label{differentialiden}
\begin{aligned}
     \frac{1}{2}&\frac{d}{dt}\int (\mathcal{B}^s u)^2 \chi_{r,\epsilon,\tau}^2(x,t)\, \mathrm{d}x-\frac{1}{2}\underbrace{\int (\mathcal{B}^s u)^2 \partial_{ t}\left(\chi_{r,\epsilon,\tau}^2(x,t)\right)\, \mathrm{d} x}_{\overline{A}_1}\\
     &+\frac{3}{2}\underbrace{\int (\partial_x\mathcal{B}^s u)^2 \frac{\partial}{\partial_x}\chi_{r,\epsilon,\tau}^2(x,t)\, \mathrm{d}x}_{\overline{A}_2}-\frac{1}{2}\underbrace{\int (\mathcal{B}^s u)^2 \partial_x^3\big(\chi_{r,\epsilon,\tau}(x,t)\big)\, \mathrm{d}x}_{\overline{A}_3}\\
     &+ \underbrace{\int \mathcal{B}^s \big(u\partial_x u\big) \mathcal{B}^s u \chi_{r,\epsilon,\tau}^2(x,t)\, \mathrm{d}x}_{\overline{A}_4}=0.
\end{aligned}
\end{equation}
We remark that the above identity is formal, but it can be justified at the end of Section 3 in \cite{IsazaLinaresPonce2015}. Consequently, the proof of Theorem \ref{mainTHMKdV} follows by an inductive argument based on the estimates for $\overline{A}_1,\dots,\overline{A}_4$. We summarize these estimates in the following lemma.
\begin{lemma}
The results in Lemmas \ref{linearEstlemma1} and \ref{NonlinearEstlemma1} are valid in the one-dimensional case of the Cauchy problem \eqref{KdV}. More precisely, it follows
\begin{equation}\label{lemmaKdV}
\begin{aligned}
\overline{A}_2+\overline{A}_3= & 3r\int |\partial_x \mathcal{B}^s u|^2 \chi_{\epsilon,\tau}\big(\langle x+\nu t\rangle^{2r-2}(x+\nu t)\chi_{\epsilon,\tau,\sigma}\\
&\quad +\frac{1}{r}\langle x+\nu t\rangle^{2r}\chi_{\epsilon,\tau}' \big) \, \mathrm{d}x +R(\mathcal{B}^su,t),
\end{aligned}
\end{equation}
where $R(\mathcal{B}^su,t)$ satisfies \eqref{estR} (with $\sigma=1$). Moreover, the estimate for $\overline{A}_4$ is given by that of $A_3$ in \eqref{estnonT} if $\mathcal{B}^s=\partial_x^j$, $j=0,1,2$, and if $\mathcal{B}^s=J^s$, $s\geq 3$, $\overline{A}_4$ satisfies the same estimate of $A_3$ in \eqref{estnonT2} with $\widetilde{A}_3$ given by the $n=1$ version of \eqref{nonlEsteq}.
\end{lemma}
Given the validity of the previous lemma, the proof of the Theorem \ref{mainTHMKdV} follows the same steps in the deduction of Theorem \ref{mainTHM}. More precisely, we can use an inductive argument based on the size of the weight $r\geq \frac{l}{2}$, where $l\in\mathbb{N}_0$. A key observation is that \eqref{lemmaKdV} shows that on each iteration of the inductive argument the size of the weight $r$ decreases in the same magnitude as in the higher dimensional case of ZK. These comments enclose the discussion around the proof of Theorem \ref{mainTHMKdV} (i).

Once we have established Theorem \ref{mainTHMKdV} (i), we can adapt the ideas in Theorem \ref{expodecay} to deduce the propagation of exponential weights in Theorem \ref{mainTHMKdV} (ii). To see this, we use the notation  $\chi_{\epsilon,\tau,\eta}(x,t)=\chi_{\epsilon,\tau,1,\eta}(x,t)$, where $\chi_{\epsilon,\tau,1,\eta}(x,t)$ is defined as in \eqref{chiexpdef} in dimension $n=1$, and $\sigma=1$. Similarly, we define $p_{\epsilon,\tau,\eta}=p_{\epsilon,\tau,1,\eta}$, see also the notation  before \eqref{chiexpdef}. Thus, taking an integer $j\geq 0$, we formally get from the equation in \eqref{KdV} 
\begin{equation}\label{differentialiden2}
\begin{aligned}
     \frac{1}{2}&\frac{d}{dt}\int (\partial_x^j u)^2 \chi_{\epsilon,\tau,\eta}^2(x,t)\, \mathrm{d}x-\frac{1}{2}\underbrace{\int (\partial_x^j u)^2 \frac{\partial }{\partial t}\chi_{\epsilon,\tau,\eta}^2(x,t)\, \mathrm{d} x}_{\overline{\mathcal{C}}_1}\\
     &+\frac{3}{2}\underbrace{\int (\partial_x^{j+1} u)^2 \frac{\partial}{\partial_x}\chi_{\epsilon,\tau,\eta}^2(x,t)\, \mathrm{d}x}_{\overline{\mathcal{C}}_2}-\frac{1}{2}\underbrace{\int (\partial_x^j u)^2 \partial_x^3\big(\chi_{\epsilon,\tau,\eta}(x,t)\big)\, \mathrm{d}x}_{\overline{\mathcal{C}}_3}\\
     &+ \underbrace{\int \partial_x^j \big(u\partial_x u\big) \partial_x^j u \chi_{\epsilon,\tau,\eta}^2(x,t)\, \mathrm{d}x}_{\overline{\mathcal{C}}_4}=0.
\end{aligned}
\end{equation}
Once again the validity of the above results is granted from the fact that the function $\chi_{\epsilon,\tau,\eta}$ and its derivatives are bounded as well as an approximation argument by smooth solutions of KdV. The key estimates for \eqref{differentialiden2} are summarized in the following lemma.
\begin{lemma}\label{lemmaexpKdV}
The results of Lemma \ref{expolemma} are valid in the one-dimensional case of the Cauchy problem \eqref{KdV}. More precisely, it follows
\begin{equation*}
\begin{aligned}
\overline{\mathcal{C}}_2+\overline{\mathcal{C}}_3= & \frac{3}{2}\int |\partial_x^{j+1} u|^2 \chi_{\epsilon,\tau}\big(pm_{\epsilon,\tau,\eta}\chi_{\epsilon,\tau}+2p_{\epsilon,\tau,\eta}\chi_{\epsilon,\tau}' \big) \, \mathrm{d}x \\
&+R(\partial_x^j u,t),
\end{aligned}
\end{equation*}
where $R(\partial_x^j u,t)$ satisfies \eqref{eqexpodecay3} (with $\sigma=1$ and $n=1$). Moreover, $\overline{\mathcal{C}}_4$ is bounded exactly as in the estimate for $\mathcal{C}_3$ in \eqref{eqexpodecay4}. 
\end{lemma}
It is worth to remark that the estimate for $\overline{\mathcal{C}}_4$ (which is the one-dimensional estimate for $\mathcal{C}_3$ in \eqref{eqexpodecay4}) requires to bound the following factor
\begin{equation*}
\sum_{m\leq j+1}\|\chi_{\epsilon,\tau}\partial^{m}_xu\|_{L^{\infty}},    
\end{equation*}
where $\epsilon>0$, $\tau\geq 5\epsilon$. Such estimate is controlled by using Sobolev embedding and the results of Theorem \ref{mainTHMKdV} (i), which imply that we have arbitrary regularity on the region determined by $\chi_{\epsilon,\tau}$ at later times for solutions of \eqref{KdV} with initial condition $u_0\in H^{\frac{3}{4}^{+}}(\mathbb{R})$ with localized exponential decay \eqref{condo1dexp}. This is why we first deduced Theorem \ref{mainTHMKdV} (i).  Consequently, the proof of Theorem \ref{mainTHMKdV} (ii) follows by Lemma \ref{lemmaexpKdV}, the same inductive argument and the ideas in the proof of Theorem \ref{expodecay}, in which using \eqref{differentialiden2}, we obtain some uniform estimates that allow us to take $\eta \to 0$ to get the desired result. This completes the discussion about the proof of Theorem \ref{mainTHMKdV}. 
\appendix
\section{Well-posedness in weighted spaces}\label{AppendixWellposs} 
This part concerns the deduction of the local well-posedness result in polynomial weighted spaces described in Theorem \ref{wellposweighted}. To deduce Theorem \ref{wellposweighted} when $n\geq 3$, we will adapt the ideas used for the fractional KdV equation in \cite{CunhaRiano2022}, which can be adapted to the ZK equation in any spatial dimension. 
\subsection{Preliminaries}
We denote by $\{S(t)\}$ the group of solutions associated to the linear equation $\partial_t u+\partial_{x_1}\Delta u=0$. 
\begin{lemma}\label{decaylineareq}
Let $n\geq 1$ be integer, $r\in \mathbb{R}^{+}$, and $f\in H^{2r}(\mathbb{R}^n
)\cap L^2(|x|^{2 r}\, dx)$. Then
\begin{equation}
    \|\langle x \rangle^r S(t) f\|_{L^2}\lesssim \langle t \rangle^{r}\big(\|J^{2 r}f\|_{L^2}+\|\langle x \rangle^r f\|_{L^2}\big).
\end{equation}    
\end{lemma}
\begin{proof}
The proof follows by setting $a=2$ in the arguments in the deduction of \cite[Lemma 1.2 (i)]{CunhaRiano2022}. We remark that the fact that the dispersion $\partial_{x_1}\Delta$ is a local operator implies that the solution of the linear equation $\partial_t u+\partial_{x_1}\Delta u=0$ propagates polynomial weights of arbitrary size $r>0$ (what is more, it propagates exponential weights), which contrast with the fractional case studied in  \cite{CunhaRiano2022}.
\end{proof}
\begin{lemma}\label{nonlinearestimlemm}
Let $T>0$, $r\in \mathbb{R}^{+}$. Let
\begin{equation}\label{extracond}
    s>\frac{1}{2}\Big(\frac{n}{2}+2 \Big)+\Big(\frac{1}{4}\Big(\frac{n}{2}+2 \Big)^2-\frac{n}{2}\Big)^{\frac{1}{2}}.
\end{equation}
Consider
\begin{equation*}
g\in L^{\infty}([0,T];H^{s}(\mathbb{R}^n))\cap L^{1}((0,T);L^{\infty}(\mathbb{R}^n))\cap L^{\infty}((0,T);L^2(|x|^{2r}\, dx)).
\end{equation*}
Then
\begin{equation*}
    \nabla (g^2) \in L^{1}((0,T);L^2(|x|^{2r_1}\, dx)),
\end{equation*}
where $r_1>r$ is given by
\begin{equation}\label{extradecay}
r_1:=\frac{(s-1)(4s-n)}{2s^2}r.   
\end{equation}
\end{lemma}
\begin{proof}
Lemma \ref{nonlinearestimlemm} is a particular case of \cite[Corollary 4.5]{CunhaRiano2022}. It is worth mentioning that the condition \eqref{extracond} assures that $r_1>r$.
\end{proof}
We consider the following approximation of the weight $\langle x\rangle$ used in \cite{FonsecaPonce2011}. Let $N\in \mathbb{Z}^{+}$, we introduce the truncated weights $\tilde{w}_N : \mathbb{R} \rightarrow \mathbb{R}$ such that 
 \begin{equation}
 \tilde{w}_{N}(x)=\left\{\begin{aligned} 
 &\langle x \rangle, \text{ if } |x|\leq N, \\
 &2N, \text{ if } |x|\geq 3N
 \end{aligned}\right.
 \end{equation}
in such a way that $\tilde{w}_N(x)$ is smooth and non-decreasing in $|x|$ with $\tilde{w}'_N(x) \leq 1$ for all $\neq 0$  and there exists a constant $c$ independent of $N$ from which $|\tilde{w}^{(j)}_N(x)|   \leq |\partial_x^{j}(\langle x \rangle)|$, $j=2,3$. We define the $n$-dimensional weights 
\begin{equation}\label{intro2}
w_N(x)=\tilde{w}_N(|x|), \text{ where } |x|=\sqrt{x_1^2+\dots+x_n^2}.
\end{equation}
Consequently, for fixed $0<r<1$, the definition of the $\omega_N$ yields $|\partial^{\alpha}( w_{N}^{r}) | \lesssim 1$, for all multi-index $1\leq |\alpha|\leq 3$, where the implicit constant is independent of $N$.

We use the following well-posedness result, which is obtained by the classical parabolic regularization argument.
\begin{lemma}\label{comwellp}
Let $s>\frac{n}{2}+1$. Then for any $u_0 \in H^s(\mathbb{R}^n)$, there exist $T=T(\left\|u_0\right\|_{H^s})>0$ and a unique solution $u\in C([0,T]; H^s(\mathbb{R}^n))$ of the IVP \eqref{ZK}. In addition, the flow-map $u_0 \mapsto u(t)$ is continuous in the $H^s$-norm. Moreover, the existence time is independent of the regularity in the sense that if $u_0\in H^{s'}(\mathbb{R}^n
)$ with $s'\geq s>\frac{n}{2}+1$, and $u\in C([0,T_{s'}(u_0)],H^{s'}(\mathbb{R}^n))$ solves \eqref{ZK}, then $u$ can be extended, if necessary, to the interval $[0,T_{s}(u_0)]$, with $u_0$ viewed as an element in $H^s(\mathbb{R}^n)$, i.e., the existence times in $H^{s'}(\mathbb{R}^n)$ and $H^{s}(\mathbb{R}^n)$ are the same.
\end{lemma}
We require the following result concerning the persistence of the class $\mathcal{S}(\mathbb{R}^n)$ through the flow of the IVP \eqref{ZK}, where the time of existence is given by that provided in Lemma \ref{comwellp}, i.e., the time depends on the $H^s$-norm for some $s>\frac{n}{2}+1$. 
\begin{claim}\label{claimSmooth}
Let $u_0\in \mathcal{S}(\mathbb{R}^n)$, and $s>\frac{n}{2}+1$. Then there exists $T=T(\|u_0\|_{H^s})>0$, and a unique solution $u\in C([0,T];\mathcal{S}(\mathbb{R}^n))$ of \eqref{ZK} with initial condition $u_0$.
\end{claim}
\begin{proof}
Given $u_0\in \mathcal{S}(\mathbb{R}^d)$, and $s>\frac{n}{2}+1$, by Lemma \ref{comwellp} there exist $T=T(\|u_0\|_{H^s})>0$, and $u\in C([0,T];H^{\infty}(\mathbb{R}^n))$ solution of \eqref{ZK} with initial condition $u_0$, where $H^{\infty}(\mathbb{R}^n)=\bigcap_{s>0}H^s(\mathbb{R}^n)$. We will show that 
\begin{equation}\label{smoothdecay}
  u\in C([0,T];L^2(|x|^{2r}\, dx))  
\end{equation}
for all $r>0$. Once \eqref{smoothdecay} has been established, given a multi-index $\beta$, and positive numbers $r,m$, by using interpolation inequality (see \cite[Lemma 4]{NahasPonce2009} and \cite[Lemma 1]{FonsecaPonce2011}), there exist $r_1>|\beta|+m$, $r_2>r$  such that
\begin{equation}
\begin{aligned}
    \|\langle x \rangle^{r} \partial^{\beta}u\|_{H^m}\lesssim & \|J^{|\beta|+m}(\langle x \rangle^{r} u)\|_{L^2}\\
    \lesssim & \|J^{r_1} u\|_{L^2}+\|\langle x \rangle^{r_2} u\|_{L^2}.
\end{aligned}
\end{equation}
Since the right-hand side of the above inequality is bounded and continuous by assumption \eqref{smoothdecay} and the fact that $u\in C([0, T]; H^{\infty}(\mathbb{R}^n))$,  the above inequality and using that $m>0$ is arbitrary imply that $\langle x \rangle^{r} \partial^{\beta}u\in C([0, T]; H^{\infty}(\mathbb{R}^n))$ for all multi-index $\beta$ and $r>0$, thus, $u\in C([0, T];\mathcal{S}(\mathbb{R}^n))$.

It only remains to deduce \eqref{smoothdecay}. We consider two cases: $0<r<\frac{1}{2}$, and $r\geq \frac{1}{2}$.

\underline{Case $0<r<\frac{1}{2}$}.  Let $w$ be a smooth function to be specified later. By multiplying the equation in \eqref{ZK} by $w^2 u$ and integrating on the spatial variable, we get
\begin{equation}\label{weightedEneres1}
\begin{aligned}
\frac{1}{2}\frac{d}{dt}\int \big(w u(x,t)\big)^2\, \mathrm{d}x+\int \partial_{x_1}\Delta u (w^{2} u)\, \mathrm{d}x+\int u\partial_{x_1} u (w^{2} u)\, \mathrm{d}x=0.
\end{aligned}    
\end{equation}
Integration by parts yields
\begin{equation}\label{weightedEnereseq1}
    \begin{aligned}
    \int \partial_{x_1}\Delta u (w^{2} u)\, \mathrm{d}x=&\frac{1}{2}\int |\nabla u|^2 \partial_{x_1}(w^2) \, \mathrm{d}x-\frac{1}{2}\int |u|^2 \partial_{x_1}\Delta(w^2) \, \mathrm{d}x\\
    &+\int \nabla u \cdot \nabla (w^2) \partial_{x_1} u \, \mathrm{d}x. 
    \end{aligned}
\end{equation}
The nonlinear term satisfies
\begin{equation}\label{weightedEnereseq2}
    \Big|\int u\partial_{x_1} u (w^{2} u)\, \mathrm{d}x\Big|\leq \|\partial_{x_1}u\|_{L^{\infty}}\|w u\|_{L^2}^2.
\end{equation}
If $w=w_N^r$ with $0<r<\frac{1}{2}$, the properties of $w_N$ in \eqref{intro2} imply $|\nabla (w^{2r}_N)|, |\partial_{x_1}\Delta (w^{2r}_N)| \lesssim 1$, where the implicit constant is independent of $N\geq 1$. Thus, by collecting these facts and the previous identities, we get
\begin{equation*}
        \frac{d}{dt}\|w_N^r u(t)\|_{L^2}^2\leq c_0\|u(t)\|_{H^1}^2+\|\partial_{x_1}u(t)\|_{L^{\infty}}\|w_N^r u(t)\|_{L^2}^2
\end{equation*}
for some $c_0>0$ independent of $N\geq 1$. Gronwall's inequality then shows
\begin{equation*}
    \begin{aligned}
    \|w_N^r u(t)\|_{L^2}^2\leq & \big(\|w_N^r u_0\|_{L^2}^2+c_1t\big)e^{\int_0^t\|\partial_{x_1}u(s)\|_{L^{\infty}}\, ds}\\
    \leq & \big(\|\langle x \rangle^r u_0\|_{L^2}^2+c_1t\big)e^{\int_0^t\|\partial_{x_1}u(s)\|_{L^{\infty}}\, ds},
    \end{aligned}
\end{equation*}
where $c_1=c_1(\|u\|_{L^{\infty}_TH^s})$ is independent of $N\geq 1$. By taking $N\to \infty$, we get $u\in L^{\infty}([0,T];L^{2}(|x|^{2r}\, dx))$.  To prove continuity, we first notice that the fact that $u\in C([0, T]; H^s(\mathbb{R}^n))$ implies weak continuity, i.e., $u\in C_{w}([0, T]; L^{2}(|x|^{2r}\, dx))$. On the other hand, 
\begin{equation*}
\begin{aligned}
\|\langle x \rangle^{r}(u(t)-u_0)\|_{L^2}^2=&\|\langle x \rangle^{r}u(t)\|_{L^2}^2+\|\langle x \rangle^{r}u_0\|_{L^2}^2-2\int \langle x \rangle^{2r}u(t)u_0\, \mathrm{d}x \\
\leq & \big(\|\langle x \rangle^r u_0\|_{L^2}^2+c_1t\big)e^{\int_0^t\|\partial_{x_1}u(s)\|_{L^{\infty}}\, ds}+\|\langle x \rangle^{r}u_0\|_{L^2}^2\\
&-2\int \langle x \rangle^{2r}u(t)u_0\, \mathrm{d}x.
\end{aligned}    
\end{equation*}
Hence, weak continuity and taking $t\to 0$ in the above inequality imply continuity at the origin of the function $t\mapsto \langle x \rangle^{r}u(t)$. Thus, by using the fact that the equation in \eqref{ZK} is invariant under the transformations $(x,t)\mapsto (x,t+\tau)$ and $(x,t)\mapsto (-x,\tau-t)$, right continuity at the origin extends to continuity on the whole time interval $[0, T]$. This completes the deduction of \eqref{smoothdecay} when $0<r<\frac{1}{2}$.

\underline{Case $r\geq \frac{1}{2}$}. Let $0<r_1<\frac{1}{2}$ fixed and $s_1>\frac{n}{2}+1$ such that \eqref{extracond} holds. By the previous step, we know $u\in C([0,T];H^{\infty}(\mathbb{R}^n)\cap L^2(|x|^{2r_1}\,dx))$. Thus, by Lemma \ref{nonlinearestimlemm}, it follows $\nabla u\in L^1((0,T);L^2(|x|^{2r_2}\, dx))$, where $r_2=\frac{(s_1-1)(4s_1-d)}{2s_1^2}r_1>r_1$. Consequently, it follows from Lemma \ref{decaylineareq} that
\begin{equation*}
    \int_0^t S(t-\tau)(u\partial_{x_1}u)(\tau)\, d \tau \in C([0,T];L^2(|x|^{2r_2}\, dx)).
\end{equation*}
(The continuity follows from the continuity of the function $t\in[0, T]\mapsto \langle x\rangle^{r_1}u(t)$). Given that Lemma \ref{decaylineareq} implies $S(t)u_0\in C([0,T];L^{2}(|x|^{2r_2}\, dx))$, using the integral formulation of \eqref{ZK}, we conclude
\begin{equation*}
    \begin{aligned}
    u(t)=S(t)u_0-\int_0^t S(t-\tau)(u\partial_{x_1}u)(\tau)\, d \tau\in  C([0,T];L^2(|x|^{2r_2}\, dx)).
    \end{aligned}
\end{equation*}
Given the validity of the above persistence result in weighted spaces, we can apply the same line of arguments developed above with $r_3=\frac{(s_1-1)(4s_1-n)}{2s_1^2}r_2>r_2$. Consequently, an iterative argument on the sequence of weights $r_j=\frac{(s_1-1)(4s_1-n)}{2s_1^2}r_{j-1}>r_{j-1}$, $j=2,3,\dots$ yields \eqref{smoothdecay} for all $r\geq\frac{1}{2}$. The proof is complete.
\end{proof}
\subsection{Proof of Theorem \ref{wellposweighted}}
We will deduce LWP in the class $H^s(\mathbb{R}^n)\cap L^2(|x_j|^{2r}\, dx)$, with $j=1,\dots, n$ as the considerations for $H^s(\mathbb{R}^n)\cap L^2((|x_1|^{2r_1}+\dots+|x_n|^{2r_n})\, dx)$ follow by our same arguments applied to each variable individually. Given $r>0$, $s\geq \max\{(\frac{n}{2}+1)^{+},2r\}$, let $u_0\in H^s(\mathbb{R}^n)\cap L^2(|x_j|^{2r}\, dx)$. By Lemma \ref{comwellp}, there exist a time $T=T(\|u_0\|_{H^s})>0$, and a unique solution $u\in C([0,T];H^s(\mathbb{R}^d))$ of the IVP \eqref{ZK} with initial condition $u_0$. Thus, to deduce Lemma \ref{wellposweighted}, it remains to show
\begin{equation}\label{Contdecay}
    u\in C([0,T];L^2(|x_j|^{2r}\, dx)),
\end{equation}
and the continuous dependence in the norm $L^2(|x_j|^{2r}\, dx)$. Let us first establish \eqref{Contdecay}.  By density, the continuity of the flow map data-to-solution in Lemma \ref{wellposweighted}, and the proof of Claim \ref{claimSmooth}, there exist a sequence of initial conditions $u_{0,k}\in \mathcal{S}(\mathbb{R}^n)$ with corresponding solution $u_k\in C([0,T];\mathcal{S}(\mathbb{R}^n))$  such that
\begin{equation*}
\begin{aligned}
& u_{0,k}\xrightarrow[k \to \infty]{} u_0 \qquad \text{ in } H^s(\mathbb{R}^n)\cap L^2(|x_j|^{2r}\, dx),\\
& u_k \xrightarrow[k \to \infty]{} u \qquad \quad \text{ in } C([0,T];H^s(\mathbb{R}^n)).
\end{aligned}
\end{equation*}
To get \eqref{Contdecay}, we will use identity \eqref{weightedEneres1} with $w=\langle x_j\rangle^{r}$. Since at this point that identity is not justified for the function $u$, we use instead the sequence $u_k$ of smooth solutions with fast decay.  We first observe that \eqref{weightedEnereseq1} with $w=\langle x_j\rangle^{r}$ and $r>\frac{1}{2}$ satisfies
\begin{equation}\label{weightseq1}
    \begin{aligned}
    \Big|\int \partial_{x_1}\Delta u_k (\langle x_j \rangle^{2r} u_k)\, \mathrm{d}x\Big|\lesssim & \|\langle x_j\rangle^{r-\frac{1}{2}}\nabla u_k\|_{L^2}^2+\|\langle x_j\rangle^{r} u_k\|_{L^2}^2\\
    \lesssim & \|J(\langle x_j\rangle^{r-\frac{1}{2}} u_k)\|_{L^2}^2+\|\langle x_j\rangle^{r} u_k\|_{L^2}^2\\
       \lesssim & \|J^{2r} u_k\|_{L^2}^2+\|\langle x_j\rangle^{r} u_k
       \|_{L^2}^2,
    \end{aligned}
\end{equation}
where we have used the interpolation inequality Lemma \ref{InterpLemma} with $\sigma=e_j$, $\omega=0$. In particular, notice that the last line in \eqref{weightseq1} can be directly justified if $0<r\leq \frac{1}{2}$. Thus, the above estimate, \eqref{weightedEnereseq2} and the differential identity \eqref{weightedEneres1}  yield
\begin{equation*}
\begin{aligned}
\frac{d}{dt}\|\langle x_j\rangle^r u_k(t)\|_{L^2}^{2}\leq c_0 \|J^{2r} u_k(t)\|_{L^2}^2+c_0(1+\|\partial_{x_1}u_k(t)\|_{L^{\infty}})\|\langle x_j\rangle^{r} u_k(t)\|_{L^2}^2,
\end{aligned}    
\end{equation*}
for some constant $c_0>0$ independent of $k$. Hence Gronwall's inequality allows us to deduce
\begin{equation*}
    \begin{aligned}
    \|\langle x_j \rangle^n u_k(t)\|_{L^2}^2\leq \Big(\|\langle x_j \rangle^n u_{0,k}\|_{L^2}+ c_0 t(\sup_{t\in [0,T]}\|u_k(t)\|_{H^{2r}}) \Big)e^{c_1\int_0^t (1+\|\partial_{x_1}u_k(\tau)\|_{L^{\infty}})\, d\tau}.
    \end{aligned}
\end{equation*}
For some constants $c_0,c_1>0$ independent of $k$. By Sobolev embedding $\|\partial_{x_1}u_k\|_{L^{\infty}}\lesssim \|u\|_{H^s}$, then by taking $k\to \infty$ in the above inequality, we infer
\begin{equation*}
    \begin{aligned}
u\in L^{\infty}([0,T];L^2(|x|^{2r}\, dx)).
    \end{aligned}
\end{equation*}
The continuity of $u$ (that is \eqref{Contdecay}) can be deduced by similar arguments in the proof of Claim \ref{claimSmooth}. Finally, the continuous dependence in $L^2(|x|^{2r}\, dx)$ follows from this same property in $H^s$ given by Lemma \ref{comwellp}, together with approximation by continuous smooth solutions (Claim \ref{claimSmooth}), and similar arguments as above. The proof is complete.
\section{Pseudo-Differential Operators}\label{Appendix} 
In this section, we prove the continuity and asymptotic decomposition of the class of pseudo-differential operators introduced in Section \ref{pseudoSection}. These results follow by similar considerations in \cite{Hormander2007} (see also \cite{AlinhacGerard2007,Raymond1991,Taylor1981}). However, for the sake of completeness, and to clarify the dependence of the constants involved, we present the proof of Lemma \ref{contiprop} and Proposition \ref{PO1} here.
\begin{proof}[Proof of Lemma \ref{contiprop}]
We divide the proof into two main cases.

\underline{\sc Case $m=0$ and $q=0$.} In this case, $a$ is a symbol in the Kohn-Nirenberg class see \eqref{kn}. In detail, $a\in \mathbb{S}_{\sigma,\omega}^{0,0}\subseteq\mathbb{S}^{0}$. 
Nevertheless,  for $a\in \mathrm{OP}\mathbb{S}^{0}$ is well known  the continuity of operators $\Psi_{a}$ in the Lebesgue spaces $L^{p}(\mathbb{R}^{n})$ for $p\in (1,\infty).$ 
 Thus, 
\begin{equation*}
	\|\Psi_{a}f\|_{L^{p}}\leq c\|f\|_{L^{p}},\quad \mbox{for}\quad f\in\mathcal{S}(\mathbb{R}^{n}).
\end{equation*}
for  a more detailed exposition and results on the class $\mathbb{S}^{0}$ see, e.g., \cite{Hormander2007} or \cite{Taylor1981}.

\underline{\sc Case $m,q\in\mathbb{R}$ with\,  $m,q\neq 0$.} Let $f\in\mathcal{S}(\mathbb{R}^{n})$, then 
\begin{equation*}
	\begin{split}
		\|\Psi_{a}f\|_{L^{p}}&=\sup_{g\in\mathcal{S}(\mathbb{R}^{n}),\, \|g\|_{L^{p'}}\leq 1}\left|\int_{\mathbb{R}^{n}} g(x)(\Psi_{a}J^{-m}J^{m}f)(x)\,\mathrm{d}x\right|\\
		&=\sup_{g\in\mathcal{S}(\mathbb{R}^{n}),\, \|g\|_{L^{p'}}\leq 1}\left|\int_{\mathbb{R}^{n}}\langle \sigma\cdot x+\omega\rangle^{q}J^{m}f(x)\left(\frac{\overline{(J^{-m}\Psi_{a^{*}} \overline{g})(x)}}{\langle \sigma\cdot x+\omega\rangle^{q}}\right)\mathrm{d}x\right|\\
		&=\sup_{g\in\mathcal{S}(\mathbb{R}^{n}),\,\|g\|_{L^{p'}}\leq 1}\left|\int_{\mathbb{R}^{n}}\langle \sigma\cdot x+\omega\rangle^{q} J^m f(x)(\Psi_{b}g)(x)\mathrm{d}x\right|,
	\end{split}
\end{equation*}
where $\Psi_{b}\in\mathrm{OP}\mathbb{S}_{\sigma,\omega}^{0,0}\subseteq\mathrm{OP}\mathbb{S}^{0},$ and $\Psi_{a^{\ast}}\in \mathrm{OP}\mathbb{S}_{\sigma,\omega}^{m,q}$ is the adjoint pseudo-differential operator of  $\Psi_a$, which is defined as in the classical theory of pseudo-differential operators, see \cite{Hormander2007}. Thus, combining H\"older's inequality and the continuity of $\Psi_{b}$, we obtain 
\begin{equation*}
\begin{aligned}
\|\Psi_{a}f\|_{L^{p}}&\leq \sup_{g\in\mathcal{S}(\mathbb{R}^{n}),\|g\|_{L^{p'}}\leq 1}\|	\langle \sigma\cdot x+\omega\rangle^{q}J^m f\|_{L^{p}}\|\Psi_{b}g\|_{L^{p'}}\\
&\leq c\|	\langle \sigma\cdot x+\omega\rangle^{q} J^m f\|_{L^{p}}.     
\end{aligned}
\end{equation*}
The proof of the continuity results in Lemma \ref{contiprop} is complete.
\end{proof}
Next, we deduce the asymptotic expansion formula in Proposition \ref{PO1}.

\begin{proof}[Proof of Proposition \ref{PO1}]
Given  $a\in \mathbb{S}_{\sigma,\omega}^{m_1,q_1}$ and $b\in\mathbb{S}_{\sigma,\omega}^{m_2,q_2}$, it is enough to show $\Psi_c=\Psi_a \Psi_b$ where 
\begin{equation*}
c \approx \sum_{|\beta|\geq 0 }\frac{1}{(2\pi i)^{|\beta|} \beta !}(\partial^{\beta}_{\xi}a)\cdot(\partial^{\beta}_{x}b)
\end{equation*}
for all $N \geq 2$. Once we have established the previous asymptotic expansion, the desired result follows by subtracting the expansion of $\Psi_a \Psi_b$ from that of $\Psi_b \Psi_a$. 

Let $\phi \in C^{\infty}_c(\mathbb{R}^n)$ with $\phi \equiv 1$ in a neighborhood of the origin. For each integer $k \geq 2$, we define
\begin{equation*}
\begin{aligned}
&a_k(x,\xi)=\phi\left(\frac{x}{k}\right
)\phi\left(\frac{\xi}{k}\right)a(x,\xi), \\
&b_k(x,\xi)=\phi\left(\frac{x}{k}\right
)\phi\left(\frac{\xi}{k}\right
) b(x,\xi),
\end{aligned}
\end{equation*}
and let $\Psi_{a_k}$ and $\Psi_{b_k}$ be the associated pseudo-differential operators, respectively. Then, since $a_k$ and $b_k$ are Schwartz functions, we have
\begin{equation}
\begin{aligned}
\Psi_{a_k} \Psi_{b_k} f(x)=\int_{\mathbb{R}^{n}}c(x,\xi) \widehat{f}(\xi)e^{2\pi i x\cdot \xi} \, \mathrm{d}\xi,
\end{aligned}
\end{equation}
where
\begin{equation}
\begin{aligned}
c(x,\xi)=\int_{\mathbb{R}^{n}\times \mathbb{R}^n} a_k(x,\xi-\eta)b_k(x-y,\xi)e^{-2\pi iy \cdot \eta} \, \mathrm{d}y \mathrm{d}\eta.
\end{aligned}
\end{equation}
Now, by Taylor's formula, for each $N \geq 1$,
\begin{equation}\label{eqconm0}
\begin{aligned}
a_k(x,\xi-\eta)&= \sum_{|\beta|< N}\frac{(-\eta)^{\beta}}{\beta!}\partial^{\beta}_{\xi}a_k(x,\xi)+\sum_{|\beta|=N} \frac{|\beta| (-\eta)^{\beta}}{\beta!} R_{1,\beta}(x,\xi,\eta),\\
b_k(x-y,\xi)&= \sum_{|\beta|< N}\frac{(-y)^{\beta}}{\beta!}\partial^{\beta}_{x}b_k(x,\xi)+\sum_{|\beta|=N} \frac{|\beta| (-y)^{\beta}}{\beta!} R_{2,\beta}(x,y,\xi),
\end{aligned}
\end{equation}
where the remainders of the previous expansions are given by
\begin{equation}
\begin{aligned}
R_{1,\beta}(x,\xi,\eta)&=\int_0^1 (1-\mathbf{s})^{|\beta|-1}\partial^{\beta}_{\xi}a_{k}(x,\xi-\mathbf{s}\eta) \, \mathrm{d}\mathbf{s},\\
R_{2,\beta}(x,y,\xi)&= \int_0^1 (1-\mathbf{s})^{|\beta|-1}\partial^{\beta}_{x}b_{k}(x-\mathbf{s}y,\xi) \, \mathrm{d}\mathbf{s}.
\end{aligned}
\end{equation}
The previous oscillatory integral is understood in the sense
\begin{equation}\label{eqconm1.1}
\begin{aligned}
\lim_{\epsilon \to 0} \int  e^{-2\pi i y\cdot \eta} \frac{y^{\alpha}}{\alpha!}\frac{\eta^{\beta}}{\beta!} \phi(\epsilon y)\phi(\epsilon \eta)\, \mathrm{d} y  \mathrm{d} \eta. 
\end{aligned}
\end{equation}
Thus
\begin{equation}\label{eqconm2}
\begin{aligned}
c(x,&\xi)\\
=&\sum_{|\beta|<N} \frac{(-i)^{|\beta|}}{(2\pi)^{|\beta|}\beta!} \partial_x^{\beta}a_k(x,\xi)\partial_{\xi}^{\beta}b_k(x,\xi)\\
&+\sum_{\substack{|\beta|<N \\ |\alpha|=N}} \Big( \int e^{-2\pi i y \cdot \eta} |\alpha|\frac{(-\eta)^{\beta}}{\beta!}\frac{(-y)^{\alpha}}{\alpha!}R_{2,\alpha}(x,y,\xi)\, \mathrm{d} y \mathrm{d} \eta \Big) \partial_{\xi}^{\beta}a_k(x,\eta)\\
&+\sum_{\substack{|\beta|<N \\ |\alpha|=N}} \Big( \int e^{-2\pi i y \cdot \eta} |\alpha|\frac{(-y)^{\beta}}{\beta!}\frac{(-\eta)^{\alpha}}{\alpha!}R_{1,\alpha}(x,\xi,\eta)\, \mathrm{d} y \mathrm{d} \eta \Big) \partial_x^{\beta}b_k(x,\eta)\\
&+\sum_{\substack{|\beta|=N \\ |\alpha|=N}} \int e^{-2\pi i y \cdot \eta} N^2\frac{(-\eta)^{\beta}}{\beta!}\frac{(-y)^{\alpha}}{\alpha!}R_{1,\beta}(x,\xi,\eta)R_{2,\alpha}(x,y,\xi)\, \mathrm{d} y \mathrm{d} \eta.
\end{aligned}
\end{equation}
We emphasize that the second, third, and fourth integrals of the above expression are defined by a limit process as in \eqref{eqconm1.1}. Let us further clarify this statement. We use the inequality $\langle z-z_1 \rangle^s \lesssim \langle z\rangle^s \langle z_1 \rangle^{|s|}$ for any $s\in \mathbb{R}$ to deduce
\begin{equation}\label{eqconm3}
\begin{aligned}
|\partial_{\eta}^{\gamma}R_{1,\beta}(x,\xi,\eta)|&\lesssim \langle \sigma \cdot x+\omega \rangle^{q_1}\langle \xi \rangle^{m_1-|\beta|-|\gamma|}\langle \eta \rangle^{|m_1-|\beta|-|\gamma||}\\
&\lesssim \langle \sigma \cdot x+\omega \rangle^{q_1}\langle \xi \rangle^{m_1-|\beta|-|\gamma|}(1+|\eta|^2+|y|^2 )^{\frac{1}{2}|m_1-|\beta|-|\gamma||}
\end{aligned}
\end{equation}
and 
\begin{equation}\label{eqconm4}
\begin{aligned}
|\partial_y^{\gamma}R_{2,\alpha}(x,y,\xi)|&\lesssim \langle \sigma\cdot x+\omega \rangle^{q_2-|\alpha|-|\gamma|}\langle \sigma\cdot y+\omega \rangle^{|q_2-|\alpha|-|\gamma||}\langle \xi \rangle^{m_2}\\
&\lesssim \langle \sigma\cdot x+\omega \rangle^{q_2-|\alpha|-|\gamma|}\langle \xi \rangle^{m_2}\big(1+|\eta|^2 +|y|^2 \big)^{\frac{1}{2}|q_2-|\alpha|-|\gamma||},
\end{aligned}
\end{equation}
the above estimates are valid for any multi-index $\gamma$. Thus, \eqref{eqconm3}, \eqref{eqconm4}, and the fact that $y \cdot \eta$ determines a nondegenerate quadratic form allow us to apply the theorem of existence of oscillatory integrals (see for instance \cite[Theorem 2.3]{Raymond1991}, or \cite[Appendix 8]{AlinhacGerard2007}) to justify the validity of the identity \eqref{eqconm2}.

Next, for any $|\beta|<N$ and $|\alpha|=N$, we observe
\begin{equation*}
\begin{aligned}
\Big|\int e^{-2\pi i y \cdot \eta} \eta^{\beta}y^{\alpha} & R_{2,\alpha}(x,y,\xi)  \phi(\epsilon \eta)\, \mathrm{d} y \mathrm{d} \eta \Big|\\
&=\Big|c_{\beta}\epsilon^{|\alpha|-|\beta|-n} \int \widehat{\partial^{\alpha}(\eta^{\beta}\phi)}\Big(\frac{y}{\epsilon} \Big )R_{2,\alpha}(x,y,\xi)\, \mathrm{d} y \Big| \\
& \lesssim \epsilon^{|\alpha|-|\beta|}\|\langle y \rangle^{|q_2|}\widehat{\partial^{\alpha}\big(\eta^{\beta}\phi \big)}(y)\|_{L^1}\langle \sigma\cdot x+\omega \rangle^{q_2-|\alpha|}\langle \xi \rangle^{m_2}.
\end{aligned}
\end{equation*}
Since $|\beta|<|\alpha|$, the above expression seeing as an operator acting on an appropriated space (e.g., acting on the Schwartz class of functions) goes to zero as $\epsilon \to 0$.  Likewise, we derive the same conclusion for the third term on the right-hand side of \eqref{eqconm2}. Thus, the right-hand side of \eqref{eqconm2} is reduced to the first and last term. Consequently, we proceed to estimate the latter term.

We consider two fixed multi-indexes $\alpha$, $\beta$ of order $N$. 
Integrating by parts
\begin{equation*}
\begin{aligned}
 \int e^{-2\pi i y \cdot \eta} \eta^{\beta} y^{\alpha} & R_{1,\beta}(x,\xi,\eta)R_{2,\alpha}(x,y,\xi)\, \mathrm{d} y \mathrm{d} \eta \\
&\sim \sum_{\substack{\gamma \leq \beta }}\int e^{-2\pi i y \cdot \eta}  \partial_{\eta}^{\alpha-\gamma}R_{1,\beta}(x,\xi,\eta)\partial_y^{\beta-\gamma}R_{2,\alpha}(x,y,\xi)\, \mathrm{d} y \mathrm{d} \eta,
\end{aligned}
\end{equation*}
where we have omitted the constant in the previous estimate as it is not relevant to our considerations. Again, we remark that the above identity and the integration by part involved in its deduction are obtained by incorporating the function $\phi(\epsilon \eta)\phi(\epsilon y)$ to the integral, and taking the limit $\epsilon \to 0$, we omit these details.  Consequently, the previous expression, \eqref{eqconm3}, \eqref{eqconm4}, and the theorem of existence of oscillatory integrals applied with amplitude $\partial_{\eta}^{\alpha-\gamma}R_{1,\beta}(x,\xi,\eta)\partial_y^{\beta-\gamma}R_{2,\alpha}(x,y,\xi)$ yield
\begin{equation*}
\begin{aligned}
\Big|\int e^{-2\pi i y \cdot \eta} \eta^{\beta} y^{\alpha}  R_{1,\beta}(x,\xi,\eta)&R_{2,\alpha}(x,y,\xi)\, \mathrm{d} y \mathrm{d} \eta \Big| \\
&\lesssim \langle \sigma \cdot x+\omega\rangle^{q_1+q_2-N}\langle \xi \rangle^{m_1+m_2-N}.
\end{aligned}
\end{equation*}
This shows the required asymptotic expansion for the approximations $\Psi_{a_k}\Psi_{b_k}$. The estimate for $\Psi_{a} \Psi_{b}$ is obtained by taking the limit $k \to \infty$. This argument follows by rather similar considerations in \cite[Theorem 18.1.8]{Hormander2007}, we omit this analysis.
\end{proof}


\section*{Declarations}

{\bf Conflict of interest}. On behalf of all authors, the corresponding author states that there is no conflict of interest.


\bibliographystyle{abbrv}
\bibliography{References}

\end{document}